\documentclass[12pt]{amsart}
\usepackage {amssymb}
\usepackage {amsmath}
\usepackage{amsthm}
\usepackage{graphicx}
\usepackage {amscd}
\usepackage {epic}
\usepackage[alphabetic]{amsrefs}
\usepackage{xypic}
\usepackage[colorlinks, citecolor=blue]{hyperref}
\usepackage[left=1.5in,right=1.5in,top=1.5in,bottom=1.5in]{geometry}
\usepackage{soul}
\usepackage{mathrsfs}
\usepackage{comment}
\usepackage{enumitem}

\makeatletter
\@namedef{subjclassname@2020}{%
  \textup{2020} Mathematics Subject Classification}
\makeatother

\theoremstyle{plain}
\newtheorem{theorem}{Theorem}[section]
\newtheorem{corollary}[theorem]{Corollary}
\newtheorem{lemma}[theorem]{Lemma}

\newtheorem{proposition}[theorem]{Proposition}
\newtheorem{definition-lemma}[theorem]{Definition-Lemma}
\theoremstyle{remark}
\newtheorem{remark}[theorem]{Remark}
\newtheorem{conjecture}[theorem]{Conjecture}

\newcommand{\Pic}[0]{\operatorname{Pic}}

\def\ve{\varepsilon}

\def\NE{\overline{\operatorname{NE}}}
\def\NS{\operatorname{NS}}
\def\NF{\operatorname{\overline{NF}}}
\def\NM{\operatorname{\overline{NM}}}

\def\NA{\overline{\operatorname{NA}}}
\def\Null{\operatorname{Null}}
\newcommand{\<}{\leq}
\def\>{\geq}
\newcommand{\mbQ}{\mathbb{Q}}
\newcommand{\mbR}{\mathbb{R}}

\def\mcO{\mathcal{O}}
\newcommand{\num}{\equiv}

\newcommand{\OO}{{\mathcal{O}}}
\newcommand{\Q}{{\mathbb{Q}}}

\newcommand{\mult}{{\rm mult}}
\newcommand{\Supp}{{\rm Supp}}

\newcommand{\sm}{{\rm sm}}

\newcommand{\mbC}{\mathbb{C}}
\newcommand{\mbZ}{\mathbb{Z}}

\newcommand{\bir}{\dashrightarrow}

\newcommand{\del}{\partial}
\def\lru{\lceil}
\def\rru{\rceil}
\def\lrd{\lfloor}
\def\rrd{\rfloor}

\def\mbN{\mathbb{N}}
\def\mbP{\mathbb{P}}
\def\>{\geq}
\def\ve{\varepsilon}
\def\mcA{\mathcal{A}}
\def\mcO{\mathcal{O}}
\def\mcB{\mathcal{B}}
\def\mcC{\mathcal{C}}
\def\mcD{\mathcal{D}}

\def\mcF{\mathcal{F}}
\def\mcG{\mathcal{G}}

\def\mcK{\mathcal{K}}
\def\mcL{\mathcal{L}}
\def\mcM{\mathcal{M}}

\def\pt{\operatorname{pt}}
\def\Ex{\operatorname{Ex}}
\def\dim{\operatorname{dim}}
\def\codim{\operatorname{codim}}
\def\sing{\operatorname{\textsubscript{sing}}}
\def\sm{\operatorname{\textsubscript{sm}}}
\def\NS{\operatorname{NS}}
\def\NA{\operatorname{\overline{NA}}}
\def\NE{\operatorname{\overline{NE}}}
\def\Nef{\operatorname{Nef}}
\def\Im{\operatorname{Im}}
\def\Chow{\operatorname{Chow}}
\def\EnK{\operatorname{EnK}}

\theoremstyle{definition}
\newtheorem{definition}[theorem]{Definition}
\theoremstyle{definition}
\newtheorem{notation}[theorem]{Notation}
\numberwithin{equation}{section}

\theoremstyle{remark}
\newtheorem{claim}[theorem]{Claim}

\author{Omprokash Das}
\address{School of Mathematics\\
Tata Institute of Fundamental Research\\
Homi Bhabha Road, Navy Nagar\\
Colaba, Mumbai 400005}
\email{omdas@math.tifr.res.in}
\email{omprokash@gmail.com}
\thanks{O. Das was supported by the Start--Up Research Grant(SRG), Grant No. \# SRG/2020/000348 of the Science and Engineering Board Research Board (SERB), Govt. Of India. }

\author{Christopher Hacon}
\address{Department of Mathematics\\
University of Utah\\
155 S 1400 E\\
Salt Lake City, Utah 84112}
\email{hacon@math.utah.edu}
\thanks{C. Hacon was partially supported by NSF research grants no: DMS-1952522, DMS-1801851, DMS-2301374, DMS-58503413 and by a grant
from the Simons Foundation; Award Number: 256202.}

\title{The log minimal model program for K\"ahler $3$-folds}

\begin{document}
\maketitle
\begin{abstract}
	In this article we show that the log minimal model program for $\mbQ$-factorial dlt pairs $(X, B)$ on a compact K\"ahler $3$-fold holds. More specifically, we show that after finitely many divisorial contractions and flips we obtain either a (log) minimal model or a Mori fiber space. We also prove a base point free theorem for K\"ahler $3$-folds.
\end{abstract}

\tableofcontents

\date{\today}
\section{Introduction}
The minimal model program or MMP is one of the most important tools in the birational classification of complex projective varieties.
It was fully established in dimension 3 in the 80's and 90's and recently extended to many cases in arbitrary dimension including the case of varieties of log general type \cite{BCHM10}.

There are many technical difficulties in adapting the minimal model program to compact K\"ahler manifolds. Some of the standard techniques used in the MMP for projective varieties fail for compact K\"ahler manifolds, for example, Mori's Bend and Break technique for producing rational curves, the base point free theorem and the contraction of negative extremal rays fail on K\"ahler manifolds. Campana and Peternell investigated the existence of Mori contractions in \cite{CP97, Pet98, Pet01}, using the deformation theory of rational curves on smooth $3$-folds developed in \cite{Kol91, Kol96}. In \cite{HP16}, Peternell and H\"oring successfully established the minimal model program for compact K\"ahler $3$-folds $X$ with terminal singularities and $K_X$ pseudo-effective. In a subsequent paper \cite{HP15} they also proved the existence of Mori fiber spaces when $X$ has terminal singularities and $K_X$ is not pseudo-effective. 

In \cite{HP16} the authors introduced many new tools which enabled them to use several techniques from the projective MMP. Building on the work of \cite{HP16, HP15, CHP16}, in this article we show that the minimal model program on compact K\"ahler $3$-folds works in much greater generality. More precisely, we show that this program holds for $\mbQ$-factorial dlt pairs $(X, B)$. The main results of this article are the following. 

\begin{theorem}\label{t-main} Let $(X,B)$ be a dlt pair where $X$ is a $\mbQ$-factorial compact K\"ahler $3$-fold. If $K_X+B$ is pseudo-effective, then there exists a finite sequence of flips and divisorial contractions 
\[ \phi: X\dasharrow X_1 \dasharrow \ldots \dasharrow X_n \]
such that $K_{X_n}+\phi _* B$ is nef.
\end{theorem} 
This result is proved in \cite{HP16} when $X$ has terminal singularities and $B=0$.\\

\begin{theorem}\label{t-main2}
Let $(X,B)$ be a dlt pair where $X$ is a $\mbQ$-factorial compact K\"ahler $3$-fold. If $K_X+B$ is not pseudo-effective, then there exists a finite sequence of flips and divisorial contractions 
\[ \phi: X\dasharrow X_1 \dasharrow \ldots \dasharrow X_n \]
and a Mori fiber space $\varphi:X_n\to S$, i.e. a morphism such that $-(K_{X_n}+\phi _* B)$ is $\varphi$-ample and $\rho (X_n/S)=1$.
\end{theorem} 
This result is proved in \cite{HP15} when $X$ has terminal singularities and $B=0$.\\

We note that one of the main difficulties in proving the above theorems  is proving the existence of divisorial contractions. In the pseudo-effective case, the existence of flips and divisorial contractions to a point has already been established by work of \cite{CHP16}, \cite{HP16} and \cite{DO23} (see Theorem \ref{t-pmmp}), and so it remains to prove the existence of divisorial contractions to a curve. This is one of the key results of this article.\\

\begin{definition}\cite[Definition 4.3 and 7.1]{HP16}\cite[Notation 4.1]{CHP16}\label{def:divisorial-extremal-ray}
	Let $X$ be a normal $\mbQ$-factorial compact K\"ahler $3$-fold with rational singularities. 
	\begin{itemize}
		\item We say that a curve $C\subset X$ is very rigid if $\dim_{mC}\Chow(X)=0$ for all $m\>1$.
		\item Let $(X, B)$ be a log canonical pair. A $(K_X+B)$-negative extremal ray $R$ of $\NA(X)$ is called \emph{small} if every curve $C\subset X$ with $[C]\in R$ is very rigid. 
		\item An extremal ray $R$ as above is called \emph{divisorial type} if it is not small.\\
	\end{itemize}
\end{definition}



\begin{notation}\label{not:nef-reduction}
	Let $(X, B)$ be a dlt pair, where $X$ is a $\mbQ$-factorial compact K\"ahler $3$-fold. Let $R$ be a $(K_X+B)$-negative extremal ray of $\NA(X)$ of divisorial type which is defined by a nef class $\alpha$ so that $R=\NA(X)\cap \alpha ^\perp$. Let $S$ be the surface which is covered by and contains all the curves $C\subset X$ such that $[C]\in R$ (cf. \cite[Lemma 7.5]{HP16}). Let $\nu:\tilde{S}\to S$ be the normalization morphism. Consider the nef reduction $f:\tilde{S}\bir T$ of the  nef $(1, 1)$-class $\nu^*(\alpha|_S)$ (see \cite[Theorem 3.19]{HP15}). Note that since $S$ is covered by a family of $\alpha$-trivial curves, the lifts of these curves on $\tilde{S}$ give a family of $\nu^*(\alpha|_S)$-trivial curves. Thus from the definition of nef dimension it follows that $\dim T\<1$, and in particular, $f:\tilde S\to T$ is a morphism (cf. \cite[2.4.4]{BCE02}).  
  We define the notation $n(\alpha)$ as the nef dimension of $\nu^*(\alpha|_S)$,
 i.e.,
	\[
		n(\alpha):=\nu _{\rm nef}(\nu^*(\alpha|_S))=\dim T\in \{0, 1\}
	\] where $\nu _{\rm nef}(\dots )$ denotes the nef dimension.
\end{notation}

\begin{theorem}\label{t-div1} 
	Let $(X,B)$ be a $\mbQ$-factorial dlt pair where $X$ is a compact K\"ahler $3$-fold. Let $R$ be a $(K_X+B)$-negative extremal ray of divisorial type supported by a  nef and big $(1, 1)$-class $\alpha$ such that $\alpha-(K_X+B)$ is K\"ahler and the nef dimension is $n(\alpha)=1$. Then the contraction $c_R:X\to Y$ of $R$ exists.
\end{theorem}

 One of the main difficulties in proving a contraction theorem in the K\"ahler category is the lack of a base-point free theorem analogous to that of \cite[Theorem 3.3]{KM98} in the projective case. Note that, an exact analogue of \cite[Theorem 3.3]{KM98} is impossible on a compact K\"ahler variety which is not projective, since the existence of a big divisor on a compact K\"ahler variety with rational singularities implies that it is projective by \cite[Theorem 1.6]{Nam02} (see Theorem \ref{t-nam}). However, there is a base-point free conjecture in the K\"ahler category involving nef and big cohomology classes which can be thought of as an analogue of \cite[Theorem 3.3]{KM98}. This conjecture is stated in \cite[Conjecture 1.1]{H18} for manifolds.
\begin{conjecture}[Base point freeness]\label{con:bpf}
	Let $(X, B)$ be a klt pair, where $X$ is a normal $\mbQ$-factorial compact K\"ahler variety, and $\alpha\in H^{1, 1}_{\rm BC}(X)$ a nef class on $X$. If $\alpha-(K_X+B)$ is nef and big, then there exists a proper surjective morphism with connected fibers $f:X\to Y$ to a normal compact K\"ahler variety $Y$ with rational singularities and a K\"ahler class $\alpha_Y\in H^{1, 1}_{\rm BC}(Y)$ such that $\alpha=f^*\alpha_Y$. 
\end{conjecture}
As an application of Theorems \ref{t-main} and \ref{t-main2} and other related results we prove this conjecture when $X$ has dimension 3. 
\begin{theorem}\label{thm:bpf}
	Let $(X, B)$ be a log pair, where $X$ is a normal $\mbQ$-factorial compact K\"ahler $3$-fold, and $\alpha\in H^{1,1}_{\rm BC}(X)$ a nef class. Assume that one of the following conditions is satisfied:
	\begin{enumerate}[label=(\roman*)]
		\item $(X, B)$ is klt and $\alpha-(K_X+B)$ is nef and big, or
		\item $(X, B)$ is dlt and $\alpha-(K_X+B)$ is a K\"ahler class.
 	\end{enumerate} 
Then there exists a proper surjective morphism with connected fibers $\psi:X\to Z$ to a normal 
K\"ahler variety $Z$ with rational singularities and a K\"ahler class $\alpha_Z\in H^{1, 1}_{\rm BC}(Z)$ on $Z$ such that $\alpha=\psi^*\alpha_Z$. In particular, in case (ii), $\psi$ is a projective morphism.
\end{theorem} 
 
  When $X$ has terminal singularities, $B=0$ and $\alpha-K_X$ is a K\"ahler class, this theorem is proved in \cite[Theorem 1.3]{H18} (when $\alpha$ is both nef and big) and in \cite[Theorem 2.7]{TZ18} (when $\alpha$ is nef but not big). In fact our proof is a direct generalization of the techniques in \cite{H18} and \cite{TZ18} using more general MMP results such as Theorems \ref{t-main} and \ref{t-main2}.\\

This article is organized in the following manner. In Section 2 we collect some important technical results which are used throughout the article. In Section 3 we prove Theorems \ref{t-div1} and \ref{t-main} under the additional hypothesis that $X$ has \emph{strongly $\mbQ$-factorial} singularities. Section 4 is dedicated to developing the cone and contraction theorems for non-pseudo-effective pairs; the contraction theorem is again proved under the additional hypothesis that $X$ has strongly $\mbQ$-factorial singularities. These results are then used in Section 5 to prove some important technical results related to the existence of Mori fiber spaces. The main result of this section is Theorem \ref{thm:fano-contraction}, which is a special case of Theorem \ref{thm:bpf}, namely the case when $K_X+B$ is not pseudo-effective. In Section 6 we prove all the main theorems in full generality, namely Theorems \ref{t-main}, \ref{t-main2}, \ref{t-div1} and \ref{thm:bpf}.\\


\noindent
\textbf{Acknowledgment.} The authors would like to thank S\'ebastien Boucksom, Ved Datar, Andreas H\"oring, Sabyasachi Mukherjee, Mihai P\u{a}un, Valentino Tosatti, and Mingchen Xia for answering their questions and the anonymous referees for their many corrections and suggestions to improve this paper.  


\section{Preliminaries} 
An \emph{analytic variety} or simply a \emph{variety} is an irreducible and reduced complex space. 
A \emph{pair} $(X,B)$ consists of a normal analytic variety $X$ and an effective $\mbQ$-divisor $B\geq 0$ such that $K_X+B$ is $\mbQ$-Cartier. We define the singularities of the pair $(X, B)$ as in \cite{KM98}.  If $B$ is not assumed to be effective, then we will call $(X,B)$ a sub-pair and the corresponding singularities of $(X, B)$  sub-klt, sub-dlt, etc.

\begin{definition}\cite[Definition 2.2]{HP16}\label{def:kahler-space}
	 An analytic variety $X$ is  \emph{K\"ahler} or a \emph{K\"ahler space} if there exists a positive closed real $(1, 1)$-form $\omega\in\mcA_{\mbR}^{1, 1}(X)$ such that the following holds: for every point $x\in X$ there exists an open neighborhood $x\in U$ and a closed embedding $\iota_U:U\to V$ into an open set $V\subset\mbC^N$, and a strictly plurisubharmonic $\mathcal C^\infty$-function $f:V\to\mbR$ such that $\omega|_{U\cap X_{\sm}}=(i\del\bar{\del}f)|_{U\cap X_{\sm}}$. Here $X_\textsubscript{sm}$ is the smooth locus of $X$.
	
\end{definition}

In the following we collect some important definitions. For a more detailed discussion, we encourage the reader to consult \cite{HP16, HP15, CHP16} and the references therein.
\begin{definition}\label{def:lots-of-definitions}
	
\begin{enumerate}[label=(\roman*)]
	\item\label{item:c-manifolds} A compact analytic variety $X$ is said to belong to \emph{Fujiki's class} $\mcC$ if one of the following equivalent conditions are satisfied:
	\begin{enumerate}
		\item $X$ is a meromorphic image of a compact K\"ahler variety $Y$, i.e., there exists a dominant meromorphic map $f:Y\bir X$ from a compact K\"ahler variety $Y$  (see \cite[4.3, page 34]{Fuj78}).
		\item $X$ is a holomorphic image of a compact K\"ahler manifold, i.e., there is a surjective morphism $f:Y\to X$ from a compact K\"ahler manifold $Y$ (see \cite[Lemma 4.6]{Fuj78}).
		\item $X$ is bimeromorphic to a compact K\"ahler manifold (see \cite[Theorem 3.2, page 51]{Var89}).\\
	\end{enumerate}
	
	\item On a normal compact analytic variety $X$ we replace the use of N\'eron-Severi group $\NS(X)_{\mbR}$ by $H^{1, 1}_{\rm BC}(X)$, the Bott-Chern cohomology of real closed $(1, 1)$-forms with local potentials or equivalently, the closed bi-degree $(1, 1)$-currents with local potentials. See \cite[Definition 3.1 and 3.6]{HP16} for more details. More specifically, we define
\[ N^1(X):=H^{1,1}_{\rm BC}(X).\]

	\item If $X$ is in Fujiki's class $\mcC$ and has \emph{rational singularities}, then from \cite[Eqn. (3)]{HP16} we know that $N^1(X)=H^{1, 1}_{\rm BC}(X)\subset H^2(X,\mathbb R)$. In particular, the intersection product can be defined in $N^1(X)$ via the cup product of $H^2(X, \mbR)$.
	
	\item Let $X$ be a normal compact analytic variety contained in Fujiki's class $\mcC$. We define $N_1(X)$ to be the vector space of real closed currents of bi-dimension $(1, 1)$ modulo the following equivalence relation: $T_1\num T_2$ if and only if 
	\[
		T_1(\eta)=T_2(\eta)
	\]
for all real closed $(1, 1)$-forms $\eta$ with local potentials.

	\item We define $\NA(X)\subset N_1(X)$ to be the closed cone generated by the classes of positive closed currents $\Theta \geq 0$ (see \cite[\S 1.C]{Dem12}). The Mori cone $\NE(X)\subset\NA(X)$ is defined as the closure of  the cone of currents of integration $T_C$, where $C\subset X$ is an irreducible curve. 
	
	\item Let $X$ be a normal compact analytic variety and $u\in H^{1,1}_{\rm BC}(X)$. Then $u$ is called \emph{pseudo-effective} if it can be represented by a bi-degree $(1, 1)$-current $T\in \mathcal D ^{1,1}(X)$ which is locally of the form $\partial \bar \partial f$ for some psh function $f$. $u$ is called \emph{nef} if it can be represented by a form $\alpha$ with local potentials such that for some positive $(1,1)$-form $\omega $ on $X$ and for every $\epsilon >0$, there exists a $\mathcal C ^\infty$-function $f_\epsilon\in\mcA^0(X)$ such that $\alpha + i\partial \bar \partial f_\epsilon \geq -\epsilon \omega$. See \cite{Dem85} for more details. 
	
	\item The {\it nef cone} ${\rm Nef}(X)\subset N^1(X)$ is the cone generated by nef cohomology classes. Let $\mcK$ be the open cone in $N^1(X)$ generated by the classes of K\"ahler forms. Note that the nef cone $\Nef(X)$ is the closure of $\mcK$, i.e. $\Nef(X)=\overline{\mcK}$. 
	
\item We say that a variety $X$  is $\mbQ${\it -factorial} if for every Weil divisor $D\subset X$, there is a positive integer $k>0$ such that $kD$ is a Cartier divisor, and for the canonical sheaf $\omega_X$, there is a positive integer $m>0$ such that $(\omega _X^{\otimes m})^{**}$ is a line bundle. It is well known that if $X$ is a $\mbQ$-factorial 3-fold and $X\dasharrow X'$ is a flip or a divisorial contraction, then $X'$ is also $\mbQ$-factorial.

\item\label{d-sqf} If $X$ is a normal variety then we say that a coherent sheaf $\mathcal L$ is {\it divisorial} if it is reflexive of rank 1.
If $U\subset X$ is Stein, then $\mathcal L |_U\cong \OO _U(D)$ for some Weil divisor $D$ on $U$.
We say that a divisorial sheaf $\mathcal L$ is $\Q$-{\it Cartier} (or a $\Q$-{\it line bundle}) if $(\mathcal L^{\otimes m})^{**}$ is a line bundle for some $m\in\mbN$.
We will say that a variety $X$  is {\it strongly} $\mbQ${\it -factorial} if every divisorial sheaf $\mathcal L$ is a $\mbQ$-line bundle. Note that a complex manifold is an example of a strongly $\mbQ$-factorial variety (see Lemma \ref{lem:manifolds-are-strongly-Q-factorial}).

\end{enumerate}

\end{definition}

\begin{lemma}\label{lem:manifolds-are-strongly-Q-factorial}
  If $X$ is a complex manifold, then every reflexive rank $1$ sheaf on $X$ is a line bundle. In particular, $X$ is strongly $\mbQ$-factorial.   
\end{lemma}

\begin{proof}
    Let $\mcL$ be a reflexive rank $1$ sheaf on $X$. Since $X$ is a manifold, from \cite[Corollary V.5.20, page 160]{Kob87} it follows that there is an analytic subset $Z\subset X$ such that $\mcL|_{X\setminus Z}$ is a line bundle and $\codim_X(Z)\>3$. 
    Then from \cite[Theorem 4]{Har74} it follows that $\Pic(X)\cong \Pic(X\setminus Z)$. Thus $\mcL|_{X\setminus Z}$ extends to an unique line bundle on $X$, say $\mcM$. But since $\mcL$ is a reflexive sheaf, $\mcL|_{X\setminus Z}\cong\mcM|_{X\setminus Z}$ and $\codim_X(Z)\>2$, it follows that $\mcL\cong \mcM$ is a line bundle. 
\end{proof}

\begin{lemma}\label{c-small}
    Let $f :Y\to X$ be a small bimeromorphic projective morphism of analytic varieties such that  $X$ is strongly $\Q$-factorial. Then $f$ is an isomorphism.
\end{lemma}
\begin{proof}
Let $\mathcal L$ be a relatively ample line bundle. Since $X$ is strongly $\Q$-factorial, $(f_*\mathcal L)^{**}$ is $\Q$-Cartier and so $\mathcal M:=((f_*\mathcal L)^{\otimes m})^{**}$ is a line bundle for some integer $m>0$.
Since $f$ is small, there is an open subset $U\subset Y$ such that $\codim_Y(Y\setminus U)\>2$ and $f|_U$ is an isomorphism. Then $(f^*\mathcal M)|_U\cong \mathcal L^{\otimes m}|_U$, and so $f^* \mathcal M\cong \mathcal L^{\otimes m}$ is a relatively ample line bundle. Thus $f$ is an isomorphism.
\end{proof}

\begin{lemma}\label{l-sqfmmp}
    Let $X$ be a strongly $\mbQ$-factorial compact analytic variety, $(X,B)$ a dlt pair, and $f:X\dasharrow X'$ a $(K_X+B)$-flip or divisorial contraction. Then $X'$ is also strongly $\mbQ$-factorial.
\end{lemma}
\begin{proof}
    This follows easily from the base point free theorem \cite[Theorem 4.8]{Nak87}. We may assume that $(X,B)$ is klt. If $f:X\to X'$ is a $(K_X+B)$-negative divisorial contraction with exceptional divisor $E$, then for any divisorial sheaf $\mathcal L$ on $X'$, let $\mathcal M=(f^*\mathcal L)^{**}$. Then $\mcM$ is a divisorial sheaf on $X$. Recall that by assumption $X$ is strongly $\Q$-factorial, $\rho(X/X')=1$ and $E$ generates $N^1(X/X')$. Thus $\mathcal M$ is $\mbQ$-Cartier and  there are integers $m>0$ and $n\in\mbZ$  such that 
$(\mathcal M^{\otimes m})^{**}\otimes \OO _X(nE)$ is a $f$-numerically trivial line bundle. Working locally over $X'$, it follows from the base point free theorem \cite[Theorem 4.8]{Nak87} that $(\mathcal M^{\otimes m})^{**}\otimes \OO _X(nE)\cong f^*\mathcal M'$ for some line bundle $\mathcal M'$ on $X'$.
Let $U:=X'\setminus (f(E)\cup X'_{\sing})$ and $j:U\to X'$ be the inclusion, then $U\subset X'$ is a big open subset, i.e. the complement of a finite union of analytic subvarieties of codimension at least 2. Then 
$(\mathcal L^{\otimes m})^{**}\cong j_* ( (\mathcal L|_U)^{\otimes m})\cong j_* (\mathcal M'|_U)\cong \mathcal M'$ and so $\mathcal L$ is $\Q$-{\it Cartier} as required. 

A similar argument shows that $X'$ is strongly $\mbQ$-factorial when $f:X\bir X'$ is a flip. 
\end{proof}
\begin{remark}\label{rmk:nefness}
	Note that if $D$ is a $\mbQ$-Cartier divisor on a variety $X$, then in algebraic geometry we say that $D$ is nef if $D\cdot C\>0$ for every irreducible curve $C\subset X$. However the cohomology class of the current of integration of $D$ is not necessarily nef in the sense defined in Definition \ref{def:lots-of-definitions} (vi). To avoid this sort of confusion, temporarily we will call a $\mbQ$-Cartier divisor $D$ \emph{algebraically nef} if $D\cdot C\>0$ for every irreducible curve $C\subset X$, and $D$ is \emph{analytically nef}  if the cohomology class of the current of integration associated to $D$ is nef in the sense of Definition \ref{def:lots-of-definitions} (vi).  We remark that if $D=\sum d_iD_i, d_i\in\mbQ$, is a linear combination of effective Cartier divisors $D_i$, then the corresponding current of integration is given by $\int _D\ldots =\sum d_i\int _{D_i}\ldots$
 
 Note that if $X$ is a compact K\"ahler variety, then analytically nef implies algebraically nef but the converse is not true in general, see \cite[page 385]{HP17} for a counterexample. However, the following two lemmas show that these two versions of nefness are equivalent in important cases: (i) $X$ is a Moishezon space (see Lemma \ref{lem:nef-moishezon}), and (ii) $X$ is a K\"ahler $3$-fold and $D$ is an adjoint divisor (see Lemma \ref{lem:adjoint-nef}).
\end{remark}

\begin{lemma}\label{lem:nef-moishezon}\cite[Corollary 1, page 418]{Pau98}
	Let $X$ be a normal compact Moishezon variety, i.e. $X$ is bimeromorphic to a projective variety. Let $D$ be a $\mbQ$-Cartier divisor on $X$. Then $D$ is analytically nef if and only if it is algebraically nef.
\end{lemma}

\begin{remark}\label{rmk:R-nef-moishezon}
	Note that one can prove the above lemma more generally for $\mbR$-Cartier divisors by passing to a resolution of singularities $\pi:\widetilde{X}\to X$ such that $\widetilde{X}$ is a projective manifold and then use \cite[Corollary 0.2]{DP04}, \cite[Theorem 1, page 412]{Pau98} and \cite[Lemma 2.38]{DHP22}.
\end{remark}
The proof of the next lemma relies on the MRC fibration introduced in \cite{KMM92} and \cite{Cam92}; see also \cite[Remark 6.10]{CH20} for remarks on the analytic version of the MRC fibration. 

\begin{lemma}\label{lem:adjoint-nef}
	Let $(X, B)$ be a dlt pair, where $X$ is a $\mbQ$-factorial compact K\"ahler $3$-fold. Assume that one of the following conditions is satisfied:
	\begin{enumerate}[label=(\roman*)]
		\item $K_X+B$ is pseudo-effective, or
		\item $K_X+B$ is not pseudo-effective and there exists a K\"ahler form $\omega$ on $X$ such that $K_X+B+\omega$ is pseudo-effective.
	\end{enumerate}
	Then $K_X+B$ (resp. $K_X+B+\omega$) is analytically nef if and only if it is algebraically nef. 
\end{lemma}
\begin{proof}
	If $K_X+B$ is pseudo-effective, then the result follows from a similar proof as in \cite[Corollary 4.2]{HP16} and \cite[Corollary 4.1]{CHP16}. For a complete proof in this case, see \cite[Corollary 2.19]{DO23}. In the second case, first replacing $B$ by $(1-\ve)B$ and $\omega$ by $\omega+\ve B$ we may assume that $(X, B)$ is klt. Now if the base of the MRC fibration of $X$ has dimension $\<1$, then from Lemma \ref{l-proj} and its proof it follows that $X$ is projective and additionally $\NS(X)_\mbR=H^{1, 1}_{\rm BC}(X)$. Then the result follows from Lemma \ref{lem:nef-moishezon}. If the base of the MRC fibration of $X$ has dimension $2$, then the result follows by a similar proof as in \cite[Corollary 3.5]{HP15}. Note that this proof uses \cite[Lemma 3.4]{HP15}, which is replaced here by Claim \ref{c-uni}.\\
\end{proof}

\begin{remark}
	If $f:X\to Y$ is a proper surjective morphism between normal analytic varieties and $\dim Y>0$, then we say that a $\mbQ$-Cartier $\mbQ$-divisor $D$ on $X$ is \emph{nef over $Y$}  or \emph{$f$-nef} if $D\cdot C\>0$ for every irreducible (compact) curve $C\subset X$ such that $f(C)=\pt$. Note that to be more precise we should call $D$, \emph{algebraically nef over $Y$} or \emph{algebraically $f$-nef}, however, by Lemma \ref{lem:nef-moishezon} it is equivalent to \emph{analytically nef} (over $Y$) if $f$ is a Moishezon morphism (see \cite[Definition VIII.3.5 (2), page 334]{GPR94}). Since all the morphisms considered in this article are Moishezon, e.g. projective morphisms, proper bimeromorphic morphisms, etc., this will not create any confusion.
 \end{remark}

Let $X$ be a normal compact variety and $\omega$ a real closed $(1, 1)$-form on $X$ with local potentials. Then we can define
\[ \lambda _\omega \in N_1(X)^*,\qquad{\rm via}\qquad \lambda _\omega ([T])=T(\omega).\] 
This gives a well defined canonical map
\[ \Phi:N^1(X)\to N_1(X)^*,\qquad [\omega]\mapsto \lambda_\omega .\]
If in addition $X$ belongs to Fujiki's class $\mcC$ and has rational singularities, then $\Phi$ is an isomorphism by \cite[Proposition 3.9]{HP16}. Moreover, if $\dim X=3$, then $\Phi ({\rm Nef}(X))=\overline {\rm NA}(X)^*$ by \cite[Proposition 3.15]{HP16}.

We recall the following useful result from \cite{HP16} for future reference.
\begin{lemma}\label{l-HP}\cite[Lemma 3.3]{HP16}
	 Let $f:X\to Y$ be a proper bimeromorphic morphism between normal compact complex spaces in Fujiki's class $\mathcal C$ with at most rational singularities. Then we have an injection \[f^*:H^{1,1}_{\rm BC}(Y)=H^1(Y,\mathcal H _Y)\hookrightarrow H^1(X,\mathcal H _X)=H^{1,1}_{\rm BC}(X)\]
whose image is given by \[\Im(f^*)=\{ \alpha \in H^1(X,\mathcal H _X)\ |\ \alpha \cdot C=0\ \forall\; C\subset X\ {\rm curve\ s.t.\ }f_*C=0\}.\]
Furthermore, let $\alpha \in H^1(X,\mathcal H _X)\subset H^2(X,\mathbb R )$ be a class such that $\alpha =f^*\beta$ with $\beta \in H^2(Y,\mathbb R)$. Then there exists a smooth real closed $(1,1)$-form with local potentials $\omega _Y$ on $Y$ such that $\alpha =f^*[\omega _Y]$. 
\end{lemma}

 In general the push-forward of a cohomology class $[T]$ of a bi-degree $(1, 1)$-current $T$ with local potentials may not be a cohomology class in $H^{1, 1}_{\rm BC}(Y)$, since the current $f_*T$ may not have local potentials on $Y$. See \cite[Lemma 3.4]{HP16} for a sufficient condition when this does hold.\\
 
 

\begin{remark}\label{rmk:kahler-property} 
	Note that if $X$ is a compact K\"ahler space and $\pi:X'\to X$ a projective morphism, then $X'$ is again K\"ahler (see \cite[Prop. 1.3.1.(vi), page 24]{Var89}). In particular, $X$ has a K\"ahler desingularisation. A subvariety of a K\"ahler space is also K\"ahler, see \cite[Prop. 1.3.1.(i), page 24]{Var89}. 
	
\end{remark}

\subsection{Projectivity criteria}
Recall the following generalization of \cite{Moi66} due to Namikawa.
\begin{theorem}\label{t-nam}\cite[Theorem 1.6]{Nam02} Let $X$ be a compact Moishezon variety with 1-rational singularities. If $X$ is K\"ahler, then $X$ is projective.
    \end{theorem}
    \begin{remark}
        Recall that $X$ has 1-rational singularities if it admits a resolution $\nu :X'\to X$ such that $R^1\nu _*\OO _{X'}=0$.
        In particular, by Lemma \ref{lem:dlt-rational}, any klt variety has rational and hence 1-rational singularities. 
    \end{remark}
    
    Next we recall the following of Kodaira's projectivity criterion from \cite{Kod54}.
    \begin{theorem}\label{t-kod}
        Let $X$ be a compact K\"ahler variety with rational singularities such that $H^2(X, \OO _X)=0$, then $X$ is projective.
    \end{theorem}
    \begin{proof}
        Let $\nu :X'\to X$ be a resolution of singularities, then $X'$ is a K\"ahler manifold and $R^i\nu_*\OO _{X'}=0$ for all $i>0$. Thus $H^2(X', \OO _{X'})\cong H^2(X, \OO _{X})=0$.
        By \cite{Kod54}, $X'$ is projective and hence $X$ is Moishezon. By Theorem \ref{t-nam}, $X$ is projective.
    \end{proof}
\subsection{Resolution of singularities and Kawamata-Viehweg vanishing theorem} The existence of resolutions of singularities for analytic varieties and embedded resolutions are proved in \cite{AHV77} and \cite{BM97}. Unlike the case of algebraic varieties (of finite type over a field), for analytic varieties the resolution of singularities is not obtained via global blow ups of smooth centers, unless the variety is (relatively) compact. The following version of log resolution will be useful for us.\\

Let $X$ be a normal analytic variety and $D=\sum_{i=1}^nD_i$ a reduced Weil divisor on $X$. Then there exists a unique largest Zariski open subset $U$ of $X$ contained in its smooth locus such that $D|_U$ is a simple normal crossing divisor and $\codim_X (X\setminus U)\>2$. The open subset $U$ is called the \emph{simple normal crossing locus} of the pair $(X, D)$ and we denote it by ${SNC}(X, D)$. Also, recall that a locally compact topological space $X$ is called \emph{countable at infinity} or \emph{$\sigma$-compact} if it can be written as a countable union of compact subsets. Clearly, any compact space $X$ is $\sigma$-compact. Moreover, locally compact and second countable Hausdorff spaces are $\sigma$-compact. 

\begin{theorem}[Log Resolution]\cite[Theorems 13.2, 1.10 and 1.6]{BM97}\label{thm:log-resolution}
	Let $X\subset W$ be a relatively compact open subset of an analytic variety $W$ and $D\subset X$ a pure codimension $1$ reduced analytic subset of $X$. Then there exists a projective bimeromorphic morphism $f:Y\to X$ from a smooth variety $Y$ satisfying the following properties:
	\begin{enumerate}
		\item $f$ is a successive blow up of smooth centers contained in $X\setminus SNC(X, D)$,
		\item $f^{-1}(SNC(X, D))\cong SNC(X, D)$, and
		\item $\Ex(f)$ is a pure codimension $1$ subset of $Y$ such that $\Ex(f)\cup(f^{-1}_*D)$ has SNC support.
	\end{enumerate}
	
\end{theorem}

\begin{proof}
Since $W$ is a locally compact Hausdorff space and $X\subset W$ is a relatively compact open subset, there exists another relatively compact open subset $X_1\subset W$ containing $X$ such that $\overline{X}\subset X_1$. Since every point of an analytic variety has a second countable open neighborhood, it follows that $X_1$ is $\sigma$-compact. Thus by \cite[Theorem 13.3]{BM97}, there is a projective bimeromorphic morphism $f_1:Y_1\to X_1$ from a smooth variety $Y_1$ obtained by finitely many blow ups of smooth centers contained in the singular locus of $X_1$. Note that $X_2:=f_1^{-1}(X)\subset Y_1$ is relatively compact. Now consider the non-SNC locus $Z:=Y_1\setminus SNC(Y_1, (f^{-1}_{1,*}D\cup \Ex(f_1)))$; this is a closed analytic subset of $Y_1$. Note that $Z\cap X_2$ is a closed analytic subset of the manifold $X_2$ and also an open subset of $Z$. It then follows that $Z\cap X_2$ is a relatively compact open subset of $Z$. Therefore by \cite[Theorem 13.2]{BM97} applied to $Z\cap X_2 \subset X_2$ we obtain a projective bimeromorphic morphism $g:Y\to X_2$ (which is a composite of blow ups of smooth centers) from a smooth variety $Y$ such that $(Y, (f_1|_{X_2}\circ g)^{-1}_*D+\Ex(f_1|_{X_2}\circ g))$ is a log smooth pair, $f_1|_{X_2}\circ g$ is an isomorphism over the SNC locus of $(X, D)$ and $\Ex(f_1|_{X_2}\circ g)$ is a pure codimension $1$ subset of $Y$. Then we conclude by setting $f:=f_1|_{X_2}\circ g$.
	
	
\end{proof}~\\

Next we will state Chow's lemma for analytic varieties due to Hironaka \cite{Hir75}. Note that unlike Chow's lemma for algebraic varieties (of finite type over a field), in the analytic category it does not hold for arbitrary proper morphism between analytic varieties; it only holds for proper \emph{bimeromorphic} morphisms. 
\begin{theorem}[Chow's Lemma]\cite[Corollary 2]{Hir75}\label{thm:chow}
	Let $f:X\to Y$ be a proper bimeromorphic morphism between two complex spaces such that $Y$ is reduced and $\sigma$-compact. Then there exists a projective bimeromorphic morphism $\nu :X'\to X$ from a complex space $X'$ such that the composition $f'=f\circ \nu: X'\to Y$ is projective.
\end{theorem}

As an application of Chow's lemma we prove the following useful result.

\begin{lemma}[Reducing Proper Morphism to Projective Morphisms]\label{lem:proper-to-projective}	Let $f:X\to S$ be a proper surjective morphism of analytic varieties, and let $L$ be a $f$-big line bundle on $X$ and $D$ a $\mbQ$-divisor. Then over any relatively compact open subset $V\subset S$, there exists a proper bimeromorphic morphism $\alpha:W\to f^{-1}V$ from a smooth analytic variety $W$ such that $\beta=f|_{f^{-1}V}\circ \alpha: W\to V$ is a projective morphism and $(W, \alpha^{-1}_*(D|_{f^{-1}V})+\Ex(\alpha))$ is a log smooth pair.
\end{lemma}

\begin{proof}
	Let $\phi:X\bir Y$ be the relative Iitaka fibration of $L$ over $S$ and $g:Y\to S$ the induced projective morphism. Since $L$ is $f$-big, $\phi:X\bir  Y$ is bimeromorphic.
	Let $p:\Gamma\to X$ and $q:\Gamma\to Y$ be the resolution of indeterminacy of $\phi$ so that $p$ is proper (see \cite[Theorem VII.1.9]{GPR94}). 
\[
\xymatrixcolsep{3pc}\xymatrix{
& \Gamma\ar[dr]^q\ar[dl]_p & \\
X\ar@{-->}[rr]^\phi\ar[dr]_f && Y\ar[dl]^g \\
& S &
}	
\]	
	Now fix a relatively compact open subset $V\subset S$. Choose another relatively compact open set $U\subset S$ containing $V$ such that $\overline{V}\subset U$. Note that $U$ is $\sigma$-compact, since it is relatively compact. Since $f$ and $g$ are both proper morphisms, it follows that $X_U:=f^{-1}U$ and $Y_U:=g^{-1}U$ are both $\sigma$-compact. Let $\Gamma_U:=q^{-1}(g^{-1}U)=p^{-1}(f^{-1}U)$. Then from the commutative diagram above it follows that $q|_{\Gamma_U}:\Gamma_U\to g^{-1}U$ is a proper morphism. In particular, $\Gamma_U$ is $\sigma$-compact. Note that $q|_{\Gamma_U}$ is bimeromorphic. Therefore by Theorem \ref{thm:chow} there is a projective bimeromorphic morphism $h:Z\to \Gamma_U$ from an analytic variety $Z$ such that $q|_{\Gamma_U}\circ h:Z\to Y_U$ is a projective bimeromorphic morphism. Since $g$ is projective, so is $Z\to U$.
 
	 Now we replace $U$ by our previously fixed open set $V$. Then $Z_V:=(g\circ q\circ h)^{-1}V$ is a relatively compact open subset of $Z$. Let $r:W\to Z_V$ be the log resolution of $(Z_V, (p\circ h)^{-1}_*(D|_{f^{-1}V}))$ as in Theorem \ref{thm:log-resolution}. Let $\alpha:=p|_{\Gamma_V}\circ h|_{h^{-1}\Gamma_V}\circ r$ and $\beta:=g|_{g^{-1}V}\circ q|_{\Gamma_V}\circ h|_{h^{-1}\Gamma_V}\circ r$, where $\Gamma_V:=p^{-1}(f^{-1}V)=q^{-1}(g^{-1}V)$. Note that $\beta$ is a projective morphism, since it is a componsition of projective morphisms  over relatively compact bases. Then $\alpha:W\to f^{-1}V$ is a proper bimeromorpic morphism and $\beta:W\to V$ is a projective morphism such that $\beta=f|_{f^{-1}V}\circ \alpha$ and $(W, \alpha^{-1}_*(D|_{f^{-1}V})+\Ex(\alpha))$ is a log smooth pair.	 
\end{proof}~\\

\begin{definition}\label{def:f-nef-big}
	Let $f:X\to Y$ be a proper surjective morphism of analytic varieties and $L$ a line bundle on $X$. Then $L$ is called \emph{$f$-nef-big}, if $c_1(L)\cdot C\>0$ for every irreducible curve  $C\subset X$ such that $f(C)=\pt$, and $\kappa(X/Y, L)=\dim X-\dim Y$ (see \cite[(B), page 554]{Nak87}). A $\Q$-Cartier divisor $D$ on $X$ is called $f$-nef-big if and only if so is $\mcO _X(mD)$ for some $m>0$ sufficiently divisible.
\end{definition}

The following is a version of the (relative) Kawamata-Viehweg vanishing theorem for proper morphisms between analytic varieties.
\begin{theorem}\cite[Theorem 3.7]{Nak87}\cite[Corollary 1.4]{Fuj13}\label{thm:kvv-original}
	Let $\pi:X\to S$ be a proper surjective morphism from a complex manifold $X$ onto an analytic variety $S$. Let $H$ be a $\mbQ$-Cartier $\mbQ$-divisor on $X$ such that it is $\pi$-nef-big and $\{H\}$ has SNC support. Then $R^i\pi_*(\omega_X\otimes\mcO_X(\lru H\rru))=0$ for all $i>0$.  
\end{theorem}

We prove the following variant which is more convenient for us.
\begin{theorem}\label{thm:relative-kvv}
Let $\pi:X\to S$ be a proper surjective morphism of analytic varieties. Let $\Delta\>0$ be a $\mbQ$-divisor on $X$ such that $(X, \Delta)$ is klt, and $D$ is a $\mbQ$-Cartier integral Weil divisor on $X$ such that $D-(K_X+\Delta)$ is $\pi$-nef-big. Then 
        \[
                R^i\pi_*\mcO_X(D)=0\qquad \text{ for all }\ i>0.
        \]
\end{theorem}

\begin{proof}
First note that the question is local on the base. Then by Lemma \ref{lem:proper-to-projective}, over a relatively compact Stein open subset $U\subset S$, there exists a proper bimeromorphic morphism $f:Y\to \pi^{-1}U$ from a smooth variety $Y$ such that $\pi|_{\pi^{-1}U}\circ f:Y\to U$ is a projective morphism and $(Y, f^{-1}_*(\Supp(D)+\Delta)|_{\pi^{-1}U}+\Ex(f))$ is a log smooth pair. Replacing $S, X$ and $\pi$ by $U, \pi^{-1}U$ and $\pi|_{\pi^{-1}U}$, respectively, we may assume that $S$ is a Stein space and $f:Y\to X$ is log resolution of $(X, \Delta+\Supp(D))$ such that $\pi\circ f$ is projective.

Next observe that, since $S$ is (relatively compact and) Stein and $\pi\circ f$ is projective, every line bundle on $Y$ corresponds to a (non-unique) Cartier divisor. Indeed, if $\mathscr M$ is a line bundle on $Y$ and $H$ is a $(\pi\circ f)$-ample Cartier divisor on $Y$, then $\mathscr M\otimes\mcO_Y(mH)$ is relatively globally generated over $S$ for all $m\gg 0$. In particular, $f_*(\mathscr M\otimes\mcO_Y(mH))\neq 0$ for all $m\gg 0$. Since $S$ is Stein, this implies that $H^0(\mathscr M\otimes\mcO_Y(mH))\neq 0$ for all $m\gg 0$. Let $\Theta$ be an effective Cartier divisor defined by a non-zero element of $H^0(\mathscr M\otimes\mcO_Y(mH))$. Then $\mathscr M\cong\mcO_Y(\Theta-mH)$. In particular, the canonical line bundle $\omega_Y$ is given by a Cartier divisor, which we will denote by the usual notation $K_Y$.\\ 

Now write 
\[ K_Y+\Gamma=f^*(K_X+\Delta )+E\]
such that $\Gamma\>0$ and $E\>0$ do not share any common component and $f_*\Gamma=\Delta $ and $f_*E=0$.

Let $A$ be a $\mbQ$-Cartier $(\pi\circ f)$-nef-big divisor on $Y$ such that $f^*D=f^*(K_X+\Delta )+A=K_Y+\Gamma-E+A$. 
Since $\{A\}$ has SNC support, by Theorem \ref{thm:kvv-original} we have 
\begin{equation}\label{eqn:kvv}
R^if_*\mcO_Y(K_Y+\lru A\rru)=0 \text{ and } R^i(\pi\circ f)_*\mcO_X(K_Y+\lru A\rru)=0 \text{ for all } i>0.	
\end{equation}
 Now since $(X, \Delta )$ is a klt pair, the coefficients of $\Gamma$ are contained in the interval $(0, 1)$, and thus 
 \[K_Y+\lru A\rru=\lru f^*D +E-\Gamma \rru\geq \lrd  f^*D  \rrd\]
 so that $f_*\mcO_Y(K_Y+\lru A\rru)= \mcO_X(D)$.
 Combining \eqref{eqn:kvv} with a standard Leray spectral sequence argument, it follows that $R^i\pi_*\mcO_X(D)=0$ for all $i>0$.

%
%

\end{proof}~\\

\subsection{MMP for K\"ahler $3$-folds}
The following results are improvements of the important results from \cite{CHP16} and \cite{HP16}.

\begin{theorem}\cite[Theorem 2.26]{DO23}\label{t-cone} Let $(X,B)$ be a dlt pair, where $X$ is a $\mbQ$-factorial compact K\"ahler  $3$-fold.  If $K_X+B$ is pseudo-effective, then there is a rational number $d>0$ and an at 
most countable set of curves $\{\Gamma _i\}_{i\in I}$ such that
\[ 0<-(K_X+B)\cdot \Gamma _i \leq d\]
and that
\[ \overline{\rm NA}(X)=\overline{\rm NA}(X)_{(K_X+B)\geq 0}+\sum _{i\in I}\mathbb R ^+[\Gamma _i].\]
\end{theorem}
Note that if $\omega$ is modified K\"ahler, then there are only finitely many $i\in I$ such that 
$(K_X+B+\omega )\cdot \Gamma _i <0$, see {\color{blue} Remark \ref{rmk:finite-perturbation}} for a detailed discussion.

\begin{theorem} \label{t-pmmp} Let $(X,B)$ be a dlt pair, where $X$ is a $\mbQ$-factorial compact K\"ahler  $3$-fold, and $K_X+B$ is pseudo-effective. Let $R$ be a $(K_X+B)$-negative extremal ray supported by a nef class $\alpha$. Then
\begin{enumerate}
\item If $R$ is small, then the contraction $c_R:X\to Y$ of $R$ exists and $Y$ is a compact K\"ahler space.
\item If $R$ is divisorial and $n(\alpha )=0$, then  the contraction  $c_R:X\to Y$ of $R$ exists and $Y$ is a compact K\"ahler space with $\Q$-factorial singularities.
\item  Assume that $R$ is divisorial, $n(\alpha ) = 1$, and one
of the following conditions is satisfied: 
\begin{enumerate}
	\item $X$ has terminal singularities and $K_X\cdot C<0$, where $R=\mbR^+\cdot[C]$ for some curve $C\subset X$, or
	\item $S$ has semi-log canonical singularities, where $S$ is the unique surface covered by curves in the class $R$.
\end{enumerate}
Then the contraction $c_R:X\to Y$ of $R$ exists and $Y$ is a compact K\"ahler space with $\mbQ$-factorial singularities.
 \end{enumerate}
\end{theorem}
\begin{proof}
(1) follows from \cite[Theorem 2.28]{DO23}.
(2) follows from \cite[Theorem 2.30]{DO23}. 
(3) follows from \cite[Proposition 2.31]{DO23}. 
These results also follow from \cite[Theorem 1.2]{DH23}. 
\end{proof}

\begin{theorem}\cite[Theorem 4.3]{CHP16}\label{t-flip}
	  Let $(X,B)$ be a dlt pair, where $X$ is a $\mbQ$-factorial compact K\"ahler $3$-fold. Let $X\to Y$ be a $(K_X+B)$-flipping contraction, then the flip $X^+\to Y$ exists so that $X^+$ is a $\mbQ$-factorial compact K\"ahler $3$-fold  and $(X^+,B^+)$ is dlt, where $B^+:=\phi_*B,$ and $ \phi:X\bir X^+$ is the induced bimeromorphic map.
\end{theorem}

\begin{theorem}\cite[Theorem 1.12]{DO23}\label{t-term} 
	 Let $(X,B)$ be a dlt pair, where $X$ is a $\mbQ$-factorial compact K\"ahler $3$-fold. Then any sequence of $(K_X+B)$-flips is finite.
\end{theorem}

\begin{proposition}\label{pro:relative-projective-mmp}
	Let $(X, B)$ be a $\mbQ$-factorial dlt pair and $f:X\to Y$ a projective surjective morphism between two normal compact analytic varieties. If $\dim X\<3$, then we can run a relative $(K_X+B)$-MMP over $Y$ which terminates either with a minimal model or a Mori fiber space, according to whether $K_X+B$ is pseudo-effective over $Y$ or not.
Moreover, if $Y$ is a K\"ahler variety and $X=X_0\bir X_1\bir\cdots \bir X_n\cdots$ are the steps of a $(K_X+B)$-MMP over $Y$, then every $X_i$ is a $\mbQ$-factorial K\"ahler variety for $i\>0$; additionally, if $\psi:X_n\to Y'$ is a Mori fiber space over $Y$, then $Y'$ is also K\"ahler.
\end{proposition}

\begin{proof}
Since $(X, B)$ is a $\mbQ$-factorial dlt pair, $(X, (1-\ve)B)$ is klt for any $0<\ve\leq1$. Recall that the (relative) Mori cone $\NE(X/Y)$ is a strongly convex closed cone, and hence it is the convex hull of its extremal rays. In particular, if $K_X+B$ is not nef over $Y$, then there is a $(K_X+B)$-negative extremal ray $R$ of $\NE(X/Y)$; note that at this stage we do not know whether $R$ is generated by an irreducible curve or not. Then there is a $0<\delta\ll 1$ such that $(K_X+(1-\delta)B)\cdot R<0$. Since $(X, (1-\delta)B)$ is klt, and the (relative) cone and contraction theorems (for projective morphisms) are known for klt pairs due to \cite[Theorem 4.12]{Nak87}, it follows that $R$ can be contracted by a projective morphism over $Y$.
 Since $\dim X\<3$, the existence of flips (over $Y$) follows from Theorem \ref{t-flip}. The termination of flips (over $Y$) follows from Theorem \ref{t-term}. The proof of the fact that a $(K_X+B)$-MMP over $Y$ terminates with either  a minimal model or a Mori fiber space according to whether $K_X+B$ is pseudo-effective over $Y$ or not, works exactly as in the algebraic case, since $f:X\to Y$ is a projective morphism. The $\mbQ$-factoriality condition is preserved at each step as a formal consequence of the contraction theorem as in the algebraic case. 
	Now if $Y$ is K\"ahler, then by \cite[Proposition 1.3.1, page 24]{Var89}, $X=X_0$ is K\"ahler. If $g_i:X_i\to Z_i$ is a contraction of a $(K_{X_i}+B_i)$-negative extremal ray of $\NE(X_i/Y)$, then by the relative base-point free theorem  \cite[Theorem 4.10]{Nak87}, it follows that the induced morphism $h_i:Z_i\to Y$ is projective (note that arguing as above we may replace $(X,B)$ by $(X,(1-\delta)B)$ and hence we may assume that $(X,B)$ is klt so that \cite[Theorem 4.10]{Nak87} applies). Then again from \cite[Proposition 1.3.1, page 24]{Var89} it follows that $Z_i$ is K\"ahler. If $g_i$ is a flipping contraction and $g_i^+:X_{i+1}\to Z_i$ is the flip, then again $X_{i+1}$ is K\"ahler by the same argument.
	In the Mori fiber space case again by a similar argument it follows that $Y'$ is K\"ahler.
	
	\end{proof}

\begin{lemma}\label{l-ter} Let $(X,B)$ be a klt pair, where $X$ is a compact K\"ahler $3$-fold. Then the following morphisms exist:
	 \begin{enumerate}
\item a projective small bimeromorphic morphism $\nu :X'\to X$ such that 
$X'$ is strongly $\mbQ$-factorial, and
\item\label{item:terminalization} a projective bimeromorphic morphism $\nu :X'\to X$ such that 
$X'$ is strongly $\mbQ$-factorial and $(X',B_{X'})$ is a terminal pair such that $K_{X'}+B_{X'}=\nu ^*(K_X+B)$.
\end{enumerate}
\end{lemma}
\begin{proof}
	  (1) Let $\nu :X'\to X$ be a log resolution of the pair $(X, B)$. We may assume that $\nu$ is a projective morphism. Write $\nu ^*(K_X+B)=K_{X'}+B'-E'$, where $B',E'\geq 0$, $\nu_*B'=B$ and $B'$ and $E'$ have no common components. Choose $\ve\in\mbQ^{+}$ sufficiently small so that $({X'},B'+\ve{\rm Ex}(\nu))$ is klt. Next we run a $(K_{X'}+B'+\ve {\rm Ex}(\nu))$-MMP over $X$ using Proposition \ref{pro:relative-projective-mmp}. Replacing $X'$ by the output of this MMP, we may assume that $E'+\ve {\rm Ex}(\nu)$ is nef over $X$ and $X'$ is strongly $\Q$-factorial (see Lemmas \ref{lem:manifolds-are-strongly-Q-factorial} and \ref{l-sqfmmp}). Then by the negativity lemma (see \cite[Lemma 1.3]{Wang21}), we have $E'+\ve {\rm Ex}(\nu)=0$ and hence $\nu$ is small. 

For a proof of (2), first replace $\nu$ by a higher log resolution if necessary so that $X'$ contains all exceptional divisors $E$ over $X$ such that the discrepancy $a(E, X, B)\<0$. Then we run a $(K_{X'}+B')$-MMP over $X$. Replacing $X'$ by the output of this MMP and applying the negativity lemma we obtain the required result.\\
\end{proof}

\begin{lemma}[DLT Modification]
\label{lem:terminal-dlt-model}
Let $X$ be 
compact K\"ahler $3$-fold and $(X,B)$ a log canonical pair.  
Then there exists a projective bimeromorphic morphism  $f : (X', B') \to (X,B)$ such that 
\begin{enumerate}[label=(\roman*)]
\item  $X'$ has strongly $\mathbb{Q}$-factorial terminal singularities,
\item $(X',B')$ is a dlt pair, and
\item $K_{X'}+B'=f^*(K_X+B)$.
\end{enumerate}
\end{lemma}

\begin{proof} This follows from \cite[Corollary 1.30]{DO23} and Lemma \ref{l-sqfmmp}.
\end{proof}
\begin{proposition}\label{p-rccres} Let $X$ be a 
compact K\"ahler $3$-fold with klt singularities and $\mu:X'\to X$ a proper bimeromorphic morphism. Then every fiber of $\mu$ is rationally chain connected. 
\end{proposition}
\begin{proof} 
	Let $\nu:\tilde X\to X$ be a resolution of singularities dominating $X'$. Then it suffices to show that  every fiber of $\nu$ is rationally chain connected. Thus replacing $X'$ by $\tilde{X}$ and $\mu$ by $\nu$ we may assume that $X'$ is smooth and $\mu$ is projective.
 The relative minimal model program holds in this context by Proposition \ref{pro:relative-projective-mmp}, and thus 
 the result now follows from \cite[Theorem 1.2]{HM07} (with $S=X$).
\end{proof}

\subsection{Technical lemmas}
In this subsection we will prove some technical results which will be used in the rest of the article. Note that some of the results here are obvious for projective varieties but not necessarily so for analytic varieties. We will use these results throughout the article without further reference.

\begin{remark} Recall that if $f:X\to Y$ is a projective morphism between two analytic varieties and $C\subset Y$ is a (compact) irreducible curve, then there is a compact irreducible curve $\Gamma\subset X$ such that $f(\Gamma)=C$. Moreover, if  $f:X\to Y$ is a proper bimeromorphic morphism between analytic varieties, then from Chow's lemma (Theorem \ref{thm:chow}) it follows that for every $y\in f(\Ex(f))$, $f^{-1}(y)$ is covered by (compact) curves.\\
\end{remark}

\begin{lemma}\label{lem:dlt-rational}
	 If $(X,B)$ is dlt, then $X$ has rational singularities.
\end{lemma}

\begin{proof} See \cite[Theorem 3.12]{Fuj22}.

\end{proof}

\begin{lemma}\label{l-relpic}
	 Let $f:X'\to X$ be a proper bimeromorphic morphism between strongly $\mbQ$-factorial normal compact K\"ahler $3$-folds with klt  singularities. Then the $f$-exceptional divisors give a basis for $N^1(X'/X):=N^1(X')/f^*N^1(X)$. 
	 \end{lemma}

\begin{proof} 
Let $\nu:X''\to X'$ be a resolution of singularities of $X'$ such that $f\circ \nu$ is a projective morphism (see Theorem \ref{thm:chow}). We will show that $N^1(X''/X):=N^1(X'')/(f\circ \nu)^*N^1(X)$ is generated by the $(f\circ \nu)$-exceptional divisors. To this end let $f'=f\circ \nu$, then by Lemma \ref{c-small}, there is a reduced $f'$-exceptional divisor $E$ so that $\Supp E=\Ex(f')$. We run the $(K_{X''}+E)$-MMP over $X$ using Proposition \ref{pro:relative-projective-mmp}. Note that since $X$ has klt singularities, this MMP contracts all $f'$-exceptional divisors. Let $X''=X_0\dasharrow X_1\dasharrow \ldots \dasharrow X_n$ be the corresponding MMP over $X$.  Since $X$ is strongly $\mbQ$-factorial, so is $X_n$ by Lemma \ref{l-sqfmmp}, and then it follows from Lemma \ref{c-small} that $X_n=X$. By \cite[Proposition 3.1]{CHP16} it follows that if $X_i\dasharrow X_{i+1}$ is a flip, then $N^1(X_i)\cong N^1(X_{i+1})$ and if $X_i\dasharrow X_{i+1}$ is a divisorial contraction, then $N^1(X_i)/N^1(X_{i+1})$ is one dimensional and generated by the corresponding exceptional divisor.
Thus $N^1(X''/X)$ has a basis given by the $f'$-exceptional divisors.
Similarly one sees that $N^1(X''/X')$ has a basis given by the $\nu$-exceptional divisors.
The lemma now follows easily.

\end{proof}

\begin{lemma}\label{lem:movable-curves-on-surfaces}
	Let $S$ be a smooth projective surface such that $H^2(S, \mcO_S)=0$. Let $\alpha$ be a nef $(1, 1)$-class such that $\alpha^2=0$ and $K_S\cdot \alpha<0$. Then $S$ is covered by $\alpha$-trivial curves.
\end{lemma}

\begin{proof}
	Since $H^2(S, \mcO_S)=0$, from the exponential sequence and \cite[Eqn. (2) and (3), page 223]{HP16}, it follows that $\NS(S)_\mbR={H}^{1, 1}_{\rm BC}(S)$. In particular, $\alpha$ is a nef $\mbR$-divisor. Recall that on a projective surface the cone $\NF(S)$ of numerical classes of nef curves is equal to the cone of numerical classes of movable curves $\NM(S)$. Fix an ample $\mbQ$-divisor $A$. 
 By the structure theorem for the cone of movable curves $\NM(S)$, see \cite[Theorem 1.9]{Das20} (also see \cite[Theorem 1.3]{Ara10} and \cite[Theorem 1.3]{Leh12}), for every $\epsilon >0$ we have a decomposition $\alpha =C_\epsilon+\sum _{i=1}^k\lambda _{i,\epsilon}M_{i,\epsilon}$, where $C_\epsilon$ is a pseudo-effective $1$-cycle such that $(K_S+\epsilon A)\cdot C_\epsilon \geq 0$, $\lambda _{i,\epsilon}\>0$,  the $M_{i,\epsilon}$ are movable curves for all $1\<i\<k$. Since $\alpha $ is nef, then $\alpha\cdot C_\epsilon\geq 0$ and $\alpha\cdot M_{i,\epsilon}\geq 0$. Since $\alpha ^2=0$, then $\alpha\cdot M_{i,\epsilon}=\alpha\cdot C_\epsilon=0$ for all $1\<i\<k$. If $k>0$, then as $M_{i,\epsilon}$ is movable the lemma is proved. Otherwise, $k=0$ and $\alpha \equiv C_\epsilon$, so that $(K_S+\epsilon A)\cdot \alpha \geq 0$.
 Taking the limit as $\epsilon $ goes to $0$, it follows that $K_S\cdot \alpha \geq 0$ contradicting our assumptions.
\end{proof}

\begin{definition}\cite[Def. 2.2]{Bou04}\label{def:modified-kahler}
	Let $X$ be a normal compact K\"ahler variety. 
	\begin{enumerate}
		\item[(i)] A closed positive bi-degree $(1, 1)$-current $T$ with local potentials is called a \emph{K\"ahler current}, if $T-\omega$ is a positive current for some K\"ahler form $\omega$ on $X$. 
		\item[(ii)] A class $\alpha\in H^{1, 1}_{\rm BC}(X)$ is called \emph{big} if it contains a K\"ahler current $T$.
		\item[(iii)] A big class $\alpha\in H^{1, 1}_{\rm BC}(X)$ is called a \emph{modified K\"ahler class} if it contains a K\"ahler current $T$ such that the Lelong number $\nu(T, D)=0$ for all prime Weil divisors $D$ on $X$.
	\end{enumerate}	
\end{definition}
The following lemma is a singular version of \cite[Proposition 2.3]{Bou04}.
\begin{lemma}\label{lem:modified-kahler}
	Let $X$ be a normal compact K\"ahler variety and $\alpha\in H^{1, 1}_{\rm BC}(X)$. Then $\alpha$ is a modified K\"ahler class if and only if there is a projective bimeromorphic morphism $\mu:X'\to X$ from a K\"ahler manifold $X'$ and a K\"ahler class $\alpha'$ such that $\mu^*\alpha=\alpha'+E$, where $E\>0$ is an effective and $\mu$-exceptional $\mbR$-divisor. In particular, here $-E$ is $\mu$-ample and $\Supp(E)=\Ex(f)$.
\end{lemma}

\begin{proof}
The if part is obvious, e.g., see the first part of the proof of \cite[Proposition 2.3]{Bou04}. We will now prove the converse statement, to that end, let $f:X'\to X$ be a resolution of singularities of $X$. Since $\alpha$ is a modified K\"ahler class, there is a K\"ahler current $T$ with (psh) local potentials in the class $\alpha$ such that $\nu(T, P)=0$ for all prime Weil divisors $P$ on $X$. Since $T$ has psh local potentials, $f^*T$ can be defined by simply pulling back the local potentials. Since $T$ is a K\"ahler current and the exceptional locus of $f$ supports an effective relative anti-ample divisor $F$, it follows that $f^*T-\ve F$ is also a K\"ahler current in the class $f^*\alpha-\ve F$ for $0<\ve \ll 1$. 
    Choose a K\"ahler form $\omega'$ on $X'$ such that $f^*T-\ve F\>\omega'$ and let $T':=f^*T-\ve F-\omega'\>0$. Let $T'=D+S'$ be the Siu decomposition of $T'$, where $D=\sum_Q\nu(T', Q)Q$ and $S'$ is a closed positive bi-degree $(1,1)$-current such that $\nu(S', Q)=0$ for all prime Weil divisor $Q$ on $X'$. We claim that $D$ is an $\mbR$-divisor, i.e. $\nu(T', Q)=0$ for all but finitely many prime Weil divisors $Q$. To see this, observe that 
    \[\nu(f^*T, Q)=\nu(T'+\ve F+\omega', Q)=\nu(T'+\ve F, Q)+\nu(\omega', Q)=\nu(T'+\ve F, Q)\geq \nu(T', Q),\] since $\nu(\omega', Q)=0$ as $\omega'$ is a smooth form. Now from the definition of $T$ it follows that $\nu(f^*T, Q)=0$ if $Q$ is not $f$-exceptional, and hence $D$ is an effective $f$-exceptional divisor. Thus we have $f^*T=D+\ve F+U'$, where $U'=S'+\omega'$. Since $\nu(U', Q)=\nu(S'+\omega', Q)=0$ for all prime Weil divisors $Q$ on $X'$, it follows that the $(1,1)$-class $[U']$ is modified K\"ahler. 
    By Demailly's regularization theorem, there is a K\"ahler current $U_k$ with analytic singularities in the class $[U']$ such that $\nu(U_k, Q)\<\nu(U', Q)$ for all prime Weil divisors $Q$ on $X'$. In particular, $\nu(U_k, Q)=0$ for all prime Weil divisors $Q$ on $X'$. Let $g:X''\to X'$ be a resolution of singularities of $U_k$ so that $g^*U_k=\Theta+G$, where $\Theta$ is a smooth form and $G$ is an effective $g$-exceptional $\mbR$-divisor. Note that since $U_k\>\epsilon\omega'$ for some $\epsilon>0$, we have $\Theta\>\epsilon g^*\omega'$. Then by \cite[Lemma 2.9]{Bou02ar}, there is an effective $g$-exceptional $\mbR$-divisor $E$ on $X''$ such that $\Theta-E$ is cohomologous to a K\"ahler form, i.e. $[\Theta-E]=[\omega'']$ for some K\"ahler form $\omega''$ on $X''$. In particular, $[g^*U_k]=[\omega'']+[E+G]$ and thus we have
    \[
        \begin{split}
            (f\circ g)^*\alpha=g^*([D+\varepsilon F]+[U_k])=[\omega'']+[g^*(D+\ve F)+E+G],
        \end{split}
    \]
where $g^*(D+\ve F)+E+G$ is a $f\circ g$-exceptional effective $\mbR$-divisor. This completes our proof.

\end{proof}

\begin{corollary}\cite[Proposition 2.4]{Bou04}\label{cor:modified-kahler}
	Let $X$ be a $\mbQ$-factorial compact K\"ahler $3$-fold with klt singularities, and $\alpha\in H^{1, 1}_{\rm BC}(X)$ a modified K\"ahler class. If $S\subset X$ is an irreducible surface and $\tilde S\to S$ its normalization, then $\alpha |_{\tilde S}$ is big.
\end{corollary}

\begin{proof} By Lemma \ref{lem:modified-kahler}  there is a projective bimeromorphic morphism $\mu:X'\to X$ from a K\"ahler manifold $X'$, a K\"ahler class $\alpha'$ and an effective $\mu$-exceptional divisor $E\>0$ such that $\mu^*\alpha=\alpha'+E$. Let $S'=\mu ^{-1}_*S$ and $\nu:S'\to S$. We may assume that $S'$ is smooth and hence $\nu$ factors through $\tilde \nu:S'\to \tilde S$. Since $S'$ is not contained in the support of $E$, then \[\tilde \nu ^*(\alpha |_{\tilde S})=(\mu ^*\alpha ) |_{S'}=(\alpha'+[E])|_{S'}\] is big and so $\alpha |_{\tilde S}$ is also big.\\
\end{proof}

\begin{lemma}\label{l-mod}
	 Let $\phi: X\dasharrow X'$ be either a divisorial contraction or a flip between two compact K\"ahler $3$-folds with $\mbQ$-factorial klt singularities. If $\omega$ is a modified K\"ahler class on $X$, then $\omega':=\phi _*\omega$ is a modified K\"ahler class on $X'$.
 \end{lemma}
\begin{proof} This follows easily from Lemma \ref{lem:modified-kahler}.
\end{proof}

\begin{lemma}\label{lem:nef-and-big-to-modified-kahler}
Let $(X, B)$ be a klt pair, where $X$ is a $\mbQ$-factorial compact K\"ahler $3$-fold. Let $\alpha$ be a nef and big $(1, 1)$-class on $X$. Then there exist a modified K\"ahler class $\Theta$ and an effective $\mbQ$-divisor $F\>0$ such that $\alpha=\Theta+F$ and $(X, B+F)$ is klt.	
\end{lemma}

\begin{proof} Let $f:X'\to X$ be a resolution, then $f^*\alpha$ is big.
	By \cite[Theorem 1.4]{Bou02}, passing to a higher resolution, we may assume that $f^*\alpha=\omega'+E$, where $\omega'$ is a K\"ahler form and $E\>0$ is an effective $\mbR$-divisor. Since being K\"ahler is an open condition, we may assume that $E\geq 0$ is an effective $\mbQ$-divisor. 
 
 We may also assume that $f$ is a projective.   Then we can rewrite $\alpha$ as	\[		\alpha=f_*(\ve\omega'+(1-\ve)f^*\alpha)+\ve f_*E\quad\mbox{ for } 0<\ve< 1.			\]
We may assume that $\ve \in \Q$.	
Note that $\ve\omega'+(1-\ve)f^*\alpha$ is a K\"ahler class on $X'$, since $\alpha$ is nef, and therefore $\Theta':=f_*(\ve\omega'+(1-\ve)f^*\alpha)$ is a modified K\"ahler class on $X$. Observe that $(X, B+\ve f_*E)$ is klt for sufficently small $\ve\in\mbQ^+$. 
\end{proof}

\begin{lemma}[Hodge Index Theorem]\label{lem:hodge-index-theorem}\cite[Corollary
IV.2.15 and Theorem IV.3.1]{BHPV04}
	Let $S$ be a smooth compact K\"ahler surface and $\omega\in H^{1,1}(S)$ a $(1, 1)$-class such that $\omega^2>0$. Let $\alpha
	\in H^{1, 1}(S)$ be a $(1, 1)$-class on $S$. If $\omega\cdot \alpha=0$, then $\alpha^2\<0$; moreover, if $\alpha\neq  0$, then $\alpha^2<0$.\\
\end{lemma}

\begin{lemma}\label{lem:nef-reduction-property}
Let $f:S\to T$ be a proper morphism with connected fibers from a normal compact analytic surface to a smooth projective curve $T$. Let $\alpha\in H^{1,1}_{\rm BC}(S)$ be a nef class and $C\subset S$ a curve such that $f(C)=T$ and $\alpha\cdot C=\alpha\cdot F=0$, where $F$ is a general fiber of $f$. Then $\alpha\num 0$, i.e. $\alpha\cdot \Gamma=0$ for all curves $\Gamma\subset S$; in particular, the nef dimension $\nu_{\rm nef}(\alpha)=0$.  
\end{lemma}

\begin{proof}
  This follows from the same arguments as in the proof of \cite[Proposition 2.5]{BCE02} using Lemma \ref{lem:hodge-index-theorem} above.   
\end{proof}

\begin{lemma} \label{l-adj} Let $X$ be a normal $\Q$-factorial K\"ahler $3$-fold and $S$ a prime divisor with minimal resolution $\nu:S'\to S$. Then there is an effective $\mbQ$-divisor $E$ on $S'$ such that
\[ \nu^*((K_X+S)|_S)=K_{S'}+E.\]\end{lemma}
\begin{proof} See \cite[\S 2B]{CHP16}. 
\end{proof}


\begin{lemma}\label{l-proj} 
	Let $X$ be a uniruled normal 
 compact K\"ahler $3$-fold with {klt} singularities such that the base of its MRC fibration has dimension less than $2$, then $X$ is a projective variety and $H^2(X,\OO _X)=0$. 
\end{lemma}
	
\begin{proof} Let $\pi:X\bir Z$ be the {MRC fibration} (see \cite[Remark 6.10]{CH20}). By assumption $\dim Z=0$ or $1$. 
First note that, since $X$ has rational singularities, by Theorem \ref{t-nam}, $X$ is projective if and only if any resolution $\tilde X$ of the  singularities of $X$ is projective. 
Note also that by Proposition \ref{p-rccres}, the fibers of $\nu:\tilde X\to X$ are rationally chain connected, thus it follows that $\tilde X \bir Z$ is also a MRC fibration. 
Thus replacing $X$ by a resolution of singularities we may assume that $X$ is a compact K\"ahler manifold.     Since the general fibers of $\pi$ are rationally connected, by \cite[Corollary 4.18]{Deb01} $H^0(F, \Omega ^i _F)=0$ for all $i\>1$ where $F$ is a general fiber. We claim that $H^0(X, \Omega^2_X)=0$. If $\dim Z=0$, then this is clear, so assume that $\dim Z=1$. Then observe that the following exact sequence 
	\[
		\xymatrixcolsep{3pc}\xymatrix{ \pi^*\Omega_Z\ar[r] & \Omega_X\ar[r] & \Omega_{X/Z}\ar[r] & 0 }
	\]
is left exact over a Zariski open dense subset of $Z$. This follows from the generic smoothness of $\pi$ and the fact that the MRC fibration is an almost holomorphic map.
Restricting this sequence to a general fiber $F$ of $\pi$ we get the following short exact sequence
	\begin{equation}\label{eqn:mcc-forms-1}
		\xymatrixcolsep{3pc}\xymatrix{ 0\ar[r] & \mcO_F\ar[r] & \Omega_X|_F\ar[r] & \Omega_F\ar[r] & 0. }
	\end{equation}
	Thus we have a short exact sequence
	\begin{equation}\label{eqn:mcc-forms-2}
		\xymatrixcolsep{3pc}\xymatrix{ 0\ar[r] & \Omega_F\ar[r] & \Omega ^2_X|_F\ar[r] & \Omega ^2_F\ar[r] & 0. }
	\end{equation}
	Since, as observed above, $H^0(F, \Omega ^i_F)=0$ for $i=1,2$, we have  $H^0(F, \Omega^2_X|_F)=0$. In particular, $s|_F=0$ for any section $s\in H^0(X, \Omega^2_X)$. Since $\Omega^2_X$ is torsion free, it follows that $H^0(X, \Omega^2_X)=0$. Then $H^2(X, \mcO_X)=\overline{H^0(X, \Omega^2_X)}=0$ and by Kodaira's projectivity criterion (Theorem \ref{t-kod}), we have that $X$ is projective. 
	
 \end{proof}

\begin{lemma}\label{lem:projectivity-of-kahler-surface}
Let $S$ be a smooth compact K\"ahler surface. If $K_S$ is not pseudo-effective, then $S$ is projective.	
\end{lemma}

\begin{proof}
Since $K_S$ is not pseudo-effective and $S$ is K\"ahler, then $H^2(S,\OO _S)=H^0(S,K_S)^* =0$. The claim then follows from Kodaira's criterion (Theorem \ref{t-kod}).
\end{proof}

\begin{lemma}\label{lem:rational-singularities}
Let $(X, B\>0)$ be a log pair, where $X$ is a normal compact analytic variety with $\mbQ$-Gorenstein singularities, i.e. $(\omega_X^{\otimes m})^{**}$ is a line bundle for some $m>0$. Let $f:X\to Y$ be a proper surjective morphism to a normal compact analytic variety $Y$. Assume that one of the following conditions is satisfied:
\begin{enumerate}[label=(\roman*)]
	\item $(X, B)$ is dlt and $-(K_X+B)$ is $f$-ample, or
	\item $(X, B)$ is klt and $-(K_X+B)$ is $f$-nef-big.
\end{enumerate}
Then $Y$ has rational singularities.   	
\end{lemma}

\begin{proof}
Since $X$ is $\mbQ$-Gorenstein, $B$ is a $\mbQ$-Cartier divisor and $X$ has klt singularities. Thus by Lemma \ref{lem:dlt-rational}, $X$ has rational singularities. In the dlt case, perturbing the coefficients of $B$ we may assume that $(X, B)$ is klt and $-(K_X+B)$ is still $f$-ample. Therefore by Theorem \ref{thm:relative-kvv}, in both of the cases above we have $R^if_*\mcO_X=0$ for all $i>0$. Then by \cite[Theorem 1]{Kov00} $Y$ has rational singularities. Note that the proof of \cite[Theorem 1]{Kov00} uses Grothendieck duality and the Grauert-Riemenschneider vanishing theorem, both of which are known in the analytic category due to \cite{Ram74, RRV71} and \cite[Corollary II]{Tak85}, respectively.\\
\end{proof}

\section{Minimal model program for dlt pseudo-effective pairs}
 Throughout this section $(X,B)$ will be a dlt pair {with} $X$ a strongly $\mbQ$-factorial compact K\"ahler $3$-fold and $\alpha =K_X+B+\omega$ a nef and big (1,1)-class,  such that $\omega$ is a K\"ahler class, the corresponding extremal ray $R:=\Null(\alpha)\cap\NA(X)$ is of divisorial type and $n(\alpha)=1$. Since $\alpha -\epsilon \omega$ is also big for any $0<\epsilon \ll 1$, by \cite[Thm. 3.12 and Prop. 3.8]{Bou04} there exist positive real numbers $\lambda_j>0$ and irreducible surfaces $S_j\subset X$ such that
\begin{equation}\label{eq-1}
	K_X+B+(1-\epsilon)\omega \equiv \sum \lambda _jS_j+N(K_X+B+(1-\epsilon)\omega),
\end{equation}
	 where $N=N(K_X+B+(1-\epsilon)\omega)$ is a pseudo-effective class which is nef in codimension $1$, i.e. $N|_{S'}$ is pseudo-effective for any surface $S'\subset X$.
  If $[C]\in R$ and $C$ belongs to a positive dimensional family of curves, since $(K_X+B+(1-\epsilon)\omega)\cdot C=-\epsilon\omega\cdot C<0$, then $S_j\cdot R<0$ for some $j$. But then all curves $[C]\in R$ must be contained in $S_j$ and $S_j$ is covered by such curves. It follows that $R\cdot S_k\geq 0$ for all $k\ne j$. Thus $S_j$ is the unique such surface.
  
	 We will next prove Theorem \ref{t-div1} when $X$ is strongly $\Q$-factorial and $K_X+B$ is pseudo-effective. Since this is the most delicate proof of the paper, we give a brief explanation of the main ideas behind this proof. Let $(X,B)$ be a klt pair, $R$ a $(K_X+B)$-negative extremal ray of divisorial type cut out by a nef class $\alpha$.
  The strategy is to construct the divisorial contraction $X\to Z$ by running a minimal model program for a dlt pair $(X',\Delta ')$ where every step is negative for some component of $\lfloor \Delta '\rfloor$ and thus the corresponding flips and contractions exist by Theorems \ref{t-pmmp} and  \ref{t-flip}.
  The dlt pair $(X',\Delta ')$ is constructed as follows. Let $S\subset X$ be the surface covered by  the curves $C\subset X$ such that $[C]\in R$. Let $b={\rm mult}_SB$. The pair $(X,\Delta :=B+(1-b)S)$ may not be dlt, so we take the corresponding dlt model $\mu:(X',\Delta ')\to (X,\Delta)$.
  Since $(X,B)$ is klt, $K_{X'}+\Delta '-\mu ^*(K_X+\Delta)$ is an effective divisor whose support coincides with the exceptional
  locus of $\mu$. If the divisorial contraction $f:X\to Z$ exists, then by standard arguments, $\alpha =f^*\alpha _Z$ for some K\"ahler class $\alpha _Z $ on $Z$ and $Z$ is the output of the $(K_{X'}+\Delta ')$-minimal model program over $Z$ or equivalently of the  $(K_{X'}+\Delta '+t\mu ^* \alpha )$-minimal model program for any $t\gg 0$.
  Since we do not yet know that $f:X\to Z$ exists, we mimic this strategy and we run the $\mu ^* \alpha$-trivial $(K_{X'}+\Delta ')$-minimal model program (each step of this minimal model program is a step of the $(K_{X'}+\Delta '+t\mu ^* \alpha )$-minimal model program for $t\gg 0$). We then show that the output of this minimal model program is indeed the required divisorial contraction. 
  \begin{theorem}\label{t-div1sqf}
      Theorem \ref{t-div1} holds when $X$ is strongly $\Q$-factorial and $K_X+B$ is pseudo-effective.
  \end{theorem}
\begin{proof}
  Recall that $\alpha$ is a  nef $(1, 1)$-class such that $\alpha \cdot C=0$ if and only if $[C]\in R$. Let $S$ be the unique surface which is covered by and contains all the curves $C\subset X$ such that $[C]\in R$. Let $\nu : S^\nu\to S$ be the normalization morphism and $\tilde f: S^\nu\to T$ the nef reduction of $\alpha|_{S^\nu}:=\nu ^* (\alpha|_S)$ (see \cite[Theorem 3.19]{HP15}). Since $n(\alpha)=1$, $T$ is a smooth projective curve. Moreover, we have $F\cdot \nu ^* (\alpha|_S) =0 $ for any fiber $F$ of $\tilde f$. We also have that $\alpha \cdot C>0$ if $C\subset X$ is not contained in $S$ or if $C$ is contained in $S$ but dominates $T$ (cf. \cite[Proposition 2.5]{BCE02}). By assumption $\omega =\alpha -(K_X+B)$ is a K\"ahler class. Replacing $B$ by $(1-\varepsilon )B$ and $\omega $ by $\omega +\varepsilon B$ for sufficiently small $\ve\in\mbQ^+$, we may assume that $(X,B)$ is klt.
  
Let $b={\rm mult }_SB$. For two divisors $D$ and $D'$ we say $D\equiv _\alpha D'$ if and only if $(D-D')\cdot C=0$ for any curve $C\subset X$ such that $\alpha \cdot C=0$.
Since, as observed above, $S\cdot R<0$ 
and $\alpha$ supports the extremal ray $R$, we have $K_X+B\num_\alpha aS$ for some $a>0$.

 Let $\mu :X'\to X$ be a log resolution of $(X,B)$ and set $\Delta '=\mu ^{-1}_*(B+(1-b) S)+{\rm Ex}(\mu)$ and $S'=\mu ^{-1}_*S$. Since $\mu$ is a projective morphism, after running a $(K_{X'}+\Delta')$-MMP over $X$ via Proposition \ref{pro:relative-projective-mmp}, we may assume that $K_{X'}+\Delta'$ is nef over $X$. By Lemma \ref{l-sqfmmp}, $X'$ is strongly $\mbQ$-factorial, and thus by Lemma \ref{l-relpic}, if $\eta$ is a K\"ahler class on $X'$, then $\eta \equiv _X -F$, where $F$ is an effective $\mu$-exceptional divisor. In particular the support of the exceptional locus equals ${\rm Supp}(F)$.  
 If $U=X\setminus S$ and $U'=\mu ^{-1}(U)$, then $K_{U'}+\Delta'|_{U'} \equiv _U \sum a_jE_j|_{U'}$,
 where the left hand side is $\mu|_{U'}$-nef and the right hand side is an effective $\mu|_{U'}$- exceptional divisor whose support equals ${\rm Ex}(\mu|_{U'})$. 
 By the negativity lemma, $\sum c_jE_j|_{U'}=0$ and hence ${\rm Ex}(\mu|_{U'})=\emptyset$. Thus  $\mu$ is an isomorphism over the complement of $S$.
 We then have
 \begin{equation}\label{eqn:alpha-trivial}
 	K_{X'}+\Delta' =\mu^*(K_X+B)+\sum c_jE_j+(1-b) S'\equiv _{\alpha'} \sum d_jE_j+d S',
 \end{equation}
where $\alpha':=\mu ^* \alpha$, $d_j\geq c_j>0$ and $d>1-b>0$. 
We will now run the $\alpha'$-trivial $(K_{X'}+\Delta ')$-MMP.
\begin{claim}\label{c-ammp}
There exists a sequence of $\alpha'$-trivial $(K_{X'}+\Delta ')$-flips and divisorial contractions
\[\phi _n:X'=X'_0\dasharrow X_1'\dasharrow X'_2 \dasharrow \cdots \dasharrow X_n'\] 
such that there are no $\alpha'_n$-trivial, $(K_{X'_n}+\Delta_n')$-negative curves on $X'_n$. 
Moreover, we have $K_{X'_n}+\Delta ' _n\equiv _{\alpha'_n} \phi _{n,*}(\sum d_jE_j+d S')$, where $\alpha _n'=\phi _{n,*}\alpha'$ is a nef $(1,1)$-class.
\end{claim}
\begin{proof}[Proof of Claim \ref{c-ammp}]
We will show that this MMP exists and each step preserves the above relation $K_{X'_i}+\Delta ' _i\equiv _{\alpha'_i} \phi _{i,*}(\sum d_jE_j+d S')$, where $\alpha _i'=\phi _{i,*}\alpha'$ is a nef $(1,1)$-class and $\phi _i:X'\dasharrow X'_i$ is the induced bimeromorphic map. 
We proceed by induction. Let $R_i$ be an $\alpha _i'$-trivial,  $(K_{X'_i}+\Delta ' _i)$-negative extremal ray. From the above 
relation, it follows that $\phi _{i,*}(\sum d_jE_j+d S')\cdot R_i<0$ and hence the contracted locus is contained in the support of $\lfloor \Delta _i'\rfloor $. 
Note that $R_i$ intersects some component $P$ of $\lfloor \Delta _i'\rfloor $ negatively and hence each contracted curve is contained in this component $P$. Since $(X',\Delta ')$ is dlt, so is $(X'_i,\Delta '_i)$ and hence $(P,\Delta _P)$ is dlt where $K_{P}+\Delta _P=(K_{X'_i}+\Delta '_i)|_P$. In particular $P$ has semi-log canonical singularities.
By Theorem \ref{t-pmmp} and Theorem \ref{t-flip}, we may flip/contract $R_i$ via $X'_i\dasharrow X'_{i+1}$. Let $g:X'_i\to Z$ be the 
contraction of $R_i$ ($Z=X_{i+1}'$ if $R_i$ is divisorial). By  \cite[Proposition 3.1(5)]{CHP16}, $\alpha _i'=g^* \alpha _Z$, where $\alpha _Z$ is a nef $(1, 1)$-class on $Z$ and $g_*(K_{X'_i}+\Delta ' _i-\phi _{i,*}(\sum d_jE_j+d S'))$ is a $\Q$-Cartier divisor such that 
\[ g_*\left(K_{X'_i}+\Delta'_i-\phi _{i,*}\left(\sum d_jE_j+d S'\right)\right)\equiv _{\alpha _Z}0.\] 
Pulling back to $X'_{i+1}$ we have $K_{X'_{i+1}}+\Delta ' _{i+1}\equiv _{\alpha'_{i+1}} \phi _{{i+1},*}(\sum d_jE_j+d S')$. By Lemma \ref{l-sqfmmp}, $X'_{i+1}$ is strongly $\mbQ$-factorial, and by Theorem \ref{t-term} after finitely many steps we may assume that there are no $\alpha'_n$-trivial, $(K_{X'_n}+\Delta_n')$-negative extremal rays. By the cone theorem (see Theorem \ref{t-cone}) it follows that $(K_{X'_n}+\Delta_n')\cdot C\geq 0$ for any $\alpha _n'$-trivial curve $C$.\end{proof} 

Recall that $\mu $ is an isomorphism over the complement of $S$, and hence the support of $S'+\sum E_j$ is equal to the support of $\mu ^{-1}(S)$.

\begin{claim}\label{clm:coefficients-of-F} $\phi _n$ contracts $S'$ and every $E_j$ such that $\nu _{\rm nef}(\mu ^* \alpha |_{E_j})=1$. 
\end{claim}

\begin{proof}[Proof of Claim \ref{clm:coefficients-of-F}]
Let $K_{X'}+B'=\mu ^*(K_X+B)$ and $F$ an effective exceptional $\mbQ$-divisor such that $-F$ is $\mu$-ample; note that $B'$ is not necessarily effective here.
Replacing $B'$ by $B'+\epsilon F$ and letting $\omega '=\mu ^* \omega -\epsilon F$ for some $0<\epsilon \ll 1$, we may write $\alpha '=\mu^*\alpha=K_{X'}+B'+\omega '$, where  $\omega '$ is a K\"ahler class on $X'$. Then $K_{X'}+\Delta '=K_{X'}+B'+\mathcal E '$, where $\mathcal E '\geq 0$ is a $\mbQ$-divisor such that ${\rm Supp}(\mathcal E ')=S'+\sum E_j$. Thus on $X'_n$ we have
\begin{equation}\label{eqn:on-X_n}
	\alpha'_n+\mathcal E'_n=K_{X'_n}+\Delta'_n+\omega'_n,
\end{equation} 
where $\omega ' _n\in H^{1,1}_{\rm BC}(X'_n)$ is a modified K\"ahler class (see Lemma \ref{l-mod}).
\begin{claim}\label{clm:kahler-divisor}
Let $\mathcal F:=\mu ^* S+\epsilon F$. For any $0<\epsilon\ll 1$ and $t\gg 0$,  $(-\mathcal F+t\alpha ' )|_{S'}$ and $(-\mathcal F+t\alpha ')|_{E_j}$ are K\"ahler for all $j$.\end{claim}
\begin{proof}[Proof of Claim \ref{clm:kahler-divisor}]
First, note that $-aS|_{S^\nu}\equiv_{\alpha} -(K_X+B)|_{S^\nu}\equiv_\alpha \omega |_{S^\nu}$ is ample over $T$. Since $-F$ is $\mu$-ample, then $(-\mu ^*S -\epsilon F)|_{S'}$ is ample over $T$ for any $0<\epsilon\ll 1$. Since $\alpha |_{S^\nu}\num \tilde f ^* \alpha _T$ where $\alpha _T$ is a K\"ahler class on $T$, it follows that 
$(-\mu ^*S -\epsilon F+t\alpha' )|_{S'}=(-\mcF+t\alpha')|_{S'}$ is K\"ahler for all $t\gg 0$.
Next, we consider an exceptional divisor $E_j$.
If $\dim \mu (E_j)=0$, then $\alpha '|_{E_j}=0=\mu ^* S|_{E_j}$ and as $-F$ is  $\mu$-ample the claim follows.
If $\dim \mu (E_j)=1$, let $E_j\to V_j$ be the Stein factorization of $\mu |_{E_j}$, then $V_j$ is a smooth curve.
If $\alpha \cdot \mu (E_j)>0$, then $\alpha |_{V_j}$ is ample (since $V_j$ is a curve, $H^2(V_j, \mcO_{V_j})=0$ and thus every $(1,1)$-class on $V_j$ is represented by an $\mbR$-divisor) and it follows that $(-S+t\alpha)|_{V_j}$ is ample for any $t\gg 0$. Since  $-F$ is $\mu$-ample, then $(\mu ^*(-S+t\alpha) -\epsilon F)|_{E_j}=(-\mcF+t\alpha')|_{E_j}$ is ample, hence K\"ahler. Finally, if $\alpha \cdot \mu (E_j)=0$, then $[\mu (E_j)]\in R$ and so $-S\cdot \mu (E_j)>0$. Thus $-S|_{V_j}$ is ample, and since  $-F$ is $\mu$-ample, then $-(\mu ^*S+\epsilon F)|_{E_j}$ is ample for $0<\epsilon \ll 1$. Since $\alpha ' |_{E_j}\num 0$ in this case, $(-(\mu ^*S+\epsilon F)+t\alpha')|_{E_j}=(-\mcF+t\alpha')|_{E_j}$ is ample (and hence K\"ahler) for any $t>0$.\end{proof}

We will now show that $\phi_n$ is $\mcF$-non-positive. This means that if $p:W\to X'$ and $q:W\to X'_n$
is the normalization of the graph of $\phi_n:X'\bir X'_n$, then $p^* \mcF \geq q^* \mcF _n$ where $\mcF _n:=\phi_{n, *}\mcF $.
By the negativity lemma, it suffices to show that $q^*\mcF _n-p^* \mcF $ is $q$-nef.
Suppose that $C\subset W$ is a $q$-exceptional curve, then $C$ is not $p$-exceptional and in particular $\alpha '\cdot p_* C=p^* \alpha '\cdot C=\alpha _n\cdot q_*C=0$. It follows that $p_*C$ is contained in $S'$ or in $\sum E_j$. Since 
$(-\mcF+t\alpha ') |_{S'}$ and $(-\mcF+t\alpha ') |_{E_j}$ are K\"ahler, \[(q^*\mcF_n-p^*\mcF)\cdot C=-p^* \mcF \cdot C=-\mcF \cdot p_*C=(-\mcF+t\alpha')\cdot p_*C>0\] as required. Thus $\phi_n$ is $\mcF$-non-positive.

Abusing notation we let $E_0=S'$. Assume by contradiction that  $\phi _{n,*}( E_j)\ne 0$, where $E_j$ is a component of $\sum_{j\>0} E_j$ with $\nu_{\rm nef}(\alpha '|_{E_j})=1$. Let $\lambda$ be the smallest positive rational number such that ${\mult}_{E_{j,n}}(\phi _{n,*}(\mathcal E'-\lambda \mathcal F))\leq 0$, where $E_{j,n}:=\phi_{n,*}E_j$, for every component $E_j$ of $\sum_{j\geq 0} E_j$ with $\nu_{\rm nef}(\alpha '|_{E_j})=1$.
Thus there is a component $E_{k,n}$ of $\phi _{n,*}(\sum E_j+S')$ with $\nu_{\rm nef}(\alpha '|_{E_k})=1$ such that
${\rm mult}_{E_{k,n}}( \phi _{n,*}(\mathcal E'-\lambda \mathcal F))=0$.
Since $\nu_{\rm nef}(\alpha '|_{E_k})=1$, the nef reduction map $E_k\to W_k$ is a surjective morphism to a smooth curve $W_k$.
Let $C_{k,w}$ be the fiber of the nef reduction map over a point $w\in W_k$.
We claim that the induced dominant meromorphic map $E_{k,n}\dasharrow W_k$ is a morphism with connected fibers, $\alpha'_n\cdot \Gamma=0$ for all fibers $\Gamma$ of this morphism and $\alpha'_n|_{E_{k,n}}\not\num 0$, i.e. $\alpha'_n|_{E_{k,n}}\cdot C_0\neq 0$ for some curve $C_0\subset E_{k,n}$. Moreover, we also claim that if $w\in W_k$ is general, then $\phi _n$ is an isomorphism on a neighborhood of $C_{k,w}$. Proceeding by induction, suppose that the claim holds for $\phi _i$ and consider a contracted curve $\Sigma \in X'_i$ (in the step $X'_i\bir X'_{i+1}$) which is contained in $E_{k,i}$. Note that $E_{k,i}\subset X'_i$ is a normal surface, since it is a dlt center of the dlt pair $(X'_i, \Delta'_i)$. Since $\alpha'_i|_{E_{k,i}}\cdot\Sigma=\alpha'_i\cdot\Sigma=0$ and $\alpha'_i|_{E_{k,i}}\not\num 0$, from Lemma \ref{lem:nef-reduction-property} it follows that $\Sigma$ is contained in a fiber of $E_{k,i}\to W_k$. 
Since $E_{k,i}$ is not contracted by $\psi_i:X'_i\bir X'_{i+1}$, $\Sigma$ is not contained in the general fibers of $h_i:E_{k,i}\to W_k$, consequently, $\psi_i$, and hence $\phi_{i+1}:X'_0:=X'\bir X'_{i+1}$ is an isomorphism near the general fibers of $h_i$. In particular, the induced dominant meromorphic map $h_{i+1}:E_{k, i+1}\bir W_k$ is almost holomorphic, i.e. over a dense Zariski open subset of $W_k$, the fibers of $h_{i+1}$ are all compact. This implies that $h_{i+1}$ is a morphism (as $W_k$ is a curve and $E_{k, i+1}$ is normal). The rest of the claim now follows immediately.

We will now identify $C_{k,w}$ with its image in $E_{k,n}$ for a general $w\in W_k$.
By the definition of $\lambda $ and $k$ we have
\begin{equation}\label{eqn:negative-intersection}
(\alpha'_n+\mathcal E'_n-\lambda \mathcal F _n)\cdot C_{k,w}=(\mathcal E'_n-\lambda \mathcal F _n)\cdot C_{k,w}\leq 0.	
\end{equation} 
Here we have used that fact that if $\nu _{\rm nef}(\alpha '|_{E_j})=0$, then $E_j$ does not intersect $C_{k,w}$ as otherwise $E_j\cap E_k$ contains a curve $\Gamma$ dominating $W_k$ and hence such that $\alpha '\cdot \Gamma>0$, contradicting the fact that $\alpha '|_{E_j}\equiv 0$.

Since $\omega'_n$ is a modified K\"ahler class, by Corollary \ref{cor:modified-kahler}, $\omega'_n|_{E_{k,n}}$ is big. In particular, $\omega'_n\cdot C_{k,w}>0$. Since $\alpha ' _n$ is nef and $\alpha ' _n\cdot C_{k,w}=0$, 
then by Claim \ref{c-ammp} $(K_{X'_n}+\Delta' _n)\cdot C_{k,w}\geq 0$.  Then from \eqref{eqn:negative-intersection} and \eqref{eqn:on-X_n} it follows that 
\begin{equation}\label{eqn:contradiction-inequality}
	 0\geq (\alpha'_n+\mathcal E'_n-\lambda \mathcal F _n)\cdot C_{k,w}=(K_{X'_n}+\Delta' _n+\omega'_n-\lambda \mathcal F_n )\cdot C_{k,w}> -\lambda  \mathcal F_n\cdot  C_{k,w} .
\end{equation} 

Since $p^*\mathcal F-q^*\mathcal F _n$ is an effective divisor whose support does not contain $q^{-1}_*C_{k,w}=p^{-1}_*C_{k,w}$ (via above identification),  then
\begin{equation}\label{eqn:c2}-\mathcal F_n\cdot  C_{k,w}=-q^* \mathcal F_n\cdot  q^{-1}_*C_{k,w}\geq -p^*\mathcal F\cdot  p^{-1}_*C_{k,w}=-\mathcal F\cdot  C_{k,w}.\end{equation} 
Since  $(-\mathcal F+t\alpha ')|_{E_k}$ is K\"ahler and $\alpha '\cdot C_{k,w}=0$, it follows that
\begin{equation}\label{eqn:c3}-\mathcal F\cdot  C_{k,w}= (-\mathcal F+t\alpha ')\cdot  C_{k,w}>0.\end{equation} 
Putting together equations \eqref{eqn:contradiction-inequality}, \eqref{eqn:c2}, and \eqref{eqn:c3}, we obtain a contradiction.

    
\end{proof}

\begin{claim}\label{c-zero} 
	$\phi _{n,*}(\sum E_j+S')=0$.
\end{claim}

\begin{proof}
Let $W$ be the normalization of the graph of the induced bimeromorphic map $\psi _n: X\dasharrow X_n'$, and $p:W\to X$ and $q:W\to X'_n$ the projections. Let $I$ and $I'$ be the sets of all indices of the $\mu$-exceptional divisors which are contracted and respectively, not contracted by $\phi_n$. Then $\phi_n$ contracts $S'$ and the divisor $\sum_{i\in I} E_i\subset\Ex(\mu)$, and $\psi_n$ contracts $S$ and extracts $\sum_{i\in I'}E_i\subset\Ex(\mu)$. 

Set $\mathcal G =q^*\phi _{n,*}(\sum d_jE_j+dS')\>0$. We claim that $\mathcal G$ is  $p$-exceptional, i.e. if $F$ is a component of $\mcG$, then $p_*F=0$.  
Let $W'$ be the normalization of the graph of $X'\bir X'_n$, and $p':W'\to X'$ and  $q':W'\to X'_n$ the projections. Then there is a unique morphism $\theta:W'\to W$ such that $p\circ\theta=\mu\circ p'$ and $q'=q\circ\theta$. 
\begin{equation}\label{eqn:comparing-graphs}
\xymatrix{
&& W\ar@/_2pc/[llddd]_p\ar@/^2pc/[rdd]^q &\\
&& W'\ar[dl]_{p'}\ar[dr]^{q'}\ar[u]_\theta &\\
& X'\ar[ld]_\mu\ar@{-->}[rr]^{\phi_n} & & X'_n\\ 
    X\ar@{-->}[urrr]_{\psi_n} & & & 
    }
\end{equation}
Let $F'$ be the normalization of $\theta^{-1}_*F$, then $F'$  is the normalization of a component of $q'^*\phi _{n,*}(\sum d_jE_j+dS')$ and $p(F)=\mu(p'(F'))$. So if $p_*F\neq 0$, then $p_*F=\mu_*p'_*F'=S$, and hence $p'_*F'=S'$;  in particular, in this case we have
\begin{equation}\label{eqn:p-excaptional}
    \nu_{\rm nef}((\mu\circ p')^*\alpha|_{F'})=\nu_{\rm nef}(p'^*\alpha'|_{S'})=1.
\end{equation}
Now we will show that this equation leads to a contradiction by proving that in fact $\nu_{\rm nef}((\mu\circ p')^*\alpha|_{F'})=\nu_{\rm nef}((p\circ\theta)^*\alpha|_{F'})=0$. To this end, first observe that if $F'$ is the normalization of a component of the strict transform ${q'}^{-1}_*(\phi_{n,*}(\sum d_jE_j+dS'))$, then from Claim \ref{clm:coefficients-of-F} it follows that
$\nu_{\rm nef}((p\circ\theta)^*\alpha|_{F'})=0$. Now assume that $F'$ is the normalization of a $q'$-exceptional divisor. If $q'(F')$ is a point, then clearly $\nu_{\rm nef}((p\circ\theta)^*\alpha|_{F'})=0$ as $(\mu\circ p')^*\alpha={p'}^*\alpha'={q'}^*\alpha'_n$. If $q'(F')=\Gamma'$ is a curve, then $\Gamma'$ is contained in a component of $\phi_{n,*}(\sum d_iE_j+dS')$ by our construction of $F'$ and thus from the Claim \ref{clm:coefficients-of-F} again it follows that $\alpha'_n\cdot \Gamma'=0$, in particular, $\nu_{\rm nef}((p\circ\theta)^*\alpha|_{F'})=0$.

Thus $p_*F=0$ must hold; in particular, the divisor $\mcG=q^*\phi _{n,*}(\sum d_jE_j+dS')$ is $p$-exceptional.

Now we will show that $\mathcal G$ is nef over $X$. To this end assume by contradiction  that $\mathcal G\cdot C<0$ for some curve $C\subset W$ such that $p(C)={\rm pt}$. 
Thus $C$ is contained in $F$, a component of $\mathcal G$. 
We have $\alpha '_n\cdot q_*C=q^*\alpha ' _n \cdot C=p^*\alpha\cdot C=0$.  
Since $\phi_{n, *}(\sum d_jE_j+dS')\equiv _{\alpha '_n}K_{X'_n}+\Delta'_n$, it follows that 
\[\mathcal G\cdot C=\phi_{n, *}\left(\sum d_jE_j+dS'\right)\cdot q_*C= (K_{X'_n}+\Delta'_n)\cdot q_*C\geq 0,\] 
which is  a contradiction. Thus $\mathcal G$ is nef over $X$.

Then by the negativity lemma we have $\mathcal G=0$, and hence $\phi _{n,*}(\sum E_j+S')=0$.

\end{proof}

\begin{claim} 
	$\psi_n :X\dasharrow X_n'$ is a morphism.
\end{claim}
\begin{proof}
By contradiction assume that $\psi_n$ is not a morphism. Let $W$ be a resolution of singularities of the graph of $\psi_n$ and $p:W\to X$ and $q:W\to X'_n$ the induced morphisms. By Theorem \ref{thm:chow}, possibly replacing $W$ by a higher resolution we may assume that $p$ is projective.
If $\psi_n$ is not a morphism, then there is a curve $C$ in $W$ such that $p_*C=0$ but $q_*C=C_n'\neq 0$. Let $\omega_n$ be a K\"ahler form on $X'_n$. Then $\omega_n\cdot C'_n>0$. Note that since $X$ is strongly $\mbQ$-factorial with klt singularities and $p$ is projective, by Lemma \ref{l-relpic}, $N^1(W/X)$ is generated by the $p$-exceptional divisors, say $E_1,\ldots ,E_r$. Then there exist real numbers $e_1,\ldots ,e_r\in \mathbb R$ such that $[q^*\omega_n+\sum e_iE_i]=0$ in $N^1(W/X)$. Thus there exists a  $(1, 1)$-form $\omega$ with local potentials on $X$ such that $[p^*\omega] =[q^*\omega_n+\sum e_iE_i]$. 
 Now since $S\cdot R<0$, there is a real number $r\in\mathbb{R}$ such that $(\omega+rS)\cdot R=0$.
 
	 Since $\psi _{n, *} S=0$ by Claim \ref{c-zero}, it follows that $\sum e_iE_i+rp^*S$ is a $q$-exceptional divisor and $[p^*(\omega+rS)-q^*\omega_n]=[\sum e_iE_i+rp^*S]$. Now we claim that $[p^*(\omega+rS)-q^*\omega_n]\equiv_{X'_n} 0$. Indeed, if $\gamma\subset W$ is a curve contracted by $q$, then $\alpha\cdot p_*\gamma=q^*\alpha ' _n\cdot \gamma=0$. 
  In particular, if $p_*\gamma\ne 0$, then $p_*\gamma$ is a curve contained in the extremal ray $R$, and hence $[\omega+r S]\cdot p_*\gamma=0$ and the claim follows. Then applying the negativity lemma we get $\sum e_iE_i+rp^*S=0$, and hence we have $[p^*(\omega+rS)]=[q^*\omega_n]$. Therefore
	 \[0<\omega_n\cdot C'_n=q^*\omega_n\cdot C=(\omega+rS)\cdot p_*C=0,\]
	  and this is a contradiction.
\end{proof}
Thus $X\to X_n'$ is a morphism of strongly $\Q$-factorial varieties with klt singularities which contracts a unique divisor $S$. By Lemma \ref{l-relpic}, it follows that $\rho (X/X_n')=1$ and hence this is the desired divisorial contraction.
\end{proof}



\section{The minimal model program for uniruled pairs}
In this section we will consider the minimal model program for non-pseudo-effective dlt compact K\"ahler $3$-fold pairs $(X,B)$.
Since $K_X+B$ is not pseudo-effective, neither is $K_X$ and hence the MRC fibration $X\dasharrow Z$ is non-trivial (see e.g. \cite[Introduction]{HP15}). {Let $\nu:X'\to X$ be any resolution, then since $(X,B)$ is dlt, the fibers of $\nu$ are rationally chain connected} (see Proposition \ref{p-rccres}) and hence $X'\dasharrow Z$ is the MRC fibration of $X'$. Recall that if $\dim Z \leq 1$, then by Lemma \ref{l-proj}, $X$ is projective and $H^2(X,\OO _X)=0$. Since this case is well understood, we will focus on the case where $\dim Z=2$. Note that by \cite[Remark 3.2]{HP15}, $Z$ is not uniruled and hence $K_Z$ is pseudo-effective. Moreover, from Definition \ref{def:lots-of-definitions}\ref{item:c-manifolds} it follows that $Z$ is in Fujiki's class $\mcC$. Then replacing $Z$ by a resolution of singularities we may assume that $Z$ is a smooth compact complex surface in Fujiki's class $\mcC$, and hence by \cite[Remark 1.1, page 236]{Fuj83}, $Z$ is K\"ahler.

\begin{definition}\label{def:normalized-kahler-form}
	Let $(X, B)$ be a log pair, where $X$ is a $\mbQ$-factorial compact K\"ahler $3$-fold. Suppose that the base of the MRC fibration $f:X\bir Z$ has dimension $2$. Let $X_z\cong\mbP^1$ be a general fiber of $f$. Then a modified K\"ahler class $\omega$ on $X$ is called $(K_X+B)$-\emph{normalized} if $(K_X+B+\omega)\cdot X_z=0$. Moreover, if $\omega$ is a K\"ahler class, then we call it a $(K_X+B)$-\emph{normalized K\"ahler class}.
 Note that as $\omega$ is a modified K\"ahler class, it is positive on general fibers $X_z$, and hence $(K_X+B)\cdot X_z<0$. 
\end{definition}

\begin{lemma}\label{lem:normalized-pair-is-psef}
	With the same notations and hypothesis as in the definition above, assume that $(X,B)$ is dlt and $\omega$ is a $(K_X+B)$-normalized modified K\"ahler class. Then $K_X+B+\omega$ is  pseudo-effective.
\end{lemma}

\begin{proof}
Since $\omega $ is a modified K\"ahler class, there is a bimeromorphic morphism from a compact K\"ahler manifold $\mu :X'\to X$ and a K\"ahler class $\omega '\in H^{1, 1}_{\rm BC}(X')$ represented by a K\"ahler form on $X'$ such that $\mu _* \omega '=\omega$. We may assume that there is an effective $\mu$-exceptional $\mbQ$-divisor $F$ such that $-F$ is $\mu$-ample. We may write $K_{X'}+B'=\mu ^*(K_X+B)+E$ where $E\>0, B'\geq 0, \mu_*B'=B$, and $E$ and $B'$ do not share any common component.
Since $\dim Z=2$ and $\dim \mu(\Ex(\mu))\leq 1$, it follows that $X'\to X$ is an isomorphism on a neighborhood of a general fiber $X_z$ of 
$X\bir Z$. Therefore $\omega '$ is $(K_{X'}+B')$-normalized. Note that if $K_{X'}+B'+\omega '$ is pseudo-effective, then so is $K_{X}+B+\omega =\mu _*(K_{X'}+B'+\omega ')$. Replacing $(X,B)$ by $(X',B')$ we may therefore assume that $X$ is smooth and $\omega $ is a K\"ahler class.

Since the pseudo-effective cone is closed, it is enough to show that $K_X+(1-\delta)B+(1+\ve)\omega$ is pseudo-effective for all $1\gg \ve \gg \delta >0$. Since $\omega$ is a $(K_X+B)$-normalized K\"ahler class, it follows easily that $(K_X+(1-\delta)B+(1+\ve)\omega)|_{X_z}$ is K\"ahler (or equivalently, it has positive degree). Let $Z'\to Z$ be a resolution of singularities of $Z$ and $\mu: X'\to X$ a log resolution of $(X,B)$ such that the induced meromorphic map $\varphi':X'\to Z'$ is a morphism. We may assume that $X_z=X'_z$ for general $z\in Z$.  Since $\omega$ is a K\"ahler class, there exists an effective $\mu$-exceptional divisor $F$ such that $\omega':=\mu^*\omega-F$ is a K\"ahler class.


Set
\[
		K_{X'}+B'_{\delta,\ve }: =\mu^*(K_X+(1-\delta)B)+\ve F.
	\]	

	Then we have
	\[
		\mu^*(K_{X}+(1-\delta)B +(1+\ve)\omega)=K_{X'}+B'_{\delta,\ve }+\omega '_\ve,
	\]
where $\omega '_\ve =\mu ^*\omega +\ve \omega '$ is a K\"ahler class, since $\omega$ is a K\"ahler class on $X$.

Therefore it is enough to show that $K_{X'}+(B'_{\delta,\ve })^{\geq 0} +\omega '_\ve$ is pseudo-effective.   Let $X'_{z'}$ be a general fiber of $\varphi':X'\to Z'$. Then $X'_{z'}\cong\mbP^1$ and $c_1((K_{X'}+(B'_{\delta,\ve })^{\geq 0} +\omega '_\ve)|_{X'_{z'}})$ is a K\"ahler class.
Thus by \cite[Theorem]{G16}, $K_{X'/Z'}+(B'_{\delta,\ve })^{\geq 0} +\omega '_\ve$ is pseudo-effective.  Now since $Z'$ is not uniruled, $K_{Z'}$ is pseudo-effective. Therefore $K_{X'}+(B'_{\delta,\ve })^{\geq 0} +\omega '_\ve$ is pseudo-effective as required.\\

\end{proof}

\begin{corollary}\label{cor:nef-and-big-normalized}
	Let $(X, B)$ be a klt pair, where $X$ is a $\mbQ$-factorial compact K\"ahler $3$-fold. Assume that $X$ is uniruled and the dimension of the base of the MRC fibration $X\bir Z$ is $2$. Let $\omega$ be a nef and big (1,1)-class on $X$ such that $(K_X+B+\omega)\cdot F=0$ for general fibers $F$ of $X\bir Z$. Then $K_X+B+\omega$ is pseudo-effective.   
\end{corollary}

\begin{proof}
	Since $(X, B)$ is klt and $\omega$ is nef and big, by Lemma \ref{lem:nef-and-big-to-modified-kahler} we can write $\omega=\Theta+F$, where $\Theta$ is a modified K\"ahler class and $F$ is an effective $\mbQ$-divisor such that $(X, B+F)$ klt. Thus replacing $B$ by $B+F$ and $\omega$ by $\Theta$ we may assume that $\omega$ is a modified K\"ahler class. Then the result follows from Lemma \ref{lem:normalized-pair-is-psef}.\\
		
\end{proof}

\begin{lemma}\label{lem:non-vanishing}
	Let $(X, B)$ be a dlt pair, where $X$ is a $\mbQ$-factorial compact K\"ahler $3$-fold. Suppose that $X$ is uniruled and the base of the MRC fibration $f:X\bir Z$ has dimension $2$. Let $F(\cong\mbP^1)$ be a general fiber of $f$. If $(K_X+B)\cdot F\>0$, then $K_X+B$ is pseudo-effective.
\end{lemma}

\begin{proof}
The proof is similar to  the proof of Lemma \ref{lem:normalized-pair-is-psef} and so we omit it.

\end{proof}

We will also need the following result.
\begin{lemma}\label{lem:null-locus}
	Let $(X, B)$ be a dlt pair, where $X$ is a $\Q$-factorial compact K\"ahler $3$-fold.  Let $\omega$ be a modified K\"ahler class on $X$ such that $\alpha=K_X+B+\omega$ is a  nef and big $(1, 1)$-class. Let $S\subset X$ be an irreducible surface such that $S\subset \Null(\alpha)$, i.e. $\alpha^2\cdot S=0$. Then $S$ is a Moishezon space and it is covered by $\alpha$-trivial curves. 
\end{lemma}

\begin{proof}
Let $\pi:S'\to S$ be the minimal resolution of $S$ dominating the normalization $\tilde{S}\to S$. By Lemma \ref{l-adj} we have 
\begin{equation}\label{eqn:minimal-resolution}
	K_{S'}+E=\pi^*((K_X+S)|_S),
\end{equation}	
where $E\>0$ is an effective $\Q$-divisor.

Note that since $S$ is a K\"ahler  surface, then so is  $S'$. We separate two cases based on the numerical dimension of $\pi^*(\alpha|_S)$.\\

\noindent
\textbf{Case I:} $\pi^*(\alpha|_S)\num 0$. In this case $-\pi^*((K_X+B)|_S)\equiv \pi^*(\omega|_S)$. Since $\omega$ is a modified K\"ahler class, $\omega|_S$ is a big $(1, 1)$-class by Corollary \ref{cor:modified-kahler}. Therefore $-\pi^*((K_X+B)|_S)$ is a big divisor on $S'$; in particular, $S'$ is a Moishezon space. Furthermore, since $S'$ is smooth and K\"ahler, it is projective by \cite{Moi66} (also see Theorem \ref{t-nam}). Consequently, $S'$ can be covered by a family of $\pi^*(\alpha|_S)$-trivial curves $\{C_t\}_{t\in T}$. Pushing forward these curves gives a covering family of $\alpha|_S$-trivial curves on $S$.\\

\noindent
\textbf{Case II:} $\pi^*(\alpha|_S)\not\num 0$. Since $\omega $ is modified K\"ahler, $\omega |_S$ is big. Since $\alpha$ is nef, 
\begin{equation}\label{eqn:hodge-inequality}
\omega\cdot\alpha\cdot S=(\omega|_S\cdot \alpha|_S)> 0.
\end{equation}

Note also that since $\alpha $ is big, $K_X+B+(1-\epsilon)\omega $ is also big for $0<\epsilon \ll 1$ (recall in fact that the big cone is open cf. \cite[\S 2.3]{Bou04}). We may write 
\[K_X+B+(1-\epsilon) \omega \equiv \sum_{j=1}^r \lambda _jS_j+P,\]  
where $\lambda _j>0$ for all $j$ and  $P$ is nef in codimension $1$.

Since $ \alpha ^2\cdot S=0$, it follows that 
\[ \left( \sum_{j=1}^r \lambda _jS_j+P\right)\cdot \alpha \cdot S =(K_X+B+(1-\epsilon) \omega)\cdot \alpha \cdot S =-\epsilon \omega \cdot S\cdot \alpha <0 .\] 
Since $\alpha $ is nef and $P|_S$ is pseudo-effective, we must have $S=S_j$ for some $1\leq j\leq r$ and $\alpha \cdot S^2<0$.

Let $B=aS+B'$ such that $0\<a\<1$ and $S$ is not contained in the support of $B'$. Then we have
\begin{equation}\label{eqn:canonical-intersection}
\begin{split}
	K_{S'}\cdot \pi^*(\alpha|_S)\<(K_{S'}+E)\cdot\pi^*(\alpha|_S) &=\pi^*((K_X+S)|_S)\cdot\pi^*(\alpha|_S)\\
	&=(K_X+S)|_S\cdot\alpha|_S\\
	&=(\alpha-aS-B'-\omega+S)\cdot\alpha\cdot S\\
	&=\alpha^2\cdot S+(1-a)\alpha\cdot S^2-B'|_S\cdot \alpha|_S-\omega\cdot\alpha\cdot S\\
	&<0.
\end{split}	
\end{equation}
Since $\pi^*(\alpha|_S)$ is nef, this shows that $K_{S'}$ is not pseudo-effective. Thus by Lemma \ref{lem:projectivity-of-kahler-surface}, $S'$ is projective, and hence $S$ is a Moishezon space. Furthermore, by Lemma \ref{lem:movable-curves-on-surfaces}, $S'$ is covered by $\pi^*(\alpha|_S)$-trivial curves. Pushing forward these curves on $S$ we see that $S$ is also covered $\alpha|_S$-trivial curves.
\end{proof}

\subsection{Cone Theorem}
The purpose of this section is to prove the following cone theorem which is a direct generalization of the results of  \cite{HP15}, \cite{HP16}, \cite{CHP16} and \cite{DO23}.
The techniques that we use are all inspired by these papers.
\begin{theorem}\label{t-cone-mrc}
	Let $(X,B)$ be a dlt pair, where $X$ is a $\mbQ$-factorial compact K\"ahler $3$-fold. Suppose that $X$ is uniruled and the base of the MRC fibration $X\dasharrow Z$ is a surface. Suppose that $(K_X+B)\cdot X_z<0$ for general $z\in Z$. Let $\omega$ be a $(K_X+B)$-normalized modified K\"ahler class. Then there exist a countable family of curves $\Gamma _i$ on $X$ and a positive number $d$ such that
$0<-(K_X+B+\omega )\cdot \Gamma _i\leq d$ and
\[ \NA (X)=\NA (X)_{(K_X+B+\omega ){\geq 0}}+\sum _{i\in I}\mathbb R ^+[\Gamma _i].   \]
\end{theorem}

\begin{proof}
 By Lemma \ref{lem:normalized-pair-is-psef}, $K_X+B+\omega$ is pseudo-effective and so by \cite{Bou04} (applied to a resolution of $X$ and then pushing forward) we have a divisorial Zariski decomposition 
\begin{equation} \label{e-le}
	 K_X+B+\omega \equiv \sum _{j=1}^r\lambda _jS_j+P,
\end{equation}
where the $S_j$'s are surfaces, $\lambda _j\in \mathbb R ^+$  for all $j$ and $P$ is a pseudo-effective class which is nef in codimension 1.

\begin{claim}\label{c-uni}
	 Let $S\subset X$ be a surface such that $(K_X+B+\omega)|_S$ is not pseudo-effective, then $S=S_j$ for some $1\leq j\leq r$, $S$ is Moishezon and any desingularization $\hat S \to S$ is a uniruled projective surface.
 \end{claim}
 
\begin{proof} Since $P$ is nef in codimension 1, then $P|_S$ is pseudo-effective. If $S\ne S_j$   for all $1\leq j\leq r$, then $(\sum _{j=1}^r\lambda _jS_j)|_S\geq 0$ and hence $(K_X+B+\omega)|_S$ is pseudo-effective, contradicting our assumptions. Thus, possibly reindexing, we may assume that $S=S_1$. Let $b={\rm mult}_S(B)$, then $0\leq  b\leq 1$. We then have
\begin{equation} \label{e-le1} 
K_X+S+B-bS+\omega +\frac {1-b}{\lambda _1}\left(\sum _{j=2}^r\lambda _j S_j+P\right)\equiv \left(1+\frac {1-b}{\lambda _1}\right)(K_X+B+\omega).\end{equation}
Since $(K_X+B+\omega)|_S$ is not pseudo-effective and $(B-bS+\omega +\frac {1-b}{\lambda _1}(\sum _{j=2}^r\lambda _j S_j+P))|_S$ is pseudo-effective, from the above equality
  it follows that $(K_X+S)|_S$  is not pseudo-effective.
Let $\pi :\hat S \to S$ be the minimal resolution of $S$ dominating its normalization $\tilde{S}$, then by Lemma \ref{l-adj}, there exists an effective divisor $E\geq 0$ on $\hat S$ such that $K_{\hat S}=\pi^*((K_X+S)|_S)-E$. But then it follows that $K_{\hat S}$ is not pseudo-effective, and hence $\hat{S}$ is projective by Lemma \ref{lem:projectivity-of-kahler-surface}, and $\hat S$ is uniruled. In particular, $S$ is Moishezon. Finally, observe that if $S'$ is any resolution of singularities of $S$, then it factors through the minimal resolution $\hat{S}$, and hence $S'$ is projective and uniruled.
\end{proof}

Next we establish a form of bend and break.
\begin{claim}\label{c-bb} There exists a number $d>0$ such that if $C\subset X$ is a curve with $-(K_X+B+\omega )\cdot C>d$, then $[C]=[C_1]+[C_2]$, where $C_1$ and $C_2$ are two non-zero integral effective 1-cycles.
\end{claim}

\begin{proof}  First using the arguments of \cite[Lemma 4.2]{CHP16} and passing to a dlt model as in Lemma \ref{lem:terminal-dlt-model}, we may assume that $(X,B)$ is dlt and $X$ has terminal singularities. The proof of this claim involves two main steps, in the first step we will construct four sets $\mcA, \mcB, \mcC$ and $\mcD$ of finitely many curves on $X$. These sets will determine the number $d>0$. Next we will show that if $C$ is a curve in $X$ such that $-(K_X+B+\omega )\cdot C>d$, then $\dim_C\Chow(X)>0$. We proceed with the constructions of the sets.

Let $\mathcal A$ be the set of all curves $C\subset X$ satisfying the following properties:
\begin{enumerate}
 \item $(K_X+B)\cdot C<0$, 
 \item $B\cdot C<0$ and 
 \item $C$ is contained in a horizontal (over $Z$) component $T$ of $B$.	
\end{enumerate}
We claim that $\mathcal A$ is a finite set. Indeed, if $T$ is a horizontal component of $B$, then $T$ is not uniruled, since the induced morphism $T\to Z$ is generically finite and $Z$ is not uniruled. Let $\pi:\hat{T}\to T$ be the minimal resolution of $T$. Then $\hat{T}$ is not uniruled, and hence $K_{\hat{T}}$ is pseudo-effective. Moreover, by the MMP and abundance for compact K\"ahler surfaces, there exists an effective $\mbQ$-divisor $D\>0$ on $\hat{T}$ such that $K_{\hat{T}}\sim_\mbQ D$. Now since the coefficients of $B$ are contained in the interval $(0, 1]$, there is a non-negative\ rational number $\lambda\>0$ such that $\mult_{T}(1+\lambda)B=1$. Then from our hypothesis it follows that $(K_X+(1+\lambda)B)\cdot C<0$. Then by Lemma \ref{l-adj} there is an effective $\mbQ$-divisor $E\>0$ on $\hat{T}$ such that 
\[
	\pi^*((K_X+(1+\lambda)B)|_T)=K_{\hat{T}}+E\sim_\mbQ D+E .
\]
Therefore, from the projection formula, it follows that $\pi_*(D+E)\cdot C<0$, i.e. $C$ is contained in the support of $\pi_*(D+E)$. In particular, the set $\mcA$ is finite.

Now let $S$ be a component of $B$ which is vertical over $Z$ and write
\begin{equation}\label{eqn:minimal-adjunction}
	 K_{\hat S}+E\sim _\Q \pi ^*((K_X+S)|_S),
\end{equation}
where $\pi:\hat{S}\to S$ is the minimal resolution of $S$.

If $\kappa (\hat S)\geq 0$, set $F:=\mathbf B (K_{\hat S})$, where $\mathbf{B}(\cdot)$ denotes the stable base locus. If $\kappa(\hat{S})=-\infty$, the we set $F:=0$.
Let $\mathcal B$ be the union of the curves $\pi _*(E+F)$ as  $S$ varies over all components of $B$ that do not dominate $Z$.

Let $\mathcal C$ be the finite set of curves $C\subset X$ which are contained in the singular locus of the support of $B+\sum S_j$.

 Let $\mathcal D$ be the set of all curves $C\subset X$ such that $\omega \cdot C \leq 0$. We claim that $\mcD$ is a finite set. Indeed, since $\omega $ is a modified K\"ahler class, by Lemma \ref{lem:modified-kahler}, there exists a projective bimeromorphic morphism $f:Y\to X$ from a compact K\"ahler manifold $Y$, a K\"ahler class $\omega _Y\in H^{1,1}_{\rm BC}(Y)$, and an effective $\mathbb R$-divisor $E\geq 0$ such that $-E$ is $f$-ample, ${\rm Supp}(E) = {\rm Ex}(f )$ and
$f ^*\omega  = \omega _Y + [E]$. If $C$ is a curve from the set $\mcD$ and $C'\subset Y$ is a curve in $Y$ such that $f(C')=C$, then $0\>\omega\cdot C=f^*\omega\cdot C'=(\omega_Y+E)\cdot C'$. This implies that $E\cdot C'<0$, since $\omega_Y\cdot C'>0$, as $\omega_Y$ is a K\"ahler class. In particular, $C'$ is contained in the support of $E$, since $E$ is effective. Therefore $C=f(C')$ is contained in $f(\Supp E)$. Since $E$ is $f$-exceptional and $\dim X=3$, it follows that $\mcD$ is a finite set.

Now we define
 \[d:={\rm max}\{4,\ -(K_X+B+\omega )\cdot C\;|\;C\subset X\ {\rm is\ a\ curve,\ and}\ C\in \mathcal A\cup \mathcal B\cup \mathcal C\cup \mathcal D\}.\]

\begin{claim}\label{clm:length-of-extremal-ray}
	For any curve $C\subset X$ such that $-(K_X+B+\omega )\cdot C>d$, we have $\dim _C{\rm Chow}(X)>0$.
\end{claim}

\begin{proof} By assumption $\omega\cdot C>0$ and so $-(K_X+B)\cdot C>d$. We will separate two cases depending on whether $B\cdot C<0$ or $B\cdot C\>0$. \\

\noindent
\textbf{Case I:} Suppose that $B\cdot C<0$. In this case there is a component $S$ of $B$ containing $C$. Moreover, we see that $C$ is not contained in any other component of $B$ or $\sum S_j$, since $C\not\in \mathcal C$, and also that $S$ is not horizontal over $Z$, since $C\not \in \mathcal A$. In particular, $S$ is vertical over $Z$. Now since $C\not \in \mathcal B$, $C$ is not contained in $\pi (E)$, where $E$ and $\pi :\hat S\to S$ as in \eqref{eqn:minimal-adjunction}. Let $\hat C\subset \hat S$ be the strict transform of $C$ and $\mult_S B=b$. 
Since $C\not \in \mathcal D$, then $\omega\cdot C>0$ and since $C$ is not contained in $\pi (E)$, then $E\cdot \hat  C\geq 0$. Then as $S\cdot C< 0$ and  $(B-bS)\cdot C\geq 0$, we have
	\begin{equation*}
	\begin{split}
		 K_{\hat S}\cdot \hat C &\leq (K_{\hat S}+E)\cdot \hat  C\\
		                        &=(K_X+S)\cdot C\\
		 						&=(K_X+B+(1-b)S-(B-bS))\cdot C\\
								&\<(K_X+B)\cdot C\\
								&<(K_X+B+\omega )\cdot C\\
								&< -d.
	\end{split}
	\end{equation*}
Now if $\kappa (\hat S)\geq 0$, then as $\hat C$ is not contained in $F=\mathbf B (K_{\hat S})$, it follows that $K_{\hat S}\cdot \hat C\geq 0$, which is a contradiction to the above inequality. Thus $\kappa (\hat S)<0$ and then from MMP and abundance for smooth compact K\"ahler surfaces it follows that $K_{\hat{S}}$ is not pseudo-effective. Hence $\hat{S}$ is projective by Lemma \ref{lem:projectivity-of-kahler-surface}. Now since $\hat S$ is a smooth and projective surface and $-K_{\hat{S}}\cdot \hat{C}> d\geq 4$, by \cite[Theorem II.1.15]{Kol96} we have $\dim _{\hat C}{\rm Chow}(\hat S)>0$, i.e. $\hat{C}$ deforms in $\hat{S}$, in particular its push-forward $C$ deforms in $S\subset X$, and hence $\dim _C \Chow (X)>0$.\\

\noindent
\textbf{Case II:} Suppose now that $B\cdot C\geq 0$. Since $C\not \in \mathcal D$, $\omega \cdot C>0$ and hence $K_X\cdot C<(K_X+B+\omega )\cdot C< -d\leq -4$. Since $X$ has $\mbQ$-factorial terminal singularities, it follows from \cite[Theorem 4.5]{HP16} that $C$ is not very rigid (see \cite[Definition 4.3]{HP16}). Let $m$ be the smallest positive integer such that $\dim _{mC}{\rm Chow}(X)>0$. Let $\Gamma\to T$ be the corresponding family. Replacing $\Gamma$ by an irreducible component which contains $C\subset X$ we may assume that $\Gamma$ is irreducible, and hence so is the locus covered by the family $(\Gamma_t)_{t\in T}$. Consequently, $\Gamma_t$ is irreducible for $t\in T$ very general. Then from the minimality of $m$ it also follows that this family has no fixed component. Now by \eqref{e-le} there is a unique  surface $S_j$ covered by the $\{\Gamma_{t}\}_{t\in T}$. Note that $(K_X+B+\omega)|_{S_j}$ is not pseudo-effective, since $(K_X+B+\omega)|_{S_j}\cdot \Gamma_t=m(K_X+B+\omega)\cdot C<0$. By Claim \ref{c-uni}, the minimal resolution $\hat{S}$ of $S=S_j$ is projective and uniruled.
Since $C\not \in \mathcal C$, we have that $C$  is not contained in $S_l$ for $l\ne j$ and hence $S_l\cdot C \geq 0$ for $l\ne j$. Since $P|_{S_j}$ is pseudo-effective, \[ P\cdot mC =P|_{S_j}\cdot mC=P|_{S_j}\cdot \Gamma_t
\geq 0. \] 
Using the same notation as in Case I and its proof we see that  $E\cdot m\hat{C}\>0$ and $\omega\cdot mC>0$, where $\hat{C}$ is the strict transform of $C$ under $\pi:\hat{S}\to S=S_j$. We also know that $S_j\cdot C<0$. Therefore from \eqref{e-le1} we have
\[ K_{\hat S }\cdot \hat C\leq (K_{\hat S}+E)\cdot \hat C = (K_X+S_j)\cdot C \leq \left(1+\frac {1-b}{\lambda _j}\right) (K_X+B+\omega )\cdot C< -d \leq -4. \]
By \cite[Theorem II.1.15]{Kol96}, we have $\dim _{\hat C}{\rm Chow}(\hat S)>0$, i.e. $\hat C$ deforms in $\hat S$. Thus by pushing forward $\hat{C}$ we have that $C$ deforms, i.e. $\dim_C \Chow(X)>0$.\\
\end{proof}

We will now prove the bend and break property, i.e. Claim \ref{c-bb}. If $C\subset X$ is a curve satisfying $-(K_X+B+\omega )\cdot C>d$, then by Claim \ref{clm:length-of-extremal-ray} and its proof it follows that, the curve $C$ deforms in a family $\{\Gamma_t\}_{t\in T}$ covering a unique uniruled surface $S$. Since $(K_X+B+\omega )\cdot C<0$, then $(K_X+B+\omega)|_S$ is not pseudo-effective. Since $P|_S$ is pseudo-effective, by \eqref{e-le} it follows easily that $S=S_j$.
 We also know that the curve $C$ is contained in $S$ but not in $S_\textsubscript{sing}$. Moreover, if $\pi:\hat{S}\to S$ is the minimal resolution of $S$, then we know from Claim \ref{c-uni} that $\hat{S}$ is a projective uniruled surface. Now from the proof of Claim \ref{clm:length-of-extremal-ray} we have $K_{\hat S}\cdot \hat C<-d\leq -4$. Thus by \cite[Lemma 5.5(b)]{HP16} there is an effective 1-cycle $\sum _{k=1}^mC_k$ with $m\geq 2$ such that $[\hat C]=[\sum _{k=1}^mC_k]$ and $K_{\hat S}\cdot C_i<0$ for $i=1,2$. Since $\pi:\hat S \to S$ is the minimal resolution, $K_{\hat S}$ is $\pi$-nef and hence $\pi _*C_j\ne 0$ for $j=1,2$. In particular, we also have  a decomposition $[ C]=[\sum _{k=1}^m\pi _* C_k]$ with at least two non-zero terms. This concludes the proof of Claim \ref{c-bb}.\\
\end{proof}
The theorem now follows from Claim \ref{c-bb} and the arguments in the proof of \cite[Theorem 4.1]{CHP16}.
\end{proof}

\begin{corollary}\label{cor:cone-finite}
	Let $(X,B)$ be a dlt pair, where $X$ is a $\mbQ$-factorial compact K\"ahler $3$-fold. Suppose that $X$ is uniruled and the base of the MRC fibration $X\dasharrow Z$ is a surface. Let $\omega$ be a modified K\"ahler class such that $K_X+B+\omega$ is pseudo-effective. If $(K_X+B+\omega )\cdot X_z >0$ for general $z\in Z$, then there exist finitely many curves $\{\Gamma _i\}_{i=1}^N$ such that 
	\[ \overline{\rm NA}(X)=\overline{\rm NA}(X)_{(K_X+B+\omega ){\geq 0}}+\sum_{i=1}^N \mathbb R ^+[\Gamma _i].\]
	\end{corollary}
 
\begin{proof} Pick $t=-(K_X+B)\cdot X_z/\omega \cdot X_z$, then $0<t<1$ and $(K_X+B+t\omega )\cdot X_z=0$. By Theorem \ref{t-cone-mrc}, there exists a countable family of curves $\Gamma _i$ with  $0<-(K_X+B+t\omega )\cdot \Gamma _i \leq d$ and 
 \[ \overline{\rm NA}(X)=\overline{\rm NA}(X)_{(K_X+B+t\omega )\geq 0}+\sum \mathbb R ^+[\Gamma _i].\]
 
 \begin{claim}\label{clm:finiteness-of-extremal-ray}
 	$\omega \cdot \Gamma _i> d/(1-t)$ for all but finitely many $i$'s.
 \end{claim}
 Assuming the claim for the time being we will complete the proof first. Using the claim we have 
 \[(K_X+B+\omega )\cdot \Gamma _i=(K_X+B+t\omega )\cdot \Gamma _i+(1-t)\omega \cdot \Gamma _i> 0\] for all but finitely many $i$'s. This concludes the proof.\\
 
\begin{proof}[Proof of Claim \ref{clm:finiteness-of-extremal-ray}]
Since $\omega$ is a modified K\"ahler class, by Lemma \ref{lem:modified-kahler}  there exist a projective bimeromorphic morphism $\nu :X'\to X$ from a K\"ahler manifold $X'$, a K\"ahler class $\omega'\in H^{1, 1}_{\rm BC}(X')$ and an effective $\nu$-exceptional divisor $F\>0$ such that $\nu^*\omega=\omega'+[F]$. Note that from the negativity lemma it follows that $\Supp(F)=\Ex(\nu)$.
Now let $\Gamma _i'$ be the strict transform of the curve $\Gamma_i$ which is not contained in $\nu(F)$. Note that since $\dim \nu(F)\<1$, there are only finitely many  $\Gamma_i$'s contained in $\nu(F)$.

We will assume by contradiction that the claim is false, i.e. $\omega \cdot \Gamma _i\< d/(1-t)$ for infinitely many indices $i$. Let $\Lambda$ be the set of indices for all such curves $\Gamma_i$. Without loss of generality we may assume that $\Gamma_i$ is not contained in $\Supp(\nu(F))$ for any $i\in\Lambda$. Observe that the $\Gamma_i$'s for $i\in\Lambda$ belong to distinct equivalence classes in $\NA(X)$; in particular, the strict transforms $\Gamma'_i$ also belong to distinct equivalence classes in $\NA(X')$ for all $i\in\Lambda$. Now by the projection formula we have 
\[
	\omega'\cdot\Gamma'_i=(\nu^*\omega-F)\cdot \Gamma'_i\<\omega\cdot\Gamma_i\<d/(1-t)\quad \mbox{ for all } i\in\Lambda.
\]
Since $\omega' $ is a K\"ahler class, by \cite[Lemma 4.4]{Toma16} the curves $\Gamma'_i$ belong to a bounded family for all $i\in\Lambda$. In particular, the $\Gamma'_i$'s belong to finitely many distinct equivalence classes in $\NA(X')$, this is a contradiction.
\end{proof}

\end{proof}


\begin{remark}\label{rmk:nef-big-perturbation}
 Observe that in the settings of Corollary \ref{cor:cone-finite} if $(X, B)$ is klt and $\omega$ is only a nef and big class such that $(K_X+B+\omega)\cdot X_z>0$ for general $z\in Z$ (this last condition is satisfied for example if $K_X+B+\omega$ itself is big), then the same finiteness conclusion holds. Indeed, it follows from Corollary \ref{cor:cone-finite} using Lemma \ref{lem:nef-and-big-to-modified-kahler}. 
\end{remark}


\begin{remark}\label{rmk:finite-perturbation}
We note that a similar finiteness result as in Corollary \ref{cor:cone-finite} also holds for pseudo-effective pairs. More specifically, if $(X, B)$ is a $\mbQ$-factorial compact K\"ahler $3$-fold dlt (resp. klt) pair and $K_X+B$ is pseudo-effective, then for any modified K\"ahler (resp. nef and big) class $\omega$, there are only finitely many $(K_X+B+\omega)$-negative extremal rays of $\NA(X)$. Indeed, by Theorem \ref{t-cone} there are countably many $(K_X+B)$-negative extremal rays generated by the curves $\{\Gamma_i\}_{i\in I}$ and a positive rational number $d>0$ such that $0<-(K_X+B)\cdot \Gamma_i\<d$ for all $i\in I$. First we will deal with the modified K\"ahler case. Since $\omega$ is modified K\"ahler, as in the proof of Claim \ref{clm:finiteness-of-extremal-ray} (using the same notations) we have $\nu^*\omega=\omega'+[F]$. If $(K_X+B+\omega)\cdot \Gamma_i<0$ and $\Gamma_i$ is not contained in $\nu(\Supp(F))$, let $\Gamma'_i$ be the strict transform of $\Gamma_i$. Then we have $F\cdot \Gamma'_i\>0$ and $\omega'\cdot\Gamma'_i\<(\omega'+F)\cdot \Gamma'_i=\omega\cdot\Gamma_i\<d$. Since $\omega'$ is K\"ahler, this implies that there are only finitely many such curves $\Gamma_i$. Furthermore, since $\nu(\Supp(F))$ contains only finitely many curves (as $\dim X=3$), our claim follows.

Now if $(X, B)$ is klt and $\omega$ is a nef and big class, then using Lemma \ref{lem:nef-and-big-to-modified-kahler} we may assume that $\omega$ is a modified K\"ahler class, and thus we are done by the previous case.

\end{remark}

Next we prove a technical result which will be used in the existence of small contractions (see Theorem \ref{c-cont}) and also in Section \ref{sec:applications}.

\begin{proposition}\label{pro:non-null-locus}
 Let $(X, B)$ be a dlt pair, where $X$ is a $\mbQ$-factorial compact K\"ahler $3$-fold. Assume that $X$ is uniruled, $K_X+B$ is not pseudo-effective, and the base of the MRC fibration $X\bir Z$ has dimension $2$. Let $R=\mbR^+[\Gamma_i]$ be an extremal ray of $\NA(X)$ with a nef supporting class $\alpha$. Assume that there is a nef and big class $\eta$ such that $K_X+B+\eta$ is pseudo-effective and $(K_X+B+\eta)\cdot R<0$. If $R$ is small and $S\subset X$ is an irreducible surface, then the following holds:
 \[
 	 \alpha^2\cdot S>0.
 \]
\end{proposition}

\begin{proof}
First note that by a standard technique using  Theorem \ref{t-cone-mrc} and Corollary \ref{cor:cone-finite} we can write $\alpha-(K_X+B+\eta)=\omega$, for some K\"ahler class $\omega$, i.e. $\alpha=K_X+B+\eta+\omega$. Assume by way of contradiction that $\alpha^2\cdot S=(\alpha|_S)^2=0$. First assume that $\alpha|_S=0$. Then  we have $-(K_X+B)|_S=(\eta+\omega)|_S$. Thus $-(K_X+B)|_S$ is an ample divisor on $S$, in particular, $S$ is projective. So we can cover $S$ by a family of curves. But since $\alpha|_S=0$, all these curves are contained in $R$, this is a contradiction, since $R$ is small.

Next assume that $\alpha|_S\neq 0$ but $(\alpha|_S)^2=0$. Then we have 	
\[
	0=\alpha^2\cdot S=(K_X+B)\cdot\alpha\cdot S+(\eta+\omega)\cdot\alpha\cdot S
\]
and
 \[
	(\eta+\omega)\cdot\alpha\cdot S=(\eta+\omega)|_S\cdot\alpha|_S>0, \quad\mbox{since } \alpha|_S \mbox{ is a non-zero nef class.}
\]	
Therefore we have
\begin{equation}\label{eqn:adjoint-negative}
	(K_X+B)\cdot\alpha\cdot S<0.
\end{equation}	
By a similar computation we also have
\begin{equation}\label{eqn:adjoint-negative2}
	(K_X+B+\eta)\cdot\alpha\cdot S<0.
\end{equation}	

In particular, $(K_X+B)|_S$ (resp. $(K_X+B+\eta)|_S$) is not pseudo-effective, since $\alpha|_S$ is a nef class.

Let $\pi:\hat{S}\to S$ be the minimal resolution of $S$ dominating the normalization of $S$. We make the following claim.
\begin{claim}\label{clm:surface-uniruledness}
	There is an effective $\mbQ$-divisor $E\>0$ on $\hat{S}$ such that \[(K_{\hat{S}}+E)\cdot \pi^*(\alpha|_S)<0.\]
\end{claim}	
	Note that once we have this claim, the rest of the proof works exactly as in the proof of \cite[Proposition 4.4]{CHP16}, since the only thing used there is this property of the nef class $\pi^*(\alpha|_S)$, and it does not depend on whether $K_X+B$ is pseudo-effective or not. In the following we will prove our claim. \\
	
	\begin{proof}[Proof of Claim \ref{clm:surface-uniruledness}]
		We will split the proof into two cases.\\
		
\noindent
\textbf{Case I:} Assume that $B\cdot\alpha\cdot S=B|_S\cdot\alpha|_S<0$. Then $S$ is contained in the support of $B$, since $\alpha|_S$ is nef. Then there exists a real number $\lambda\>0$ such that the coefficient of $S$ in $(1+\lambda)B$ is $1$. Then using \eqref{eqn:adjoint-negative} we have
\[
	(K_X+(1+\lambda)B)\cdot\alpha\cdot S\<(K_X+B)\cdot\alpha\cdot S<0.
\] 

Thus by adjunction (see Lemma \ref{l-adj}), there exists an effective $\mbQ$-divisor $E$ on $\hat{S}$ such that
\[
	(K_{\hat{S}}+E)\cdot \pi^*(\alpha|_S)=\pi^*((K_X+(1+\lambda)B)|_S)\cdot \pi^*(\alpha|_S)=(K_X+(1+\lambda)B)\cdot\alpha\cdot S<0.
\]
This proves our claim in this case.\\

\noindent
\textbf{Case II:}  Assume that $B\cdot\alpha\cdot S\>0$. Then we have 
\begin{equation}\label{eqn:K_X-negative}
	K_X\cdot\alpha\cdot S\<(K_X+B)\cdot\alpha\cdot S<0.
\end{equation} 
Now consider the Zariski decomposition of $K_X+B+\eta$:
\[
	K_X+B+\eta\num \sum_{j=1}^r\lambda_j S_j+P,
\]	
where $\lambda_j\>0$ for all $j$ and $P$ is nef in codimension $1$.
	
We claim that $S=S_j$ for some $j$. If not, then $\sum\lambda_j S_j|_S+P|_S$ is pseudo-effective. But then from \eqref{eqn:adjoint-negative2} we have
\[
	0>(K_X+B+\eta)\cdot\alpha\cdot S=\left(\left(\sum\lambda_jS_j|_S\right)\cdot\alpha|_S+P|_S\cdot\alpha|_S\right)\>0, \mbox{ a contradiction.}
\]
 Next we claim that $\alpha\cdot S^2<0$. To see this first assume that $S=S_1$. Then we have
 \[
 	0>(K_X+B+\eta)\cdot\alpha\cdot S=\left(\sum_{j=2}^r\lambda_jS_j|_S\cdot\alpha|_S\right)+\lambda_1 \alpha\cdot S^2+P|_S\cdot\alpha|_S.
 \]
 Since the first and the last term on the right hand side are non-negative, it follows that $\alpha\cdot S^2<0$. Combining this with \eqref{eqn:K_X-negative} we have $(K_X+S)\cdot\alpha\cdot S<0$. Then by adjunction (see Lemma \ref{l-adj}) there exists an effective $\mbQ$-divisor $E$ on $\hat{S}$ such that
\[
	(K_{\hat{S}}+E)\cdot\pi^*(\alpha|_S)=\pi^*((K_X+S)|_S)\cdot\pi^*(\alpha|_S)=(K_X+S)\cdot\alpha\cdot S<0.
\]
This completes the proof of the claim.

\end{proof}

\end{proof}

\subsection{Existence of divisorial contractions and flips}


\begin{theorem}\label{c-cont} Let $(X,B)$ be a dlt pair, where $X$ is a strongly $\mbQ$-factorial compact K\"ahler $3$-fold. Suppose that $X$ is uniruled and the dimension of the base of the MRC fibration $X\bir Z$ is $2$, and $(K_X+B)\cdot F<0$ for a general fiber $F$ of $X\bir Z$. Let $\omega$ be a  K\"ahler class (or more generally a nef and big class) such that $K_X+B+\omega$ is pseudo-effective, and $R$ a $(K_X+B+\omega)$-negative extremal ray. 
Then the contraction $c_R:X\to Y$ of the ray $R$ exists. Moreover, if $c_R$ is a small contraction, then the flip of $c_R$ exists.
\end{theorem}
Note that we will only use the strongly $\Q$-factorial hypothesis for the case of contractions of divisors to curves (i.e. $n(\alpha )=1$).

\begin{proof} 


Since $R$ is $(K_X+B+\omega)$-negative, for any K\"ahler class $\omega'$, there is a number $0<\ve\ll 1$ such that $R$ is also $(K_X+B+\omega+\ve\omega')$-negative. Since $\omega$ is nef, $\omega+\ve\omega'$ is K\"ahler, and hence $\omega+\ve\omega'+\delta B$ is K\"ahler for $0<\delta\ll\ve$. Thus replacing $B$ by $(1-\delta)B$ and $\omega$ by $\omega+\ve\omega'+\delta B$ we may assume that $(X, B)$ is klt and $\omega$ is a K\"ahler class. Note that $K_X+B+\omega$ is pseudo-effective and $(K_X+B+\omega)\cdot R<0$.
By a similar argument as in the proof of \cite[Proposition 4.3]{CHP16} (using Theorem \ref{t-cone-mrc} and Corollary \ref{cor:cone-finite}) we may assume that the extremal ray $R$ is cut out by a  nef $(1, 1)$-class $\alpha$. Re-scaling $\alpha$ if necessary, we see that $\eta:=\alpha -(K_X+B+\omega)$ is positive on $\overline{\rm NA}(X)\setminus \{0\}$, and hence $\eta$ is a K\"ahler class by \cite[Corollary 3.16]{HP16}. Thus it follows that $\alpha $ is a nef and big class, since $K_X+B+\omega$ is pseudo-effective.

Suppose that $R$ is small, then the contraction $c_R$ exists by \cite[Theorem 4.2]{CHP16}, see also \cite[Theorem 4.12]{DH23}. Note that even though \cite[Theorem 4.2]{CHP16} is stated for non-uniruled varieties, its proof works in our case if we replace \cite[Proposition 4.4]{CHP16} by Proposition \ref{pro:non-null-locus}. Since $\omega \cdot R> 0$ (as $\omega$ is K\"ahler), we have $(K_X+B)\cdot R<0$ and so this is also a $(K_X+B)$-flipping contraction. The existence of the flip then follows from Theorem \ref{t-flip}.

If $R$ is of divisorial type, then the corresponding irreducible divisor $S$ is covered by and contains all the curves $C\subset X$ such that $[C]\in R$. 
Recall from Notation \ref{not:nef-reduction} that $\nu :\tilde S \to S$ is the normalization and $f:\tilde S \to T$ is an $\alpha$-trivial fibration with $\dim T=n(\alpha )\in \{0,1\}$. 

If $n(\alpha)=0$, then the existence of $c_R$ follows by the same arguments as in \cite[Corollary 7.7]{HP16}, also see \cite[Corollary 3.4]{DH23}. Note in fact that in this case it suffices to show that if $mS$ is Cartier, then $(-mS)|_S$ is ample. This in turn is a consequence of the fact that $-mS\cdot R>0$.

If $n(\alpha)=1$, the existence of the corresponding divisorial contraction will follow from the proof of Theorem \ref{t-div1sqf}. 
 Note that, the only place in the proof of Theorem \ref{t-div1sqf} where we used that $K_X+B$ is pseudo-effective, is in the proof of Claim \ref{c-ammp}.
 We will now explain how to run the corresponding minimal model program in the non-pseudo-effective case. We will adopt the notation of Claim \ref{c-ammp}.

In the first part of the proof of this claim we must show that we can perform the required flips and divisorial contractions $X'_i\dasharrow X'_{i+1}$. 
Consider an $\alpha _i'$-trivial, $(K_{X'_i}+\Delta _i')$-negative extremal ray $R_i$ such that $P\cdot R_i<0$ for some component $P$ of $\lfloor \Delta _i'\rfloor$. By adjunction, $K_P+\Delta _P=(K_{X'_i}+\Delta _i')|_P$ induces a dlt surface pair $(P,\Delta _P)$ such that $\alpha _P:=\alpha '_i|_P$ is nef. Note however that, as $P\cdot R_i<0$, all curves $[C]\in R_i$ are contained in $P$ and hence $\alpha _P$ is not K\"ahler.
Consider the induced linear map $\iota:\overline {\rm NA}(P)\to \overline {\rm NA}(X'_i)$ and let $F=\iota ^{-1}(R_i)$, then $F$ is a $(K_P+\Delta _P)$-negative extremal face of $ \overline {\rm NA}(P)$. By the surface cone theorem and contraction theorem (see \cite[Theorem 1.31]{DO23} and \cite[Lem. 2.30 and Cor. 2.32(2)]{DHY23}), there is a corresponding contraction morphism $\gamma :P\to W$. By \cite[Proposition 7.4]{HP16} (see also \cite[Theorem 4.7]{DH23}), there exists a bimeromorphic morphism $g : X'_i\to Z$ such that $g|_P=\gamma$ and $g|_{X_i'\setminus P}$ is an isomorphism onto $Z\setminus W$. If $g$ is a flipping contraction, then the flip exists by Theorem \ref{t-flip} and termination follows from Theorem \ref{t-term}.

 We may therefore assume that we have constructed $X'_n$ so that there are no $\alpha _n'$-trivial $(K_{X'_n}+\Delta _n')$-negative extremal rays and we must conclude that there are no $\alpha _n'$-trivial, $(K_{X'_n}+\Delta _n')$-negative curves. This follows from the cone theorem proved in Theorem \ref{t-cone-mrc}  and Corollary \ref{cor:cone-finite}
 as we will now explain. Recall that  $\omega _n'$ is modified K\"ahler and $\alpha _n'=K_{X_n'}+\Delta _n'+\omega _n'$ is nef and big. Choose $0<\ve\ll 1$ so that $K_{X'_n}+\Delta'_n+(1-\ve)\omega'_n$ is big, in particular, $(K_{X_n'}+\Delta _n'+(1-\ve) \omega _n')\cdot X_z> 0$.
 Thus by Corollary \ref{cor:cone-finite} it follows that there exist finitely many curves $\{\Gamma _i\}_{i=1}^N$ such that
\[ \NA (X_n')=\NA (X'_n)_{(K_{X_n'}+\Delta _n'+(1-\ve)\omega_n' ){\geq 0}}+\sum _{i=1}^N\mathbb R ^+[\Gamma _i].   \]
 Suppose that $\Sigma $ is an $\alpha _n'$-trivial $(K_{X_n'}+\Delta _n')$-negative curve, then  $\Sigma = \Sigma _1+\Sigma _2$, where $\Sigma _1\in \NA (X)_{(K_{X_n'}+\Delta _n'+(1-\ve)\omega_n' ){\geq 0}}$ and $\Sigma _2\in \sum _{i=1}^N\mathbb R ^+[\Gamma _i]$. 
 Since $\alpha _n'\cdot \Sigma =0$ and $\alpha _n'$ is nef, we have $\alpha _n'\cdot \Sigma_1 =\alpha _n'\cdot \Sigma_2 =0$. Observe that the rays $\mbR^+[\Gamma_i], i=1,2,\ldots, n$ are $(K_{X'_n}+\Delta'_n)$-negative. Then by our assumption above $\alpha _n'\cdot \Gamma _i>0$ for all $i\in \{1,2,\ldots, n\}$, and hence $\Sigma _2=0$, in particular, $\Sigma \in \NA (X)_{(K_{X_n'}+\Delta _n'+(1-\ve)\omega_n'){\geq 0}}$. Since $\Sigma$ is $\alpha'_n$-trivial, it follows that $(K_{X'_n}+\Delta'_n)\cdot\Sigma\geq 0$, this is a contradiction to the fact that $(K_{X_n'}+\Delta _n')\cdot \Sigma <0$. This proves our claim.

\end{proof}

\section{Towards the Existence of Mori fiber spaces}
In this section we will prove the key technical results needed for the proof of
Theorem \ref{t-main2} in the next section. We start with the following log version of  \cite[Theorem 1.3]{HP15}.

\begin{theorem}\label{thm:partial-mfs}
	Let $(X, B)$ be a dlt pair, where $X$ is a strongly-$\mbQ$-factorial compact K\"ahler $3$-fold. Suppose that $X$ is uniruled and the base of the MRC fibration $f:X\bir Z$ has dimension $2$ and $(K_X+B)\cdot F<0$, where $F$ is a general fiber of $f$. Then there is a bimeromorphic map $\phi:X\bir X'$ given by a sequence of $(K_X+B)$-flips and divisorial contractions 
	such that for any $(K_{X'}+B')$-normalized K\"ahler class $\omega'$ on $X'$ (see Definition \ref{def:normalized-kahler-form}), the adjoint class $K_{X'}+B'+\omega'$ is nef, where $B'=\phi _*B$.
\end{theorem}
\begin{proof} Suppose that there is a $(K_{X}+B)$-normalized K\"ahler class $\omega$ on $X$, such that the adjoint class $K_{X}+B+\omega$ is not nef. Replacing $B$ by $(1-\epsilon )B$ and $\omega $ by $\omega +\epsilon B$ for $0<\epsilon \ll 1$, we may assume that $(X,B)$ is klt.
By Theorem \ref{t-cone-mrc}, there is a  $(K_X+B+\omega)$-negative extremal ray $R$. Note that $K_X+B+\omega$ is pseudo-effective by Lemma \ref{lem:normalized-pair-is-psef}, and thus by Theorem \ref{c-cont} we may flip or contract $R$. Repeating this procedure, since every step is $(K_X+B)$-negative,  by Theorem \ref{t-term} we obtain the required bimeromorphic map $\phi:X\bir X'$.
\end{proof}

Following the same ideas as in \cite{HP15} we prove a log version of \cite[Theorem 1.4]{HP15} below.
\begin{theorem}\label{thm:log-contraction-terminal-pair}
	Let $(X, B)$ be a terminal pair, where $X$ is a $\mbQ$-factorial compact K\"ahler $3$-fold. Suppose that $X$ is uniruled and  the base of the MRC-fibration $f:X\bir Z$ has dimension $2$.  Let $\omega$ be a nef and big class on $X$ such that $K_X+B+\omega$ is nef and $(K_X+B+\omega)\cdot F=0$, where $F$ is a general fiber of $f$. Then there exists a proper surjective morphism with connected fibers 	$\varphi:X\to S$ onto a normal compact K\"ahler surface $S$ such that $K_X+B+\omega$ is $\varphi$-trivial, i.e. $(K_X+B+\omega)|_{X_s}\num 0$ for all $s\in S$.
\end{theorem}

\begin{proof}

We will closely follow the proof of \cite[Theorem 1.4]{HP15} here. We will consider  the nef dimension $\nu_{\rm nef}(K_X+B+\omega)$. First note that, since a dense open subset of $X$ is covered by $(K_X+B+\omega)$-trivial curves, the nef dimension $\nu_{\rm nef}(K_X+B+\omega) \leq 2$. We claim that $\nu_{\rm nef}(K_X+B+\omega)=2$. Indeed, if the nef dimension is 0, then $-(K_X+B)\equiv \omega$, and thus $-(K_X+B)$ is nef and big. In particular, $X$ is a Moishezon space. Since $X$ is also K\"ahler and has rational singularities, by Theorem \ref{t-nam} $X$ is projective. Then by \cite[Theorem 1]{Zha06} $X$ is rationally connected. This contradicts the fact that the base of the MRC fibration of $X$ has dimension $2$. 
 If  the nef dimension is 1, then there is a proper surjective morphism $f:X\to C$ to a smooth projective curve $C$ such that $(K_X+B+\omega)|_F\num 0$ for all general fibers $F$ of $f$ (see \cite[2.4.4]{BCE02}). Now since $X$ has terminal singularities, a general fiber $F$ is smooth, and thus by adjunction we have $(K_X+B)|_F=K_F+B|_F$ such that $(F, B|_F\>0)$ has klt singularities. Moreover, from Lemma \ref{lem:restriction-of-nef-and-big} we see that $-(K_F+B|_F)\num \omega|_F$ is nef and big. Therefore by \cite[Theorem 1]{Zha06} the general fiber of $f$ is rationally connected. This contradicts the fact that the base of the MRC fibration of $X$ has dimension $2$. Thus $\nu_{\rm nef}(K_X+B+\omega)=2$ as claimed.

There is an induced rational map $Z\dasharrow {\rm Chow}(X)$ sending the general points of $Z$ to the points corresponding to the general fibers of the MRC fibration $X\bir Z $.
Replace $Z$ by an appropriate resolution so that $Z\to \Chow(X)$ is a morphism, and let $\Gamma \to Z$ be the normalization of the pull-back of the universal family over $ {\rm Chow}(X)$. Then $\Gamma$ is a normal compact complex $3$-fold with equi-dimensional fibers of dimension $1$ over $Z$. Let $p:\Gamma\to X$ and $q:\Gamma\to Z$ be the projection maps. Note that as observed in the proof of \cite[Theorem 1.4]{HP15}, $\Gamma$ is in Fujiki's class $\mathcal C$ and hence by \cite[Theorem 3]{Var86}, $Z$ is in Fujiki's  class $\mathcal C$. Moreover, since $Z$ is a smooth compact surface, by \cite[Remark 1.1, page 236]{Fuj83} $Z$ is K\"ahler.

  We claim that there is a nef and big $(1, 1)$-class $\alpha$ on $Z$ such that
\begin{equation}\label{eqn:log-omega-pullback}
	p^*(K_X+B+\omega)=q^*\alpha.
\end{equation} 

We split our proof into 4 steps below as in the proof of \cite[Theorem 1.4]{HP15}.\\

\noindent
\emph{Step $1$}: This step shows the existence of a  nef $(1, 1)$-class $\alpha$ on $Z$ satisfying \eqref{eqn:log-omega-pullback}. The proof of this step is exactly same as the proof of \emph{Step $1$} of \cite[Theorem 1.4]{HP15}. However, for the convenience of the reader, we include the details here. Since the general fiber of $q$ is a rational curve, by \cite[II, 2.8.6.2]{Kol96}, it follows that $R^1q_*\mathcal O _\Gamma =0$. Consider the exponential sequence 
\[ 0\to \mathbb Z\to \mathcal O _{\Gamma}\to  \mathcal O _{\Gamma}^*\to 0.\]
Since $\mathcal O _Z=q_*\mathcal O _{\Gamma}\to  q_* \mathcal O _{\Gamma}^*=\mathcal O _Z^*$ is surjective, we have
 $R^1q_*\mathbb Z=0$. By the universal coefficient theorem $R^1q_* \mathbb R=0$. We now consider the Leray spectral sequence $E_2^{i,j}=H^i(Z,R^jq_* \mathbb R)$ degenerating to $H^*(\Gamma ,\mathbb R)$. Since $E_2^{i,1}=0$ for all $i$, it follows that there is an exact sequence
\[ 0\to H^2(Z,\mathbb R)\to H^2(\Gamma , \mathbb R)\to H^0(Z, R^2q_* \mathbb R). \]
To show that $[p^*(K_X+B+\omega)]=[q^*\alpha ]$ it then suffices to show that if $s\in E^{0,2}_2=H^0(Z,R^2q_* \mathbb R)$ is the section defined by $s(z)=[p^*(K_X+B+\omega )|_{\Gamma _z}]\in H^2(\Gamma _z,\mathbb R)$, then $s=0$. Since $R^2q_* \mathbb R$ is constructible (see \cite[Proposition 3.5]{ES10}), then the section $s$ vanishes if and only if it vanishes pointwisely.
Note that since $\omega $ is a $(K_X+B)$-normalized K\"ahler class, the claim clearly holds for general $z\in Z$. Since $\Gamma \to Z$ has equi-dimensional fibers and $p^*(K_X+B+\omega )$ is nef, it follows that $p^*(K_X+B+\omega )\equiv 0$ on every irreducible component of every fiber. Note that since $q^*\alpha$ is nef, $\alpha$ is also nef by \cite[Theorem 1]{Pau98}.\\

\noindent
\emph{Step $2$}: With the notation as in the \emph{Step $2$} of the proof of \cite[Theorem 1.4]{HP15} let $\mu:\hat{X}\to \Gamma$ be a resolution of singularities of $\Gamma$ such that  the exceptional locus of $\hat{p}=p\circ\mu:\hat{X}\to X$ has pure codimension $1$. Set $\hat{q}=q\circ\mu$. Then we have
\begin{equation}\label{eqn:restriction-on-the-exceptional-locus}
	\hat{q}^*\alpha\cdot\hat{D}=0\quad\mbox{in } N_1(\hat X)
\end{equation}
for every irreducible component $\hat{D}$ of the exceptional locus of $\hat{p}$. Since $X$ has $\mbQ$-factorial terminal singularities, the proof of this step is also the same proof of \emph{Step $2$} of \cite[Theorem 1.4]{HP15}, so we skip the details here. 
 
Next we claim that $\alpha$ is a big class on $Z$, i.e. $\alpha^2>0$. This is Step 3 below.\\

\noindent
\emph{Step $3$}: The proof of this step is almost identical to the proof of \emph{Step $3$} of \cite[Theorem 1.4]{HP15}. We include the details here for the convenience of the reader.
{ Note that using Lemma \ref{lem:nef-and-big-to-modified-kahler} from now on we may assume that $\omega$ is a modified K\"ahler class; observe that we loose the nefness of $\omega$ here, but it is not needed in the rest of the proof. Then by Lemma \ref{lem:modified-kahler}, replacing $\hat{X}$ by a higher resolution of $\Gamma$ if necessary, we may assume that there is an effective $\hat{p}$-exceptional $\mbR$-divisor $F\>0$ on $\hat{X}$ such that $\hat{p}^*\omega-F$ is a K\"ahler class.}
Since the K\"ahler cone is open, there is a K\"ahler class $\eta_Z$ on $Z$ such that $\hat{p}^*\omega-F-\hat{q}^*\eta_Z$ is a K\"ahler class on $\hat{X}$. Then by a similar argument as in the proof of Lemma \ref{lem:normalized-pair-is-psef}  it follows that
\begin{equation}\label{eqn:relative-canonicla-is-psef}
	K_{\hat{X}/Z}+\hat{B}+\hat{p}^*\omega-F-\hat{q}^*\eta_Z\quad\mbox{is pseudo-effective},
\end{equation}
where $ K_{\hat{X}}+\hat{B}=\hat{p}^*(K_X+B)+E$, $\hat{B}\>0, E\>0, \hat{p}_*\hat{B}=B$, and $\hat{B}$ and $E$ do not share any common component.

Now we have
\begin{equation}\label{eqn:log-psef}
	\hat{p}^*(K_X+B+\omega)=(K_{\hat{X}/Z}+\hat{B}+\hat{p}^*\omega-F-\hat{q}^*\eta_Z)-E+F+\hat{q}^*K_Z+\hat{q}^*\eta_Z.
\end{equation}
We will show that $\hat{q}^*\alpha^2=\hat{q}^*\alpha\cdot\hat{p}^*(K_X+B+\omega)$ is a non-zero class in $\NA(\hat{X})$. To that end first observe that, since $\alpha$ is nef and $K_{\hat{X}/Z}+\hat{B}+\hat{p}^*\omega-F-\hat{q}^*\eta_Z$ is pseudo-effective, the intersection product $\hat{q}^*\alpha\cdot(K_{\hat{X}/Z}+\hat{B}+\hat{p}^*\omega-F-\hat{q}^*\eta_Z)$ is an element of $\NA(\hat{X})$. Since $E$ and $F$ are both $\hat{p}$-exceptional, from \eqref{eqn:restriction-on-the-exceptional-locus} it follows that $\hat{q}^*\alpha\cdot (-E+F)=0$ in $\NA(\hat{X})$. Now since the surface $Z$ is not uniruled, by classification $K_Z$ is pseudo-effective; in particular, $\hat{q}^*\alpha\cdot\hat{q}^*K_Z$ is an element of $\NA(\hat{X})$. Now recall that $\alpha\neq 0$, since $K_X+B+\omega\neq 0$. Since $\eta_Z$ is a K\"ahler class and $\alpha$ is a non-zero nef class, the Hodge index theorem yields $\eta _Z\cdot\alpha>0$ (see Lemma \ref{lem:hodge-index-theorem}). In particular, $\hat{q}^*\alpha\cdot\hat{q}^*\eta_Z$ is a non-zero element of $\NA(\hat{X})$. Therefore $\hat{q}^*\alpha^2=\hat{q}^*\alpha\cdot\hat{p}^*(K_X+B+\omega)$ is a non-zero element of $\NA(\hat{X})$, and thus $\alpha^2\neq 0$ in $N^1(Z)$.\\

 \noindent
\emph{Step $4$}: Finally,  \emph{Step $4$} of the proof of \cite[Theorem 1.4]{HP15} shows the existence of a fibration $\varphi:X\to S$ such that $(K_X+B+\omega)|_{X_s}\num 0$ for $s\in S$. This step works here without any change, and completes our proof.\\
\end{proof}

\begin{lemma}\label{lem:restriction-of-nef-and-big}
 	Let $X$ be a normal compact K\"ahler variety and $f:X\to C$  a proper surjective morphism to a smooth projective curve $C$. Let $\omega\in H^{1, 1}_{\rm BC}(X)$ be a nef and big class. Then the restriction $\omega|_F$ is nef and big for general fibers $F$ of $f$. 
\end{lemma}

\begin{proof}
    The proof follows immediately from the definition of nef and big class.
\end{proof}


\begin{corollary}\label{cor:partial-mfs}
	Let $(Y, B_Y)$ be a klt pair, where $Y$ is a $\mbQ$-factorial compact K\"ahler $3$-fold. Suppose that $Y$ is uniruled and the base of the MRC-fibration $g:Y\bir Z'$ has dimension $2$. Let $\omega_Y$ be a nef and big class on $Y$ such that $K_Y+B_Y+\omega_Y$ is nef and $(K_Y+B_Y+\omega_Y)\cdot F=0$, where $F$ is a general fiber of $g$. Then there exists a proper surjective morphism with connected fibers $\psi:Y\to S$ onto a normal compact K\"ahler surface $S$ such that $K_Y+B_Y+\omega_Y$ is $\psi$-trivial, i.e. $(K_Y+B_Y+\omega_Y)|_{Y_s}\num 0$ for all $s\in S$. 
\end{corollary}

\begin{proof}
 Let $h:X\to Y$ be a terminalization of the pair $(Y, B_Y )$ such that $X$ is strongly $\Q$-factorial (cf. Lemma \ref{l-ter}). Set $K_X+B:=h^*(K_Y+B_Y)$ and $\omega=h^*\omega_Y$. Then we have
 \[
		K_X+B+\omega=h^*(K_Y+B_Y+\omega_Y).
	\]
Note that since $\dim Z'=2$, then general fibers of $g:Y\bir Z'$ and $g\circ h:X\bir Z'$ are isomorphic. In particular, if $F$ is a general fiber of $g\circ h:X\bir Z'$, then $(K_X+B+\omega)\cdot F=(K_Y+B_Y+\omega_Y)\cdot F=0$. Thus by Theorem \ref{thm:log-contraction-terminal-pair} there is a proper surjective morphism with connected fibers $\varphi:X\to S$ to a K\"ahler surface $S$ such that $K_X+B+\omega$ is $\varphi$-trivial. With the notations as in the proof of Theorem \ref{thm:log-contraction-terminal-pair} we get the following commutative diagram.

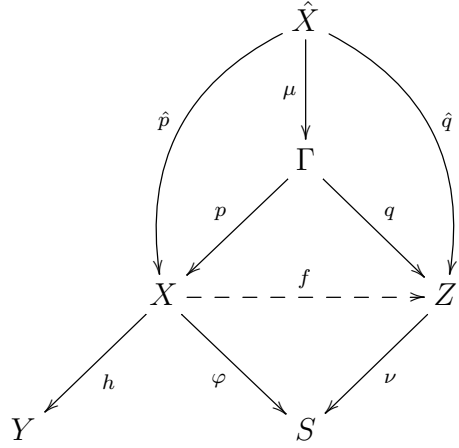
\begin{equation}\label{eqn:contraction}
	\xymatrixcolsep{3pc}\xymatrixrowsep{3pc}\xymatrix{&& \hat{X}\ar@/_2pc/[ddl]_{\hat{p}}\ar@/^2pc/[ddr]^{\hat{q}}\ar[d]_\mu &\\
	&& \Gamma\ar[dl]_p\ar[dr]^q&\\
&X\ar[ld]^h\ar@{-->}[rr]^f\ar[dr]_\varphi && Z\ar[dl]^\nu\\
Y && S &}
\end{equation}

Using the rigidity lemma now we will show that the fibration $\varphi:X\to S$ factors through $h:X\to Y$. Note that the fibers of $h$ are covered by curves, so it is enough to work with the curves contained in the fibers of $h$. Let $C$ be a curve in $X$ contracted by $h$. Then $(K_X+B+\omega)\cdot C=h^*(K_Y+B_Y+\omega_Y)\cdot C=0$. Let $\hat{C}$ be a curve in $\hat{X}$ such that $\hat{p}(\hat{C})=C$. Recall from the proof of Theorem \ref{thm:log-contraction-terminal-pair} that $\hat{p}^*(K_X+B+\omega)=\hat{q}^*{\alpha}$ for some nef and big class $\alpha$ on $Z$. Thus we have
\[\hat{q}^*\alpha\cdot\hat{C}=\hat{p}^*(K_X+B+\omega)\cdot\hat{C}=(K_X+B+\omega)\cdot C=0.\]
Therefore $\alpha\cdot\hat{q}_*(\hat{C})=0$. Now recall from the construction of $\nu:Z\to S$ in Step $4$ of the proof of \cite[Theorem 1.4]{HP15} that the morphism $\nu:Z\to S$ contracts exactly the $\alpha$-trivial curves. Therefore, either $\hat{C}$ is contracted by $\hat{q}$, or its image is contracted by $\nu$. In particular, from the diagram \eqref{eqn:contraction} it follows that the curve $C\subset X$ is contracted by $\varphi:X\to S$. Therefore, by the rigidity lemma (see \cite[Lemma 4.1.13]{BS95}), there is a proper surjective morphism with connected fibers $\psi:Y\to S$ such that $\varphi=\psi\circ h$. It is then clear that $K_Y+B_Y+\omega_Y$ is $\psi$-trivial.

\end{proof}

The following theorem is an application of Corollary \ref{cor:partial-mfs} and also a generalization of \cite[Theorem 2.7]{TZ18}. This establishes Theorem \ref{thm:bpf} when $K_X+B$ is not pseudo-effective.   
\begin{theorem}\label{thm:fano-contraction}
	 Let $(X,B)$ be a dlt pair, where $X$ is a $\mbQ$-factorial compact K\"ahler $3$-fold. Let $\omega$ be a nef and big class on $X$ such that $\alpha:=K_X+B+\omega$ is nef but not big. If either $(X,B)$ is a klt pair or $\omega$  is K\"ahler, then there exists a proper surjective morphism with connected fibers $\psi:X\to Y$ onto a normal compact K\"ahler variety $Y$ with rational singularities and a K\"ahler class $\alpha _Y\in H^{1, 1}_{\rm BC}(Y)$ such that $\alpha =\psi ^* \alpha _Y$. 
\end{theorem}

\begin{proof}
Note that if $(X,B)$ is dlt and $\omega$ is K\"ahler, then replacing $B$ by $(1-\epsilon)B$ and $\omega $ by $\omega +\epsilon B$, we may assume that $(X,B)$ is klt.
We will closely follow the arguments in \cite[Theorem 2.7]{TZ18}. First observe that, since $\omega$ is a big class and $\alpha$ is not big, $K_X+B$ is not pseudo-effective; in particular $K_X$ is not pseudo-effective. Therefore $X$ is uniruled. Let $X\bir T$ be the MRC fibration of $X$. 
If $\dim T\<1$, then from Lemma \ref{l-proj} and its proof it follows that $X$ is projective and $H^2(X, \mcO_X)=0$. In particular, $\alpha$ is an $\mbR$-divisor in this case, and our result follows from the well known base-point free theorem. So from now on we assume that $\dim T=2$. Now we claim that $\omega$ is a $(K_X+B)$-normalized nef and big class, i.e. $(K_X+B+\omega)\cdot F=0$ for general fibers $F$ of $X\bir T$. If not, then by Lemma \ref{lem:non-vanishing} there exists a $0<\mu<1$ such that $(K_X+B+\mu\omega)\cdot F=0$, where $F\cong \mbP^1$ is a general fiber of the MRC fibration $X\bir T$. Then by Corollary \ref{cor:nef-and-big-normalized}, $K_X+B+\mu\omega$ is pseudo-effective. Thus we have $\alpha=(K_X+B+\mu\omega)+(1-\mu)\omega$ is a big class, a contradiction.

Now by Corollary \ref{cor:partial-mfs} there exists a proper surjective morphism $f:X\to Y$ to a normal compact K\"ahler surface $Y$ such that $(K_X+B+\omega)|_{X_y}\num 0$ for all $y\in Y$.
 Also, as in the proof of Theorem \ref{thm:log-contraction-terminal-pair}, 
 we get the following commutative diagram:
\begin{equation}\label{eqn:bpf-diagram}
	\xymatrixcolsep{3pc}\xymatrixrowsep{3pc}\xymatrix{
	\Gamma\ar[r]^{q}\ar[d]_p & Z\ar[d]^\nu\\
	X\ar[r]^f & Y
	}
\end{equation}
where $Z$ is a smooth compact K\"ahler surface, $\Gamma$ is a normal $3$-fold in  Fujiki's class $\mcC$, and $p$ and $\nu$ are bimeromorphic. 

Replacing $\Gamma$ by a resolution we may further assume that $\Gamma$ is a compact K\"ahler $3$-fold (see Definition \ref{def:lots-of-definitions}.(i)). As in the proof of Corollary \ref{cor:partial-mfs} and Theorem \ref{thm:log-contraction-terminal-pair} we have $p^*\alpha=q^*\beta$ for some nef and big $(1, 1)$-class $\beta$ on $Z$. We claim that $\beta$ is pullback of a K\"ahler class from $Y$. To see this, first recall that from the Step 4 of the proof of \cite[Theorem 1.4]{HP15} it follows that $\Ex(\nu)=\Null(\beta)$ (see Definition \ref{def:null-locus} for $\Null(\beta)$). Moreover, by Lemma \ref{lem:rational-singularities}, $Y$ has rational singularities. Thus by Lemma \ref{l-HP} there is a $(1, 1)$-class $\gamma\in H^{1,1}_{\rm BC}(Y)$ such that $\beta=\nu^*\gamma$. Then $\gamma$ is nef and big (see \cite[Theorem 1]{Pau98}), and by the projection formula we have $\gamma^2=(\nu^*\gamma)^2=\beta^2>0$. Moreover, if $C\subset Y$ is a curve and $C'\subset Z$ its strict transform, then $\beta\cdot C'>0$, since $\Null(\beta)=\Ex(\nu)$ and $C'$ is not contained $\Ex(\nu)$. Then again by the projection formula we have $\gamma\cdot C=\nu^*\gamma\cdot C'=\beta\cdot C'>0$. Finally, since $Y$ is a normal surface, it has only finitely many (rational) singular points. Therefore by \cite[Lemma 2.1]{H18}, $\gamma$ is a K\"ahler class on $Y$. This proves the claim.

Now from the commutativity of the diagram \eqref{eqn:bpf-diagram} we have $p^*(f^*\gamma-\alpha)=0$. Since $p^*:H^{1, 1}_{\rm BC}(X)\to H^{1, 1}_{\rm BC}(\Gamma)$ is injective, we have 
$f^*\gamma-\alpha=0$, i.e. $K_X+B+\omega=f^*\gamma$, where $\gamma$ is a K\"ahler class on $Y$. This completes the proof.

\end{proof}

\section{Main Theorems}\label{sec:applications}
In this section we prove all the main theorems in full generality. The key technical result of this section is the proof of Theorem \ref{thm:bpf} for $X$ strongly $\mbQ$-factorial and $K_X+B$ pseud-effective, see Theorem \ref{thm:birational-contraction}. We start with the following defintion.
\begin{definition}\label{def:null-locus}
	Let $X$ be a normal compact K\"ahler variety and $\alpha\in H^{1, 1}_{\rm BC}(X)$ a nef and big class. Then we define the null locus of $\alpha$ as follows:
	 \[\Null(\alpha)=\bigcup_{\substack{Z\subset X,\\ \dim Z>0,\\ \alpha^{\dim Z}\cdot Z=0}} Z.\]
	 A priori we do not know whether $\Null(\alpha)$ is a closed analytic subset of $X$ or not.
\end{definition}

\begin{proposition}\label{pro:null-locus-contraction}
	Let $X$ be a normal $\mbQ$-factorial compact K\"ahler $3$-fold with klt singularities. Let $\alpha$ be a nef and big $(1, 1)$-class on $X$. Assume that $\Null(\alpha)$ consists of only finitely many curves of $X$. Then there exists a proper bimeromorphic morphism $\mu:X\to Z$ onto a normal analytic variety $Z$ such that every connected component of $\Null(\alpha)$ is contracted to a point and $X\setminus\Null(\alpha)\cong Z\setminus \mu(\Null(\alpha))$.
\end{proposition}

\begin{proof}

	The proof of \cite[Theorem 4.2]{CHP16}
	 holds here without any change. We note that \cite[Theorem 4.2]{CHP16} is stated only for non-uniruled varieties, however this hypothesis is only used to show that $\Null(\alpha)$ consists of finitely many curves  (see \cite[Proposition 4.4]{CHP16}).
 This is however part of our hypothesis.

\end{proof}

The following lemma provides a sufficient criteria in dimension $3$ for a nef and big class to be K\"ahler. The same result in arbitrary dimension is established in \cite[Theorem 2.29]{DHP22}.
\begin{lemma}\label{lem:nef-and-big-kahler}
	Let $X$ be a normal compact analytic variety of dimension $3$. Let $\alpha\in H^{1, 1}_{\rm BC}(X)$ be a nef and big class such that $\alpha^{\dim V}\cdot V>0$ for every positive dimensional subvariety $V\subset X$. Then $\alpha$ is a K\"ahler class.	
\end{lemma}  

\begin{proof}
	Let $f:Y\to X$ be a resolution of singularities of $X$. Then $f^*\alpha$ is a nef and big class on $Y$, and by the projection formula and \cite[Theorem 1.1]{CT15} it follows that $\EnK(f^*\alpha)=\Null(f^*\alpha)=\Ex(f)$,  where $\EnK(\cdot )$ denotes the non-K\"ahler locus, see \cite[Definition 3.16]{Bou04}. 
 Next, by Demailly's regularization theorem \cite{Dem92} and \cite[Theorem 3.17(ii)]{Bou04}, there exists a K\"ahler current $T$ with analytic singularities contained in the class $f^*\alpha$ such that $T$ is singular (i.e. not a smooth form) precisely along the exceptional locus of $f$. Therefore the current $f_*T$ (contained in the class $\alpha$) is a K\"ahler current which is singular along the closed set $f(\Ex(f))$. Since $\dim X=3$, we have that $\dim f(\Ex(f))\<1$. Now, since by construction $f_*T$ is a smooth form on the open set $X\setminus f(\Ex(f))$, the Lelong numbers $\nu(f_*T, x)=0$ for all $x\in X\setminus f(\Ex(f))$. Thus for any positive real number $c>0$, the Lelong sub-level sets $E_c(f_*T):=\{x\in X\; |\; \nu(f_*T, x)\geq c\}$ are contained in $f(\Ex(f))$. By a theorem of Siu \cite{Siu74}, we know that $E_c(f_*T)$ is a closed analytic subset of $X$ for all $c>0$. Therefore every irreducible component of $E_c(f_*T)$ is either a projective curve or a point contained in $f(\Ex(f))$ for all $c>0$. Let $C\subset f(\Ex(f))$ be an irreducible projective curve. Then first observe that if $\nu :\tilde C \to C$ is the normalization, then ${\rm NS}(\tilde C)_\mbR=H^{1, 1}_{\rm BC}(\tilde C)$, since $H^2(\tilde C, \mcO_{\tilde C})=0$. Therefore $\alpha|_{\tilde C}$ is a class of an $\mbR$-Cartier nef divisor on $\tilde C$ and in fact, it is an ample class, since 
	$\deg (\alpha|_{\tilde C})=\alpha\cdot C>0$ by hypothesis. { Now pushing forward $\alpha|_{\tilde{C}}$ by $\nu$, we see that $\alpha|_C$ is a class of an $\mbR$-Cartier divisor on $C$. Since $\nu$ is finite, $\alpha|_C$ is an ample class on $C$, hence a K\"ahler class.} Then by \cite[Proposition 3.3(iii)]{DP04} it follows that $\alpha\in H^{1, 1}_{\rm {BC}}(X)$ is a K\"ahler class on $X$.\\
		 
\end{proof}

The following result is a special case of Theorem \ref{thm:bpf}.
\begin{theorem}\label{thm:birational-contraction}
	 Let $(X,B)$ be a dlt pair, where $X$ is a normal strongly $\mbQ$-factorial compact K\"ahler $3$-fold. Let $\omega$ be a nef and big class on $X$ such that $\alpha:=K_X+B+\omega$ is nef and big. If either $(X,B)$ is klt or $\omega$ is K\"ahler, then there exists a proper bimeromorphic morphism $\psi:X\to Z$ onto a normal compact K\"ahler $3$-fold $Z$ with rational singularities and a K\"ahler class $\alpha _Z\in H^{1, 1}_{\rm BC}(Z)$ on $Z$ such that $\alpha =\psi^* \alpha _Z$. 
\end{theorem}

\begin{proof} Note that if $(X,B)$ is dlt and $\omega$ is K\"ahler, then replacing $B$ by $(1-\epsilon)B$ and $\omega $ by $\omega +\epsilon B$, we may assume that $(X,B)$ is klt.
Therefore, in this proof, we will assume that $(X,B)$ is klt and $\omega$ is nef and big.  We closely follow the arguments of \cite[Theorem 1.3]{H18}. For the convenience of the reader, we reproduce these arguments here indicating the necessary changes. The main idea is to construct $\psi$ by running an $\alpha$-trivial MMP and then to contract the remaining $\alpha$-trivial curves. First note that using Lemma \ref{lem:nef-and-big-to-modified-kahler} we may assume that $\omega$ is a modified K\"ahler class on $X$. Now if $H^2(X, \mathcal O _X)=0$, then $H^{1,1}_{\rm BC}(X)\cong\NS(X)_\mbR$ and $\alpha$ is represented by an $\mathbb R$-divisor. Therefore $X$ is a Moishezon space with rational singularities. Then by Theorem \ref{t-nam} $X$ is projective, and the result follows from a well known base-point free theorem for projective varieties.  So assume that $H^2(X, \mcO_X)\neq 0$. Then either $X$ is not uniruled and hence $K_X$ is pseudo-effective, or $X$ is uniruled and the base of the MRC fibration $X\bir Z$ has dimension $2$ (see Lemma \ref{l-proj} and its proof).\\



\begin{claim}\label{clm:contraction} We may run the $\alpha$-trivial $(K_X+B)$-MMP starting with $(X_0, B_0):=(X, B)$,  $\omega_0:=\omega$ and $\alpha_0:=\alpha$ and ending with a bimeromorphic map $\phi:X\bir X_n$ such that every $\alpha _n$-trivial curve is  $(K_{X_n}+B_n)$-non-negative, where $\alpha _n=\phi _* \alpha$ and $B_n=\phi _* B$. \end{claim}
\begin{proof}[Proof of Claim \ref{clm:contraction}]
By induction assume that we have already constructed the first $i$ steps
\[
	\phi_i:X_0\bir X_1\bir X_2\bir\ldots\bir X_i.
\]
Note that $(X_i,B_i:=\phi_{i,*}B)$ is a strongly $\mbQ$-factorial klt pair (by Lemma \ref{l-sqfmmp}), $\omega _i=\phi_{i,*}\omega$ is modified  K\"ahler and $\alpha _i=\phi_{i,*}\alpha$ is nef and big.
Now suppose that there is an $\alpha_i$-trivial curve $C\subset X_i$ such that $(K_{X_i}+B_i)\cdot C<0$ (and thus $\omega_i\cdot C>0$). We will show that there is an $\alpha_i$-trivial flip or divisorial contraction $\mu_i :X_i\dasharrow X_{i+1}$. 

If $K_{X_i}+B_i$ is not pseudo-effective (and hence the base of the MRC fibration $X_i\dasharrow Z_i$ has dimension $2$), then by Lemma \ref{lem:non-vanishing} it follows that $(K_{X_i}+B_i)\cdot F_i<0$ for general fibers $F_i$ of $X_i\bir Z_i$. 

The existence of $\mu_i$ follows from Theorem \ref{t-cone-mrc},  Corollary \ref{cor:cone-finite} and Theorem \ref{c-cont}. To see this, pick $0<\epsilon \ll 1$ so that $K_{X_i}+B_i+(1-\epsilon)\omega _i$ is big, in particular, $(K_{X_i}+B_i+(1-\epsilon)\omega _i)\cdot F_i>0$. Then by Corollary \ref{cor:cone-finite} 
\[\overline {\rm NA}(X_i)=\overline {\rm NA}(X_i)_{(K_{X_i}+B_i+(1-\epsilon)\omega _i)\geq 0}+\sum _{j=1}^N \mathbb R ^+[\Gamma _j],\]
where $(K_{X_i}+B_i+(1-\epsilon)\omega _i)\cdot \Gamma _j<0$.
We decompose $C=\eta+\sum _{j\in J}\lambda _j\Gamma _j$ accordingly, where $J=\{1,\ldots , N\}$. Since $\alpha_i \cdot C=0$, we may assume that $\alpha_i \cdot \Gamma _j=0$ for all $j\in J$. Since $(K_{X_i}+B_i+(1-\epsilon)\omega _i)\cdot C<0$, we may assume that $J\ne \emptyset$ and hence for some $j_0\in J$, we have $\lambda _{j_0}\ne 0$ and $(K_{X_i}+B_i+(1-\epsilon)\omega _i)\cdot \Gamma _{j_0}<0$.
 Note that Theorem \ref{c-cont} does not immediately apply  here 
 since $\omega _i$ is no longer nef. 
 We will instead argue as follows. Since $\alpha _i\cdot\Gamma_{j_0}=0$ and $(K_{X_i}+B_i+(1-\epsilon)\omega _i)\cdot \Gamma _{j_0}<0$, it follows that $\omega _i\cdot \Gamma _{j_0}>0$ and
 $(K_{X_i}+B_i)\cdot \Gamma _{j_0}<0$.
 Since $\Gamma _{j_0}$ generates an extremal ray, it is cut out by a nef class $\gamma$. Since $\alpha _i$ is nef and big, replacing $\gamma $ by $\gamma +\alpha_i$, we may assume that $\gamma$ is nef and big. By standard arguments, $\eta:=t\gamma-(K_{X_i}+B_i)$ is K\"ahler for $t\gg 0$. Then $K_{X_i}+B_i+\eta=t\gamma$ is nef and big, $\Null(t\gamma)\cap\NA(X_i)=R_{j_0}$ and $\alpha_i\cdot R_{j_0}=0$, where $R_{j_0}=\mbR^+[\Gamma_{j_0}]$.  Moreover, recall that if $F$ is a general fiber of the MRC fibration $X_i\bir Z_i$, then $(K_{X_i}+B_i)\cdot F<0$ (by Lemma \ref{lem:non-vanishing}). Now choose $0<\delta\ll 1$ so that $K_{X_i}+B_i+(1-\delta)\eta$ is still big (hence pseudo-effective); then $R_{j_0}$ is a ($K_{X_i}+B_i+(1-\delta)\eta$)-negative extremal ray.
  Thus by Theorem \ref{c-cont} the contraction $c_{R_{j_0}}:X_i\to Y$ of $R_{j_0}$ exists. If $c_{R_{j_0}}$ is divisorial, then we let $X_{i+1}=Y$ and 
 $\mu _i=c_{R_{j_0}}$. If $c_{R_{j_0}}$ is small, then the flip $X_i\dasharrow X_{i}^+$ exists by Theorem \ref{t-flip}. In this case we let $X_{i+1}=X_{i}^+$ and $\mu _i:X_i\dasharrow X_{i+1}$ the induced bimeromorphic map.

  On the other hand if $K_{X_i}+B_i$ is pseudo-effective, then the existence of $\mu_i$ follows from  Theorems \ref{t-cone}, \ref{t-pmmp}, \ref{t-flip} and \ref{t-div1sqf}. Note that $\alpha_{i+1}:=\mu_{i, *}\alpha_i$ is nef and big by \cite[Pro. 3.1, eqn. (5)]{CHP16}, and $\omega_{i+1}:=\mu_{i, *}\omega_i$ is a modified K\"ahler class by Lemma \ref{l-mod}. This MMP terminates after finitely many steps, say $\phi:X\bir X_n$ by Theorem \ref{t-term}. In particular, every $\alpha_n$-trivial curve on $X_n$ is $(K_{X_n}+B_n)$-non-negative, or equivalently, every $(K_{X_n}+B_n)$-negative curve is $\alpha_n$-positive.
  \end{proof}

By Lemma \ref{lem:null-locus}, if ${\Null}(\alpha _{n})$ contains a surface $S$, then $S$ is Moishezon and it is covered by a family of $\alpha_n$-trivial curves $\{C_t\}_{t\in T}$. Since $\omega _n$ is a modified K\"ahler class, $\omega _n|_S$ is big and hence $\omega _n\cdot C_t=\omega_n|_S\cdot C_t>0$. But then from $0=\alpha_n\cdot C_t=(K_{X_n}+B_n+\omega_n)\cdot C_t$ it follows that $(K_{X_n}+B_n)\cdot C_t<0$, this is a contradiction. Therefore $\Null(\alpha_n)$ is a union of curves.

Now let $f:X'\to X_n$ be a resolution of singularities of $X_n$. Then $f^*\alpha_n$ is nef and big, and thus by \cite[Theorem 1.1]{CT15}, the non-K\"ahler locus is $\EnK(f^*\alpha_n)=\Null(f^*\alpha_n)$. If $S'\subset \EnK(f^*\alpha_n)$ is a divisor, then by the projection formula $0=(f^*\alpha_n)^2\cdot S'=\alpha_n^2\cdot f_*S'$. Thus $S'$ is $f$-exceptional, since $\Null(\alpha_n)$ does not contain any surface, as proved above. Moreover, since $\EnK(f^*\alpha_n)$ is a proper closed analytic subset of $X'$ (see \cite[Theorem 2.2]{CT15}), it follows that $\Null(f^*\alpha_n)$ is a finite union of curves and $f$-exceptional divisors. But since $\Null(\alpha_n)\subset\EnK(\alpha_n)\subset f(\EnK(f^*\alpha_n))$, it follows that $\Null(\alpha_n)$ is a finite union of curves.
Then by Proposition \ref{pro:null-locus-contraction} there exists a proper bimeromorphic morphism $\mu:X_n\to Z$ onto a normal compact analytic variety $Z$ contracting each curve in  ${\Null}(\alpha _{n})$ to a point and inducing an isomorphism on the open sets $X_n\setminus {\Null}(\alpha _{n})\cong Z\setminus\mu(\Null(\alpha_n))$.

Set $\psi:=\mu\circ\phi_n:X\bir Z$. We claim that $\psi$ is a morphism. To see this, we will use descending induction to show that the induced map $\psi _i: X_i\bir Z$ is a morphism. Note that  $\psi _n$ is a morphism as constructed above. Suppose that we have already shown that $\psi _i: X_i\to Z$ is a morphism. If $X_{i-1}\to X_i$ is a divisorial contraction, then clearly $X_{i-1}\to Z$ is a morphism and we are done by induction on $i$. If on the other hand $X_{i-1}\dasharrow X_i$ is a flip, then let 
$X_i\to Z_{i-1}$ be the flipped contraction. If $Z_{i-1}\bir Z$ is not a morphism, then by the rigidity lemma (see \cite[Lemma 4.1.13]{BS95}) there is a flipped curve $C_i\subset X_i$  such that $C:=\psi _{i,*}C_i\ne 0$. 
Let $W$ be the normalization of the graph of the induced bimeromorphic map $\pi_i:X_i\bir X_n$, and $p:W\to X_i$ and $q:W\to X_n$ are the projections. Now by construction $\alpha_i$ and $\alpha_n$ are both nef classes. Moreover, since $\pi_i:X_i\dasharrow X_n$ is a composition of a finite sequence of $\alpha_i$-trivial flips and divisorial contractions, by a repeated application of \cite[Proposition 3.1, eqn. (5)]{CHP16} it follows that $p^*\alpha_i=q^*\alpha_n$. Let $C_W\subset W$ be a curve such that $p_* C_W=dC_i$ for some $d>0$ and $C_n:=q_*C_W\subset X_n$. Since $\psi _n(C_n)=C$, then $\alpha _n\cdot C_n>0$. Then we have
\[0=\alpha _i \cdot dC_i= p^*\alpha _i \cdot C_W=q^*\alpha_n\cdot C_W=\alpha _n\cdot C_n >0,\mbox{ a contradiction.}\]
 Therefore $Z_{i-1}\to Z$ is a morphism, and since $X_{i-1}\to Z_{i-1}$ is a flipping contraction, $X_{i-1}\to Z$ is also a morphism. This concludes the proof that $\psi:X\to Z$ is a morphism.

 Next we claim that $Z$ is a K\"ahler variety with rational singularities and there exists a K\"ahler class $\alpha_Z\in H^{1, 1}_{\rm BC}(Z)$ such that $\alpha_n=\mu^*\alpha_Z$. To this end, first observe that $(X, B)$ has $\mbQ$-factorial klt singularities, and $\alpha=K_{X}+B+\omega$ and $\omega$ are both nef and big classes. { Now recall that $\mu:X_n\to Z$ contracts precisely the null locus $\Null(\alpha_n)$, and since this locus is 1 dimensional, $\mu$ is $\alpha _n$-trivial.
 Since
 $X\dasharrow X_n$ is a sequence of $\alpha$-trivial flips and divisorial contractions, it follows that $\psi:X\to Z$ is also $\alpha$-trivial, i.e. the restriction of $\alpha$ to every fiber of $\psi$ is numerically trivial, so $\alpha|_{X_{ z}}\num 0$ for all $z\in Z$}. 
  Thus $-(K_{X}+B)|_{X_{z}}\num \omega|_{X_{z}}$ for all $z\in Z$. In particular, $-(K_{X}+B)$ is $\psi$-nef, since $\omega$ is nef by hypothesis. Moreover, since $\psi$ is a bimeromorphic morphism, $-(K_{X}+B)$ is $\psi$-big. 
 By Lemma \ref{lem:rational-singularities}, $Z$ has rational singularities. 
 
 Now from  Definition \ref{def:lots-of-definitions}.(i) it follows that $Z$ is in Fujiki's class $\mcC$. 
Then by Lemma \ref{l-HP} there exists a $(1, 1)$-class $\alpha_Z\in H^{1, 1}_{\rm BC}(Z)$ (represented by a real closed $(1, 1)$-form with local potentials) such that $\alpha_n=\mu^*\alpha_Z$.

Next we claim that $\alpha_Z$ is a K\"ahler class on $Z$. Indeed, let $V\subset Z$ be a subvariety of positive dimension and $V'$ the strict transform of $V$ under $\mu$. 
By the projection formula we have $(\alpha_Z)^{\dim V}\cdot V=(\alpha _n)^{\dim V'}\cdot V'>0$, since $V'$ is not contained in $\Null(\alpha _n)$. Then by Lemma \ref{lem:nef-and-big-kahler}, $\alpha_Z$ is a K\"ahler class on $Z$.

 Finally notice that, since every step of the above MMP is $\alpha$-trivial, from the construction above it follows that $\alpha=\psi^*\alpha_Z$. This completes the proof.\\

\end{proof}


We are now ready to prove our main theorems in full generality. We start with the base-point free Theorem \ref{thm:bpf}.\\

\begin{proof}[Proof of Theorem \ref{thm:bpf}]
If $(X,B)$ is dlt and $\alpha -(K_X+B)$ is K\"ahler, then $(X,(1-\epsilon)B)$ is klt and $\alpha  -(K_X+(1-\epsilon)B)$ is K\"ahler for any $0<\epsilon \ll 1$. Therefore in either case we may assume that $(X,B)$ is klt and $\alpha -(K_X+B)$ is nef and big.
Let $\nu :X'\to X$ be the small projective morphism given by Lemma \ref{l-ter}, then $X'$ is strongly $\Q$-factorial, $K_{X'}+B'=\nu ^* (K_X+B)$ is klt, $\alpha '=\nu ^*\alpha$ is nef and $\alpha ' -(K_{X'}+B')$ is nef and big.
	 By Theorem \ref{thm:fano-contraction} and Theorem \ref{thm:birational-contraction} there exist a proper surjective morphism with connected fibers $\phi :X'\to Z$ to a normal K\"ahler variety $Z$ with rational singularities and a K\"ahler class $\alpha _Z\in H^{1,1}_{\rm BC}(Z)$  such that $\alpha '=\phi ^* \alpha _Z$. 
  Let $C$ be a $\nu$-exceptional curve, then $\alpha '\cdot C=\alpha _Z\cdot \phi _*C=0$. Since $\alpha _Z$ is K\"ahler, then $C$ is contracted by $\psi$. By the rigidity lemma (see \cite[Lemma 4.1.13]{BS95}), there is a morphism $\psi: X\to Z$ such that $\phi=\psi\circ \nu$. Thus $\alpha=\nu_*(\alpha')=\nu_*(\nu^*(\psi^*\alpha_Z))=\psi^*\alpha_Z$; this completes our proof.

\end{proof}

\begin{proof}[Proof of Theorem \ref{t-div1}]
    This is an immediate corollary of Theorem \ref{thm:bpf}.
\end{proof}

\begin{proof}[Proof of Theorem \ref{t-main}]  By Theorem \ref{t-cone} the cone theorem holds for $(X,B)$. By Theorem \ref{t-pmmp}, Theorem \ref{t-flip} and Theorem \ref{t-div1} flips and divisorial contractions exist and hence we may run a MMP which terminates by Theorem \ref{t-term}.\\
\end{proof}

\begin{proof}[Proof of Theorem \ref{t-main2}]
Since $K_X+B$ is not pseudo-effective, $K_X$ is not pseudo-effective. Thus by \cite[Corollary 1.2]{Bru06} applied to a resolution of $X$ it follows that $X$ is uniruled. By Lemma \ref{l-proj}, we may assume that the dimension of the base of the MRC fibration $X\bir Z$ is 2. Let $F$ be a general fiber of the MRC fibration $f:X\dasharrow Z$. By Lemma \ref{lem:non-vanishing}, if $(K_X+B)\cdot F\geq 0$, then $K_X+B$ is pseudo-effective, contradicting our assumption. Therefore $(K_X+B)\cdot F<0$.

By Theorem \ref{thm:partial-mfs}, there is a $(K_X+B)$-MMP, $\phi:X\dasharrow X'$ such that for every $(K_{X'}+B')$-normalized K\"ahler class $\omega '$ (see Definition \ref{def:normalized-kahler-form}), the class $K_{X'}+B'+\omega'$ is nef, where $B'=\phi_*B$. 
Note that thanks to Theorem \ref{t-div1}, Theorem \ref{thm:partial-mfs} and its proof also applies to $\Q$-factorial pairs that are not necessarily strongly $\Q$-factorial.

Since $X'$ is K\"ahler and $(K_{X'}+B')\cdot F'<0$, where $F'\cong F$ is a general fiber of the induced MRC fibration $X'\dasharrow Z$, we may pick a $(K_{X'}+B')$-normalized K\"ahler class, say, $\omega '$. Now choose $0<\epsilon\ll 1$ so that $\omega'+\epsilon B'$ is a K\"ahler class. Then $(X', (1-\epsilon)B')$ is klt, and thus by Corollary \ref{cor:partial-mfs}, there exists a holomorphic fibration $\psi : X'\to S'$ onto a normal compact K\"ahler surface $S'$ such that $(K_{X'}+B'+\omega')|_{X'_{s'}}\num (K_{X'}+(1-\epsilon)B'+(\omega'+\epsilon B'))|_{X'_{s'}}\num 0$ for all $s'\in S'$, i.e. $K_{X'}+B'+\omega'$ is $\psi$-trivial. In particular, $\psi$ is a projective morphism and so the theorem now follows from the usual relative Minimal Model Program for projective morphisms as in Proposition \ref{pro:relative-projective-mmp}.
\end{proof}

\bibliographystyle{hep}
\bibliography{KahlerReferences}

\begin{bibdiv}
\begin{biblist}

\bib{AHV77}{book}{
      author={Aroca, Jos\'{e}~M.},
      author={Hironaka, Heisuke},
      author={Vicente, Jos\'{e}~L.},
       title={Desingularization theorems},
      series={Memorias de Matem\'{a}tica del Instituto ``Jorge Juan''
  [Mathematical Memoirs of the Jorge Juan Institute]},
   publisher={Consejo Superior de Investigaciones Cient\'{\i}ficas, Madrid},
        date={1977},
      volume={30},
        ISBN={84-00-03602-6},
      review={\MR{480502}},
}

\bib{Ara10}{article}{
      author={Araujo, Carolina},
       title={The cone of pseudo-effective divisors of log varieties after
  {B}atyrev},
        date={2010},
        ISSN={0025-5874},
     journal={Math. Z.},
      volume={264},
      number={1},
       pages={179\ndash 193},
         url={http://dx.doi.org/10.1007/s00209-008-0457-8},
      review={\MR{2564937}},
}

\bib{BCE02}{incollection}{
      author={Bauer, Thomas},
      author={Campana, Fr\'{e}d\'{e}ric},
      author={Eckl, Thomas},
      author={Kebekus, Stefan},
      author={Peternell, Thomas},
      author={Rams, S\l~awomir},
      author={Szemberg, Tomasz},
      author={Wotzlaw, Lorenz},
       title={A reduction map for nef line bundles},
        date={2002},
   booktitle={Complex geometry ({G}\"{o}ttingen, 2000)},
   publisher={Springer, Berlin},
       pages={27\ndash 36},
      review={\MR{1922095}},
}

\bib{BCHM10}{article}{
      author={Birkar, Caucher},
      author={Cascini, Paolo},
      author={Hacon, Christopher~D.},
      author={McKernan, James},
       title={Existence of minimal models for varieties of log general type},
        date={2010},
        ISSN={0894-0347},
     journal={J. Amer. Math. Soc.},
      volume={23},
      number={2},
       pages={405\ndash 468},
         url={http://dx.doi.org/10.1090/S0894-0347-09-00649-3},
      review={\MR{2601039 (2011f:14023)}},
}

\bib{BHPV04}{book}{
      author={Barth, Wolf~P.},
      author={Hulek, Klaus},
      author={Peters, Chris A.~M.},
      author={Van~de Ven, Antonius},
       title={Compact complex surfaces},
     edition={Second},
      series={Ergebnisse der Mathematik und ihrer Grenzgebiete. 3. Folge. A
  Series of Modern Surveys in Mathematics [Results in Mathematics and Related
  Areas. 3rd Series. A Series of Modern Surveys in Mathematics]},
   publisher={Springer-Verlag, Berlin},
        date={2004},
      volume={4},
        ISBN={3-540-00832-2},
         url={https://doi.org/10.1007/978-3-642-57739-0},
      review={\MR{2030225}},
}

\bib{BM97}{article}{
      author={Bierstone, Edward},
      author={Milman, Pierre~D.},
       title={Canonical desingularization in characteristic zero by blowing up
  the maximum strata of a local invariant},
        date={1997},
        ISSN={0020-9910},
     journal={Invent. Math.},
      volume={128},
      number={2},
       pages={207\ndash 302},
         url={https://doi.org/10.1007/s002220050141},
      review={\MR{1440306}},
}

\bib{Bou02ar}{article}{
      author={{Boucksom}, Sebastien},
       title={{On the volume of a line bundle}},
        date={2002-01},
     journal={arXiv Mathematics e-prints},
       pages={math/0201031v1},
      eprint={math/0201031v1},
}

\bib{Bou02}{article}{
      author={Boucksom, S\'{e}bastien},
       title={On the volume of a line bundle},
        date={2002},
        ISSN={0129-167X},
     journal={Internat. J. Math.},
      volume={13},
      number={10},
       pages={1043\ndash 1063},
         url={https://doi.org/10.1142/S0129167X02001575},
      review={\MR{1945706}},
}

\bib{Bou04}{article}{
      author={Boucksom, S\'ebastien},
       title={Divisorial {Z}ariski decompositions on compact complex
  manifolds},
        date={2004},
        ISSN={0012-9593},
     journal={Ann. Sci. \'Ecole Norm. Sup. (4)},
      volume={37},
      number={1},
       pages={45\ndash 76},
         url={https://doi.org/10.1016/j.ansens.2003.04.002},
      review={\MR{2050205}},
}

\bib{Bru06}{article}{
      author={Brunella, Marco},
       title={A positivity property for foliations on compact {K}\"{a}hler
  manifolds},
        date={2006},
        ISSN={0129-167X},
     journal={Internat. J. Math.},
      volume={17},
      number={1},
       pages={35\ndash 43},
         url={https://doi.org/10.1142/S0129167X06003333},
      review={\MR{2204838}},
}

\bib{BS95}{book}{
      author={Beltrametti, Mauro~C.},
      author={Sommese, Andrew~J.},
       title={The adjunction theory of complex projective varieties},
      series={De Gruyter Expositions in Mathematics},
   publisher={Walter de Gruyter \& Co., Berlin},
        date={1995},
      volume={16},
        ISBN={3-11-014355-0},
         url={https://doi.org/10.1515/9783110871746},
      review={\MR{1318687}},
}

\bib{Cam92}{article}{
      author={Campana, F.},
       title={Connexit\'{e} rationnelle des vari\'{e}t\'{e}s de {F}ano},
        date={1992},
        ISSN={0012-9593},
     journal={Ann. Sci. \'{E}cole Norm. Sup. (4)},
      volume={25},
      number={5},
       pages={539\ndash 545},
         url={http://www.numdam.org/item?id=ASENS_1992_4_25_5_539_0},
      review={\MR{1191735}},
}

\bib{CH20}{article}{
      author={Cao, Junyan},
      author={H\"{o}ring, Andreas},
       title={Rational curves on compact {K}\"{a}hler manifolds},
        date={2020},
        ISSN={0022-040X},
     journal={J. Differential Geom.},
      volume={114},
      number={1},
       pages={1\ndash 39},
         url={https://doi.org/10.4310/jdg/1577502017},
      review={\MR{4047551}},
}

\bib{CHP16}{article}{
      author={Campana, Fr\'ed\'eric},
      author={H\"oring, Andreas},
      author={Peternell, Thomas},
       title={Abundance for {K}\"ahler threefolds},
        date={2016},
        ISSN={0012-9593},
     journal={Ann. Sci. \'Ec. Norm. Sup\'er. (4)},
      volume={49},
      number={4},
       pages={971\ndash 1025},
         url={https://doi.org/10.24033/asens.2301},
      review={\MR{3552019}},
}

\bib{CP97}{article}{
      author={Campana, Fr\'{e}d\'{e}ric},
      author={Peternell, Thomas},
       title={Towards a {M}ori theory on compact {K}\"{a}hler threefolds. {I}},
        date={1997},
        ISSN={0025-584X},
     journal={Math. Nachr.},
      volume={187},
       pages={29\ndash 59},
         url={https://doi.org/10.1002/mana.19971870104},
      review={\MR{1471137}},
}

\bib{CT15}{article}{
      author={Collins, Tristan~C.},
      author={Tosatti, Valentino},
       title={K\"{a}hler currents and null loci},
        date={2015},
        ISSN={0020-9910},
     journal={Invent. Math.},
      volume={202},
      number={3},
       pages={1167\ndash 1198},
         url={https://doi.org/10.1007/s00222-015-0585-9},
      review={\MR{3425388}},
}

\bib{Das20}{article}{
      author={Das, Omprokash},
       title={Finiteness of log minimal models and nef curves on 3-folds in
  characteristic {$p>5$}},
        date={2020},
        ISSN={0027-7630},
     journal={Nagoya Math. J.},
      volume={239},
       pages={76\ndash 109},
         url={https://doi.org/10.1017/nmj.2018.28},
      review={\MR{4138896}},
}

\bib{Deb01}{book}{
      author={Debarre, Olivier},
       title={Higher-dimensional algebraic geometry},
      series={Universitext},
   publisher={Springer-Verlag, New York},
        date={2001},
        ISBN={0-387-95227-6},
         url={http://dx.doi.org/10.1007/978-1-4757-5406-3},
      review={\MR{1841091}},
}

\bib{Dem12}{book}{
      author={Demailly, Jean-Pierre},
       title={Analytic methods in algebraic geometry},
      series={Surveys of Modern Mathematics},
   publisher={International Press, Somerville, MA; Higher Education Press,
  Beijing},
        date={2012},
      volume={1},
        ISBN={978-1-57146-234-3},
      review={\MR{2978333}},
}

\bib{Dem85}{article}{
      author={Demailly, Jean-Pierre},
       title={Mesures de {M}onge-{A}mp\`ere et caract\'{e}risation
  g\'{e}om\'{e}trique des vari\'{e}t\'{e}s alg\'{e}briques affines},
        date={1985},
        ISSN={0037-9484},
     journal={M\'{e}m. Soc. Math. France (N.S.)},
      number={19},
       pages={124},
      review={\MR{813252}},
}

\bib{Dem92}{article}{
      author={Demailly, Jean-Pierre},
       title={Regularization of closed positive currents and intersection
  theory},
        date={1992},
        ISSN={1056-3911},
     journal={J. Algebraic Geom.},
      volume={1},
      number={3},
       pages={361\ndash 409},
      review={\MR{1158622}},
}

\bib{DH23}{article}{
      author={{Das}, Omprokash},
      author={{Hacon}, Christopher},
       title={{On the Minimal Model Program for K{\"a}hler 3-folds}},
        date={2023-06},
     journal={arXiv e-prints},
       pages={arXiv:2306.11708v1},
      eprint={2306.11708v1},
}

\bib{DHP22}{article}{
      author={Das, Omprokash},
      author={Hacon, Christopher},
      author={P\u{a}un, Mihai},
       title={On the 4-dimensional minimal model program for {K}\"{a}hler
  varieties},
        date={2024},
        ISSN={0001-8708},
     journal={Adv. Math.},
      volume={443},
       pages={Paper No. 109615},
         url={https://doi.org/10.1016/j.aim.2024.109615},
      review={\MR{4719824}},
}

\bib{DHY23}{article}{
      author={{Das}, Omprokash},
      author={{Hacon}, Christopher},
      author={{Y{\'a}{\~n}ez}, Jos{\'e}~Ignacio},
       title={{MMP for Generalized Pairs on K{\"a}hler 3-folds}},
        date={2023-04},
     journal={arXiv e-prints},
       pages={arXiv:2305.00524v1},
      eprint={2305.00524v1},
}

\bib{DO23}{article}{
      author={Das, Omprokash},
      author={Ou, Wenhao},
       title={On the log abundance for compact {K}\"{a}hler threefolds},
        date={2024},
        ISSN={0025-2611},
     journal={Manuscripta Math.},
      volume={173},
      number={1-2},
       pages={341\ndash 404},
         url={https://doi.org/10.1007/s00229-023-01467-6},
      review={\MR{4684351}},
}

\bib{DP04}{article}{
      author={Demailly, Jean-Pierre},
      author={Paun, Mihai},
       title={Numerical characterization of the {K}\"{a}hler cone of a compact
  {K}\"{a}hler manifold},
        date={2004},
        ISSN={0003-486X},
     journal={Ann. of Math. (2)},
      volume={159},
      number={3},
       pages={1247\ndash 1274},
         url={https://doi.org/10.4007/annals.2004.159.1247},
      review={\MR{2113021}},
}

\bib{ES10}{incollection}{
      author={El~Zein, Fouad},
      author={Snoussi, Jawad},
       title={Local systems and constructible sheaves},
        date={2010},
   booktitle={Arrangements, local systems and singularities},
      series={Progr. Math.},
      volume={283},
   publisher={Birkh\"{a}user Verlag, Basel},
       pages={111\ndash 153},
         url={https://doi.org/10.1007/978-3-0346-0209-9},
      review={\MR{3025862}},
}

\bib{Fuj13}{article}{
      author={Fujino, Osamu},
       title={A transcendental approach to {K}oll\'{a}r's injectivity theorem
  {II}},
        date={2013},
        ISSN={0075-4102},
     journal={J. Reine Angew. Math.},
      volume={681},
       pages={149\ndash 174},
         url={https://doi.org/10.1515/crelle-2012-0036},
      review={\MR{3181493}},
}

\bib{Fuj22}{article}{
      author={{Fujino}, Osamu},
       title={{Minimal model program for projective morphisms between complex
  analytic spaces}},
        date={2022-01},
     journal={arXiv e-prints},
       pages={arXiv:2201.11315v1},
      eprint={2201.11315v1},
}

\bib{Fuj78}{article}{
      author={Fujiki, Akira},
       title={Closedness of the {D}ouady spaces of compact {K}\"{a}hler
  spaces},
        date={1978/79},
        ISSN={0034-5318},
     journal={Publ. Res. Inst. Math. Sci.},
      volume={14},
      number={1},
       pages={1\ndash 52},
         url={https://doi.org/10.2977/prims/1195189279},
      review={\MR{0486648}},
}

\bib{Fuj83}{article}{
      author={Fujiki, Akira},
       title={K\"{a}hlerian normal complex surfaces},
        date={1983},
        ISSN={0040-8735},
     journal={Tohoku Math. J. (2)},
      volume={35},
      number={1},
       pages={101\ndash 117},
         url={https://doi.org/10.2748/tmj/1178229104},
      review={\MR{695662}},
}

\bib{GPR94}{book}{
      editor={Grauert, H.},
      editor={Peternell, Th.},
      editor={Remmert, R.},
       title={Several complex variables. {VII}},
      series={Encyclopaedia of Mathematical Sciences},
   publisher={Springer-Verlag, Berlin},
        date={1994},
      volume={74},
        ISBN={3-540-56259-1},
         url={https://doi.org/10.1007/978-3-662-09873-8},
        note={Sheaf-theoretical methods in complex analysis, A reprint of
  {{\i}t Current problems in mathematics. Fundamental directions. Vol. 74}
  (Russian), Vseross. Inst. Nauchn. i Tekhn. Inform. (VINITI), Moscow},
      review={\MR{1326617}},
}

\bib{G16}{article}{
      author={Guenancia, Henri},
       title={Families of conic {K}\"{a}hler-{E}instein metrics},
        date={2020},
        ISSN={0025-5831},
     journal={Math. Ann.},
      volume={376},
      number={1-2},
       pages={1\ndash 37},
         url={https://doi.org/10.1007/s00208-018-1769-6},
      review={\MR{4055154}},
}

\bib{H18}{article}{
      author={H\"{o}ring, A.},
       title={Adjoint {$(1,1)$}-classes on threefolds},
        date={2021},
        ISSN={1607-0046},
     journal={Izv. Ross. Akad. Nauk Ser. Mat.},
      volume={85},
      number={4},
       pages={215\ndash 224},
         url={https://doi.org/10.4213/im9084},
      review={\MR{4295001}},
}

\bib{Har74}{article}{
      author={Harvey, Reese},
       title={Removable singularities of cohomology classes in several complex
  variables},
        date={1974},
        ISSN={0002-9327},
     journal={Amer. J. Math.},
      volume={96},
       pages={498\ndash 504},
         url={https://doi.org/10.2307/2373557},
      review={\MR{364671}},
}

\bib{Hir75}{article}{
      author={Hironaka, Heisuke},
       title={Flattening theorem in complex-analytic geometry},
        date={1975},
        ISSN={0002-9327},
     journal={Amer. J. Math.},
      volume={97},
       pages={503\ndash 547},
         url={https://doi.org/10.2307/2373721},
      review={\MR{0393556}},
}

\bib{HM07}{article}{
      author={Hacon, Christopher~D.},
      author={Mckernan, James},
       title={On {S}hokurov's rational connectedness conjecture},
        date={2007},
        ISSN={0012-7094},
     journal={Duke Math. J.},
      volume={138},
      number={1},
       pages={119\ndash 136},
         url={https://doi.org/10.1215/S0012-7094-07-13813-4},
      review={\MR{2309156}},
}

\bib{HP15}{article}{
      author={H\"oring, Andreas},
      author={Peternell, Thomas},
       title={Mori fibre spaces for {K}\"ahler threefolds},
        date={2015},
        ISSN={1340-5705},
     journal={J. Math. Sci. Univ. Tokyo},
      volume={22},
      number={1},
       pages={219\ndash 246},
      review={\MR{3329195}},
}

\bib{HP16}{article}{
      author={H\"oring, Andreas},
      author={Peternell, Thomas},
       title={Minimal models for {K}\"ahler threefolds},
        date={2016},
        ISSN={0020-9910},
     journal={Invent. Math.},
      volume={203},
      number={1},
       pages={217\ndash 264},
         url={https://doi.org/10.1007/s00222-015-0592-x},
      review={\MR{3437871}},
}

\bib{HP17}{incollection}{
      author={H\"{o}ring, Andreas},
      author={Peternell, Thomas},
       title={Bimeromorphic geometry of {K}\"{a}hler threefolds},
        date={2018},
   booktitle={Algebraic geometry: {S}alt {L}ake {C}ity 2015},
      series={Proc. Sympos. Pure Math.},
      volume={97},
   publisher={Amer. Math. Soc., Providence, RI},
       pages={381\ndash 402},
      review={\MR{3821156}},
}

\bib{KM98}{book}{
      author={Koll{\'a}r, J{\'a}nos},
      author={Mori, Shigefumi},
       title={Birational geometry of algebraic varieties},
      series={Cambridge Tracts in Mathematics},
   publisher={Cambridge University Press},
     address={Cambridge},
        date={1998},
      volume={134},
        ISBN={0-521-63277-3},
         url={http://dx.doi.org/10.1017/CBO9780511662560},
        note={With the collaboration of C. H. Clemens and A. Corti, Translated
  from the 1998 Japanese original},
      review={\MR{1658959 (2000b:14018)}},
}

\bib{KMM92}{article}{
      author={Koll\'{a}r, J\'{a}nos},
      author={Miyaoka, Yoichi},
      author={Mori, Shigefumi},
       title={Rational connectedness and boundedness of {F}ano manifolds},
        date={1992},
        ISSN={0022-040X},
     journal={J. Differential Geom.},
      volume={36},
      number={3},
       pages={765\ndash 779},
         url={http://projecteuclid.org/euclid.jdg/1214453188},
      review={\MR{1189503}},
}

\bib{Kob87}{book}{
      author={Kobayashi, Shoshichi},
       title={Differential geometry of complex vector bundles},
      series={Publications of the Mathematical Society of Japan},
   publisher={Princeton University Press, Princeton, NJ; Princeton University
  Press, Princeton, NJ},
        date={1987},
      volume={15},
        ISBN={0-691-08467-X},
         url={https://doi.org/10.1515/9781400858682},
        note={Kan\^{o} Memorial Lectures, 5},
      review={\MR{909698}},
}

\bib{Kod54}{article}{
      author={Kodaira, K.},
       title={On {K}\"{a}hler varieties of restricted type (an intrinsic
  characterization of algebraic varieties)},
        date={1954},
        ISSN={0003-486X},
     journal={Ann. of Math. (2)},
      volume={60},
       pages={28\ndash 48},
         url={https://doi.org/10.2307/1969701},
      review={\MR{68871}},
}

\bib{Kol91}{article}{
      author={Koll\'{a}r, J\'{a}nos},
       title={Extremal rays on smooth threefolds},
        date={1991},
        ISSN={0012-9593},
     journal={Ann. Sci. \'{E}cole Norm. Sup. (4)},
      volume={24},
      number={3},
       pages={339\ndash 361},
         url={http://www.numdam.org/item?id=ASENS_1991_4_24_3_339_0},
      review={\MR{1100994}},
}

\bib{Kol96}{book}{
      author={Koll\'{a}r, J\'{a}nos},
       title={Rational curves on algebraic varieties},
      series={Ergebnisse der Mathematik und ihrer Grenzgebiete. 3. Folge. A
  Series of Modern Surveys in Mathematics [Results in Mathematics and Related
  Areas. 3rd Series. A Series of Modern Surveys in Mathematics]},
   publisher={Springer-Verlag, Berlin},
        date={1996},
      volume={32},
        ISBN={3-540-60168-6},
         url={https://doi.org/10.1007/978-3-662-03276-3},
      review={\MR{1440180}},
}

\bib{Kov00}{article}{
      author={Kov\'{a}cs, S\'{a}ndor~J.},
       title={A characterization of rational singularities},
        date={2000},
        ISSN={0012-7094},
     journal={Duke Math. J.},
      volume={102},
      number={2},
       pages={187\ndash 191},
         url={https://doi.org/10.1215/S0012-7094-00-10221-9},
      review={\MR{1749436}},
}

\bib{Leh12}{article}{
      author={Lehmann, Brian},
       title={A cone theorem for nef curves},
        date={2012},
        ISSN={1056-3911},
     journal={J. Algebraic Geom.},
      volume={21},
      number={3},
       pages={473\ndash 493},
         url={https://doi.org/10.1090/S1056-3911-2011-00580-8},
      review={\MR{2914801}},
}

\bib{Moi66}{article}{
      author={Mo\u{\i}\v{s}ezon, B.~G.},
       title={On {$n$}-dimensional compact complex manifolds having {$n$}
  algebraically independent meromorphic functions. {I}},
        date={1966},
        ISSN={0373-2436},
     journal={Izv. Akad. Nauk SSSR Ser. Mat.},
      volume={30},
       pages={133\ndash 174},
      review={\MR{216522}},
}

\bib{Nak87}{incollection}{
      author={Nakayama, Noboru},
       title={The lower semicontinuity of the plurigenera of complex
  varieties},
        date={1987},
   booktitle={Algebraic geometry, {S}endai, 1985},
      series={Adv. Stud. Pure Math.},
      volume={10},
   publisher={North-Holland, Amsterdam},
       pages={551\ndash 590},
      review={\MR{946250}},
}

\bib{Nam02}{article}{
      author={Namikawa, Yoshinori},
       title={Projectivity criterion of {M}oishezon spaces and density of
  projective symplectic varieties},
        date={2002},
        ISSN={0129-167X},
     journal={Internat. J. Math.},
      volume={13},
      number={2},
       pages={125\ndash 135},
         url={https://doi.org/10.1142/S0129167X02001277},
      review={\MR{1891205}},
}

\bib{Pau98}{article}{
      author={Paun, Mihai},
       title={Sur l'effectivit\'{e} num\'{e}rique des images inverses de
  fibr\'{e}s en droites},
        date={1998},
        ISSN={0025-5831},
     journal={Math. Ann.},
      volume={310},
      number={3},
       pages={411\ndash 421},
         url={https://doi.org/10.1007/s002080050154},
      review={\MR{1612321}},
}

\bib{Pet01}{article}{
      author={Peternell, Thomas},
       title={Towards a {M}ori theory on compact {K}\"{a}hler threefolds.
  {III}},
        date={2001},
        ISSN={0037-9484},
     journal={Bull. Soc. Math. France},
      volume={129},
      number={3},
       pages={339\ndash 356},
         url={https://doi.org/10.24033/bsmf.2400},
      review={\MR{1881199}},
}

\bib{Pet98}{article}{
      author={Peternell, Thomas},
       title={Towards a {M}ori theory on compact {K}\"{a}hler threefolds.
  {II}},
        date={1998},
        ISSN={0025-5831},
     journal={Math. Ann.},
      volume={311},
      number={4},
       pages={729\ndash 764},
         url={https://doi.org/10.1007/s002080050207},
      review={\MR{1637984}},
}

\bib{Ram74}{article}{
      author={Ramis, Jean-Pierre},
      author={Ruget, Gabriel},
       title={R\'{e}sidus et dualit\'{e}},
        date={1974},
        ISSN={0020-9910},
     journal={Invent. Math.},
      volume={26},
       pages={89\ndash 131},
         url={https://doi.org/10.1007/BF01435691},
      review={\MR{352522}},
}

\bib{RRV71}{article}{
      author={Ramis, J.~P.},
      author={Ruget, G.},
      author={Verdier, J.~L.},
       title={Dualit\'{e} relative en g\'{e}om\'{e}trie analytique complexe},
        date={1971},
        ISSN={0020-9910},
     journal={Invent. Math.},
      volume={13},
       pages={261\ndash 283},
         url={https://doi.org/10.1007/BF01406078},
      review={\MR{308439}},
}

\bib{Siu74}{article}{
      author={Siu, Yum~Tong},
       title={Analyticity of sets associated to {L}elong numbers and the
  extension of closed positive currents},
        date={1974},
        ISSN={0020-9910},
     journal={Invent. Math.},
      volume={27},
       pages={53\ndash 156},
         url={https://doi.org/10.1007/BF01389965},
      review={\MR{352516}},
}

\bib{Tak85}{article}{
      author={Takegoshi, Kensh\={o}},
       title={Relative vanishing theorems in analytic spaces},
        date={1985},
        ISSN={0012-7094},
     journal={Duke Math. J.},
      volume={52},
      number={1},
       pages={273\ndash 279},
         url={https://doi.org/10.1215/S0012-7094-85-05215-9},
      review={\MR{791302}},
}

\bib{Toma16}{article}{
      author={Toma, Matei},
       title={Bounded sets of sheaves on {K}\"{a}hler manifolds},
        date={2016},
        ISSN={0075-4102},
     journal={J. Reine Angew. Math.},
      volume={710},
       pages={77\ndash 93},
         url={https://doi.org/10.1515/crelle-2013-0093},
      review={\MR{3437560}},
}

\bib{TZ18}{article}{
      author={Tosatti, Valentino},
      author={Zhang, Yuguang},
       title={Finite time collapsing of the {K}\"{a}hler-{R}icci flow on
  threefolds},
        date={2018},
        ISSN={0391-173X},
     journal={Ann. Sc. Norm. Super. Pisa Cl. Sci. (5)},
      volume={18},
      number={1},
       pages={105\ndash 118},
      review={\MR{3783785}},
}

\bib{Var86}{incollection}{
      author={Varouchas, J.},
       title={Sur l'image d'une vari\'{e}t\'{e} k\"{a}hl\'{e}rienne compacte},
        date={1986},
   booktitle={Fonctions de plusieurs variables complexes, {V} ({P}aris,
  1979--1985)},
      series={Lecture Notes in Math.},
      volume={1188},
   publisher={Springer, Berlin},
       pages={245\ndash 259},
         url={https://doi.org/10.1007/BFb0076826},
      review={\MR{926290}},
}

\bib{Var89}{article}{
      author={Varouchas, Jean},
       title={K\"{a}hler spaces and proper open morphisms},
        date={1989},
        ISSN={0025-5831},
     journal={Math. Ann.},
      volume={283},
      number={1},
       pages={13\ndash 52},
         url={https://doi.org/10.1007/BF01457500},
      review={\MR{973802}},
}

\bib{Wang21}{article}{
      author={Wang, Juanyong},
       title={On the {I}itaka conjecture {$C_{n,m}$} for {K}\"{a}hler fibre
  spaces},
        date={2021},
        ISSN={0240-2963},
     journal={Ann. Fac. Sci. Toulouse Math. (6)},
      volume={30},
      number={4},
       pages={813\ndash 897},
         url={https://doi.org/10.5802/afst.1690},
      review={\MR{4350100}},
}

\bib{Zha06}{article}{
      author={Zhang, Qi},
       title={Rational connectedness of log {${\bf Q}$}-{F}ano varieties},
        date={2006},
        ISSN={0075-4102},
     journal={J. Reine Angew. Math.},
      volume={590},
       pages={131\ndash 142},
         url={https://doi.org/10.1515/CRELLE.2006.006},
      review={\MR{2208131}},
}

\end{biblist}
\end{bibdiv}

\end{document}